\documentclass[12pt,leqno,british]{amsart}%{article}
\usepackage{ifpdf}

\usepackage[T1]{fontenc}
\usepackage{eucal}
\usepackage{url}
\usepackage[pagebackref]{hyperref}
\usepackage{amsfonts}
\usepackage{amsthm}
\usepackage{amsmath}
\usepackage{amscd}
\usepackage{amssymb}
\usepackage{rotating}

\def\CC {{\mathbb C}}     %% complex numbers
\def\HH {{\mathbb H}}     %% quaternions
\def\NN {{\mathbb N}}     %% natural numbers
\def\RR {{\mathbb R}}     %% real numbers
\def\ZZ {{\mathbb Z}}     %% integers

\def\ring#1{\ifmmode \mathaccent'027 #1\else \rm\accent'027 #1\fi}

\newcommand{\rI}{{\mathrm I}}
\newcommand{\rd}{{\mathrm d}}

\newcommand{\ri}{{\mathrm i}}

\def\ol  {\overline}
\def\ul  {\underline}

\def\wt  {\widetilde}

\def\ker {\mathfrak{ker}}

\def\mc {\mathcal}
\def\mk {\mathfrak}

\def\be  {\begin{eqnarray}}
\def\ee  {\end{eqnarray}}
\def\ben {\begin{eqnarray*}}
\def\een {\end{eqnarray*}}

\def\bpr {\begin{proof}[Proof]}
\def\epr {\end{proof}}
\def\bsp {\begin{split}}
\def\esp {\end{split}}
\def\bprr {\begin{proof}[solution]}
\def\bpru {\begin{proof}[hint]}
\def\bpro {\begin{proof}[answer]}
\def\bcd {\begin{CD}}
\def\ecd {\end{CD}}

\newcommand{\abs}[1]{\left\vert#1\right\vert}
\newcommand{\scal}[1]{\left\langle#1\right\rangle}

\newcommand{\sco}[1]{\left(#1\right)}

\newcommand{\ksco}[1]{\left[#1\right]}

\newtheorem{theorem}{Theorem}[section]
\newtheorem{lemma}[theorem]{Lemma}
\newtheorem{prop}[theorem]{Proposition}
\newtheorem{coro}[theorem]{Corollary}
\newtheorem{remark}[theorem]{Remark}
\newtheorem{df}[theorem]{Definition}

\begin{document}

\title[Homogeneous Hypercomplex Structures II]%
{Homogeneous Hypercomplex Structures II -  Coset Spaces of compact Lie Groups}

\author{George Dimitrov}
\address[Dimitrov]{University of Vienna\\
Faculty of Mathematics\\ Vienna, Austria,}
\email{gkid@abv.bg}

\author{Vasil Tsanov}
\address[Tsanov]{University of Sofia "St. Kl. Ohridski"\\
Faculty of Mathematics and Informatics,\\ Sofia, Bulgaria,}
\email{tsanov@fmi.uni-sofia.bg}

\begin{abstract}
We obtain a complete classification of  hypercomplex manifolds, on which a compact group of automorphisms acts transitively.
 The description of the spaces as well as the proofs of our results use only the structure theory of reductive groups,
 in particular the notion of "stem" of a reduced root system, introduced in the first paper of this series.
\end{abstract}

\maketitle
\setcounter{tocdepth}{2}
\tableofcontents

%\normalsize
\section{Introduction}
This paper is the second part in a series. In the first part \cite{DimTsan10} we introduced a remarkable subset\footnote{determined by a choice of a basis,
that is up to the Weyl group action}  of a reduced root system, called the {\bf stem} $\Gamma = \Gamma(\Delta^+)$. Using this we obtained a complete
 description of all left-invariant  hypercomplex structures on a compact Lie group ${\bf U}$.

Here we describe all left-invariant hypercomplex structures (for brevity we often write LIHCS further) on a coset space  $M = {\bf U}/{\bf K}_u$,
 where ${\bf U}$ is a compact, connected Lie group and ${\bf K}_u$ is a closed subgroup. In particular we describe all simply-connected, compact, homogeneous,
 hypercomplex spaces. We have chosen to assume that ${\bf U}$ acts effectively on $M$ (we always may), with this restriction the
  underlying manifold of  ${\bf U}$, does not in general admit a LIHCS, but factoring  by an appropriate subgroup ${{\bf K}_u}$, we may get a coset space $M$,
   which carries a LIHCS.

Nevertheless, it is natural to try to obtain  a homogeneous hypercomplex coset space  starting with a group ${\bf U}$, which admits LIHCS and
factoring by a subgroup ${\bf K}_u$, which is invariant under the given LIHCS. A careful reading of \cite{DimTsan10} and the present arguments should
convince the reader that the idea works.

Following this idea in section \ref{hcp0} we  define the notion of {\it   hypercomplex pair}: a pair $(\mk{u},\mk{k}_{\mk{u}})$ of compact Lie algebras $\mk{u} \supset \mk{k}_u$ satisfying certain conditions, essentially expressed in terms of the stems of the (reductive) complexifications $\mk{g}, \mk{k}$ of $\mk{u},\mk{k}_u$ respectively.

Further we construct a LIHCS on any coset space ${\bf U}/{\bf K}_u$ provided that the pair of Lie algebras $(\mk{u},\mk{k}_{\mk{u}})$ corresponding to ${\bf U},{\bf K}_u$ is a hypercomplex pair and ${\bf K}_u$ is connected.

In Section \ref{nc0} we prove that if ${\bf U}/{\bf K}_u$ admits a LIHCS  then $(\mk{u},\mk{k}_{\mk{u}})$ is a hypercomplex pair and the LIHCS has the form described in Section \ref{hcp0}. For the sake of completeness, in the first part of this section we describe in detail and prove  (some maybe known) structural results on the nonkaehler complex homogeneous spaces.

Actually in sections \ref{hcp0}, \ref{nc0} we solve  the invariant hypercomplex equations
on a coset space  ${\bf U}/{\bf K}_u$:
\begin{gather*} I, J \in C^{\infty}_{inv}( {\bf U}/{\bf K}_u, End(T({\bf U}/{\bf K}_u)) \nonumber \\
 I\circ J = - J\circ I, \qquad I^2=J^2 = -1, \qquad N_J=N_I=0,
\end{gather*}
where $N_I, N_J $ are the Nijenhuis tensors of $I,J$ respectively. So, we find all LIHCS
on  ${\bf U}/{\bf K}_u$  and describe them on the language of the stem.

In section \ref{shcp0} we show that the hypercompex pairs $(\mk{u},\mk{k}_u)$ with semi-simple $\mk{u}$ are composed of $su$-algebras. Then we show  that any  compact, simply connected homogeneous hypercomplex space is associated to a hypercompex pair $(\mk{u},\mk{k}_u)$ with semi-simple $\mk{u}$. Finally, using our description of semisimple hypercomplex pairs we give complete list of the HC spaces.

Thus we give, in terms of structure theory only, a complete classification of the homogeneous hypercomplex spaces with a transitive action of a compact group.

Less explicit results under additional constraints on the input
data in the context of differential geometry have appeared in
\cite{BeGoPo211}.

There is still a large class of compact homogeneous hypercomplex spaces unclassified, essentially those obtained from a hypercomplex structure on a nonabelian noncompact Lie group, by factorization with an uniform lattice.

\section{The stem revisited} \label{notations 0}

In this section we fix some notations used throughout the paper and recall the notion of {\bf stem}, introduced in \cite{DimTsan10}, and some of its properties.
\subsection{Notations}
 By ${\bf U}$ and $\mk{u}$ we shall denote a compact Lie group and its Lie algebra.
We have a  decomposition of $\mk{u}$ as follows
\be \label{mk{u} = mk{z}(mk{u}) oplus mk{u}_s} \mk{u} = \mk{c}_u \oplus \mk{u}_s,\ee
where $\mk{c}_u$ is the center of $\mk{u}$ and $\mk{u}_s$ is a semisimple subalgebra.

 We shall denote by $\mk{g}$ the complexification of $\mk{u}$ and by $\tau$ - the respective conjugation. The complexification of \eqref{mk{u} = mk{z}(mk{u}) oplus mk{u}_s} will be denoted
$$
 \mk{g}=\mk{c} \oplus \mk{g}_s.
 $$
By  $\mk{h}, \Delta, \Delta^+, \Pi$ we shall denote respectively{ a  $\tau$-invariant { \bf Cartan  subalgebra } of $\mk{g}$, the { \bf root system } of $\mk{g}$ w. r. to $\mk{h}$, {\bf  set of positive roots }, and  a { \bf  basis } $\Pi$ of $\Delta$.

We use the following notations related to the the Cartan subalgebra $\mk{h}$:
$$
  \mk{h}= \mk{c} \oplus \mk{h}_s, \quad \mbox{where} \ \ \mk{h}_s = \mk{h}\cap  \mk{g}_s, \quad \mk{h}_u = \mk{h}\cap \mk{u}.
$$
For $\alpha \in \Delta$, we denote $\mk{g}(\alpha) = \{ X \in \mk{g} \vert [H,X] = \alpha(H)X,\quad H \in \mk{h}\}$.

Let $h_\alpha \in \mk{h}_s$ be determined by
$Kill(H,h_\alpha) = \alpha(H)$  for $ H \in \mk{h}_s$, we denote
$H_\alpha = (2/\alpha(h_\alpha))h_\alpha$.

We also choose $E_\alpha \in \mk{g}(\alpha) $ so that
\begin{gather}\label{wcb}
\tau(E_{\alpha})= \tau_s(E_{\alpha})=- E_{-\alpha},\quad
[E_{\alpha}, E_{\beta}] = N_{\alpha, \beta}E_{\alpha +
\beta},\quad [E_{\alpha}, E_{- \alpha}] = H_\alpha.
\end{gather}
and the structural constants $N_{\alpha, \beta}$ are integer.

 Thus $\{E_\alpha,\vert \alpha \in \Delta\}\cup\{H_\beta\vert \beta \in \Pi\}$ is a  Weyl - Chevalley basis of $\mk{g}_s$.
 Having fixed a $\tau$-invariant Cartan subalgebra $\mk{h}$ and a $\Delta^+$, we have the standard nilpotent and Borelian subalgebras, which will be denoted respectively:
\begin{gather*}
\mk{n}^+ = \bigoplus_{\alpha\in \Delta^+}\mk{g}(\alpha),\quad \mk{n}^- = \tau(\mk{n}^+); \qquad \mk{b}^+ = \mk{n}^+ \oplus \mk{h},\quad \mk{b}^- = \tau(\mk{b}^+).
\end{gather*}

For $X,Y \in \mk{g}$, we denote by $\scal{X,Y}$ an ad-invariant symmetric bilinear form such that its restriction to the compact real form $\mk{u}$
is negative definite. We assume that $\scal{.,.}$ coincides with   the Killing form on the semisimple part $\mk{g}_s$.

\subsection{The stem}

In \cite{DimTsan10}, for any  $\zeta \in \Delta^+$ we defined
\be \label{prop of Delta+ def of B_{gamma}} \Phi_{\zeta}^+ &=& \{ \beta \in \Delta^+ \vert \  \zeta- \beta \in \Delta^+ \}, \qquad \Phi_{\zeta}=\Phi_{\zeta}^+ \cup \left (- \Phi_{\zeta}^+ \right ).  \ee

We proved in \cite{DimTsan10}, that there exists exactly one
subset $\Gamma \subset \Delta^+$, such that

\be
 & & \label{pdu1}  \Delta^+ = \Gamma \cup \bigcup_{\gamma \in \Gamma} \Phi_{\gamma}^+  \qquad \mbox{disjoint union}
  \ee
and we called this $\Gamma$ the {\bf stem} of $\Delta^+$. By
construction, the stem carries a certain uniquely determined
partial order $\prec$ which is extremely important for the results
(and proofs) in this paper, so we underline this fact by adopting
the notation $(\Gamma, \prec)$ for the stem.\footnote{ If $w$ is
an automorphism of the root system $\Delta$, then $w(\Gamma)$ is
the stem of $w(\Delta^+)$. We can recover $\Delta^+$ from the data
$(\Gamma, \prec)$ but different choices of Weyl chambers may lead
to the same set of roots $\Gamma$.}

The following  properties of the stem ($\Gamma $, $ \prec $) will be used throughout the paper:

\begin{prop}\label{prop of Delta+} Let $\gamma, \delta \in \Gamma$ and $\alpha \in \Phi_{\gamma}^+ \cup \{\gamma\}$, $\beta \in \Phi_{\delta}^+ \cup \{\delta\}$ then
\begin{gather} \label{pdep 00}    \delta \succ \gamma \ \mbox{and} \  \alpha \pm \beta \in \Delta  \Rightarrow \alpha \pm \beta \in \Phi_{\gamma}^+  \\
 \label{pdep 0}    \delta = \gamma \ \mbox{and} \  \alpha+ \beta \in \Delta  \Rightarrow \alpha + \beta = \gamma  \\
\label{pdep 1} \delta \prec \gamma \Rightarrow \alpha \pm \delta \not \in \Delta, \ \ \mbox{in particular} \ \ \alpha(H_{\delta})=0; \\ \label{pdep 6}  \mbox{if not} \ ( \gamma
\preceq   \delta  \ \mbox{or} \ \gamma \succeq  \delta ) \ \
\mbox{then} \ \
\alpha \pm \beta \not \in \Delta, \ \ \mbox{in particular} \ \ \alpha(H_{\delta})=0; \\
\label{pdep 2}  \alpha - \gamma \in \Delta^-,  \alpha + \gamma \not \in \Delta,   \ \ \mbox{in particular} \ \ \alpha(H_{\gamma})>0;   \\
\label{pdep 3}   \mbox{there exists a maximal root} \ \widetilde{\gamma} \in \Gamma \ \ \mbox{s. t.} \ \  \gamma \succeq \widetilde{\gamma}; \\
\label{pdep 7} \delta \prec \gamma \Rightarrow
 \ \   \exists \beta_1, \beta_2  \in  \Phi_{\delta}^+  \ \  \  \alpha = \beta_1 - \beta_2.
\end{gather}
\end{prop}
\bpr First we recall (see \cite{Helgason78}, p. 456/457) that for any $\alpha, \beta \in \Delta$ we have $2 \frac{\scal{\alpha,\beta}}{\scal{\beta,\beta}}= \alpha(H_{\beta})=-p-q$, where   $\{\alpha +n\beta,\ p\leq n \leq q\}$ is the $\beta$-series of $\alpha$.

Statements \eqref{pdep 00}, \eqref{pdep 1}, \eqref{pdep 2}  follow from formulas (23), (21), (22) in  Corollary 2.17 of \cite{DimTsan10}, respectively. Now \eqref{pdep 0} follows from a) in Corrolary 2.11 of \cite{DimTsan10}.

We go to the proof of \eqref{pdep 3}. Let $\widetilde{\Delta}$ be the irreducible component of  $\Delta $, s. t. $\gamma \in \widetilde{\Delta}$. Then  $\widetilde{\Delta}^+=\widetilde{\Delta} \cap \Delta^+ $ is a system of positive roots in $\widetilde{\Delta}$. Let's take the highest root of $(\widetilde{\Delta}, \widetilde{\Delta}^+)$ and denote it by $\widetilde{\gamma}$. Then $\widetilde{\gamma}$ is a maximal root of $(\Delta, \Delta^+)$. From the construction of the stem in Proposition 2.9, \cite{DimTsan10} we know that $\widetilde{\gamma} \in \Gamma$. From Proposition 2.15, \cite{DimTsan10}(and its proof) we know that $\Theta_{\widetilde{\gamma}} = \widetilde{\Delta}$. Thus we see that $\gamma \in \Theta_{\widetilde{\gamma}}$, which by definition (Definition 2.16 in \cite{DimTsan10}) means $\gamma \succeq \widetilde{\gamma}$.

   We go to the proof of \eqref{pdep 6}. Let $\Pi$ be the basis of $\Delta$ with positive roots $\Delta^+$.  From Proposition 2.9, \cite{DimTsan10} we  get a sequence $\Delta = \Delta_1 \supset \Delta_2 \supset \dots \supset \Delta_d$ of closed root subsystems with bases $\Pi_k = \Pi\cap\Delta_k$,
 corresponding sets of positive roots $\Delta_k^+  = \Delta^+ \cap\Delta_k \supset \Phi_{\gamma_k}^+$,   and maximal roots $\gamma_k \mbox{ of } \Delta_k^+,\quad k = 1,\dots,d$. Furthermore the stem is $ \Gamma=\{ \gamma_1, \gamma _2, \dots, \gamma_d \}$. Let $\gamma = \gamma_i$, $\delta= \gamma_j$. We can assume that $i < j$.   Now $\gamma, \delta \in \Delta_i^+$, $ \Phi_{\gamma}^+,  \Phi_{\delta}^+ \subset \Delta_i^+$. Let $\Delta'$, $\Delta''$ be the irreducible components of $\Delta_i$, s. t.  $\gamma \in \Delta'$, $\delta \in \Delta''$. Since $\gamma$ is a maximal root of $\Delta_i$, from Proposition 2.15, \cite{DimTsan10} we see  that $\Theta_{\gamma} = \Delta'$ and since $\gamma$ and $\delta$ are incomparable then $\delta \not \in \Delta'$. Hence $\Delta'$ and $\Delta''$ are different irreducible components of $\Delta_i$.  Since $\Delta'$ and $\Delta''$ are  irreducible components of $\Delta_i$, $\gamma \in \Delta'$, $\delta \in \Delta''$  and  $\Phi_{\gamma}^+,  \Phi_{\delta}^+ \subset \Delta_i^+$ one can easily observe that
 \ben
 \Phi_{\gamma}^+= \{\zeta \in \Delta_i^+ \vert \gamma - \zeta \in \Delta_i^+\} \subset \Delta'  \qquad \quad  \Phi_{\delta}^+=\{\zeta \in \Delta_i^+ \vert \delta - \zeta \in \Delta_i^+\} \subset \Delta''.
 \een
 Hence $\alpha \in \Delta'$ and $\beta \in \Delta''$. Since $\Delta'$ and $\Delta''$ are different irreducible components of $\Delta_i$ then $\alpha \pm \beta \not \in \Delta_i$. Since $\Delta_i$ is a closed subsystem of $\Delta$, then $\alpha \pm \beta \not \in \Delta$.

 Obviously,  for the proof of \eqref{pdep 7} it is sufficient to consider irreducible $\Delta$ with $\delta $ - the highest root of $\Delta^+$. Case by case through Section 3.3 of \cite{DimTsan10} one easily verifies \eqref{pdep 7}.

 \epr

\begin{coro} \label{gss1} $\Gamma$ may be represented as a sequence $\{\gamma_1,\dots, \gamma_d\}$, in which the first members are the  maximal roots (i. e. there is a number $1 \leq d' \leq d$, such that $\gamma_i$ is a maximal root iff $1 \leq i \leq d'$) and for any $\gamma_i, \gamma_j \in \Gamma$ we have
$\gamma_i \prec \gamma_j \Rightarrow i<j.$
\end{coro}

For any reductive Lie algebra $\mk{g}$ we define a number
$srank(\mk{g})$. If $\mk{g}$ is not abelian we take any Cartan
subalgebra $\mk{h}$ with root system $\Delta$ and for any choice
of a Weyl chamber we get a set of positive roots $\Delta^+$, hence
we get the stem $(\Gamma, \prec ) $ of $\Delta^+$. From the
uniqueness of the stem ( \cite{DimTsan10}, Theorem 2.12),  we see
that the number $ \# (\Gamma) $ is a well defined characteristic
of $\mk{g}$.
\begin{df}\label{srank}
If $\mk{g}$ is abelian we set $srank(\mk{g})=0$. For a nonabelian  reductive Lie algebra $\mk{g}$ we denote:
$ srank(\mk{g}) = 2 \# (\Gamma). $
\end{df}

\subsection{Decompositions}

At the end of this section we recall some useful decompositions related to the stem. Next we define important subspaces of $\mk{h}$ and $\mk{h}_u$.

\begin{df} We denote
$$
\mk{w} = span_\CC\{H_\gamma \vert \gamma \in\Gamma\},\quad \mk{o} = \bigcap_{\gamma \in \Gamma} \ker(\gamma); \qquad \mk{w}_u=\mk{w}\cap \mk{u},\quad \mk{o}_u=\mk{o}\cap \mk{u}.
$$
\end{df}
Obviously we have orthogonal decompositions:
\begin{gather}
\mk{h} = \mk{w} \oplus \mk{o},\quad \mk{h}_u = \mk{w}_u \oplus \mk{o}_u.
\end{gather}

\begin{df}\label{X_gamma etc}
 Let $\rho=\{\rho_\gamma \vert \gamma \in \Gamma \}$ be  a family of complex numbers on the unit circle. For each   $\gamma \in \Delta^+$ we denote:
 \begin{gather*}
W_\gamma = \frac{\ri}{2}H_\gamma,\quad X_\gamma(\rho) = \frac{1}{2}(\rho_\gamma E_\gamma - \ol{\rho_\gamma} E_{-\gamma}),\quad Y_\gamma(\rho) = \frac{\ri}{2}(\rho_\gamma E_{\gamma} + \ol{\rho_\gamma} E_{-\gamma}),\\
sl_\gamma(2) = span_\CC \{E_\gamma, E_{-\gamma}, H_\gamma\}, \quad
  su_\gamma(2) = span_\RR\{X_\gamma, Y_\gamma,W_\gamma \}.
\end{gather*}
 \end{df}

\begin{df}
As in \cite{DimTsan10}, for each $\gamma \in \Gamma$, we shall denote:
 \begin{gather*}  \mc{V}_{\gamma}^+ = \sum_{\alpha \in \Phi_{\gamma}^+} \mk{g}(\alpha), \quad \mc{V}_{\gamma}^- = \sum_{\alpha \in \Phi_{\gamma}^-}  \mk{g}(\alpha),\quad \mc{V}_{\gamma} = \mc{V}_{\gamma}^+ \oplus \mc{V}_{\gamma}^-,\quad
\mc{V}_{\gamma}^u = \mk{u}\cap \mc{V}_{\gamma}.
\end{gather*}
\end{df}
We have again orthogonal decompositions:
\begin{gather}\label{dkh1}
\mk{g} = \mk{o} \oplus \bigoplus_{\gamma \in \Gamma } ( sl_\gamma (2) \oplus {\mc V}_\gamma),\quad \mk{u} = \mk{o}_u \oplus \bigoplus_{\gamma \in \Gamma } ( su_\gamma (2) \oplus {\mc V}_\gamma^u).
\end{gather}

\section{Hypercomplex pairs}\label{hcp0}

\subsection{Background and definition}

 Let $I$ be a left-invariant complex structure on ${\bf U}$, and let us denote also by $I$ its restriction to the tangent space at the unit element, which we identify with $\mk{u}$. We still denote by $I$ the $\CC$- linear extention of $I$ to $\mk{g}$. Then it is well known (see e.g. \cite{Snow86}), that we may find a $\tau$-invariant Cartan subalgebra $\mk{h}$, a basis $\Pi$ of the root system $\Delta$ w.r. to $\mk{h}$, and a decomposition
 $$
 \mk{h} = \mk{h}^+ \oplus \mk{h}^-,\qquad \tau(\mk{h}^+) = \mk{h}^-,
 $$
  so that the regular subalgebras:
 $$
\mk{g}^+ = \mk{h}^+ \oplus \mk{n}^+,\quad \mk{g}^- = \mk{h}^-\oplus \mk{n}^-
$$
are respectively the $\ri$ , $-\ri$ eigenspaces of $I$.

In \cite{DimTsan10} we have shown that a left-invariant complex structure $I$  on  ${\bf U}$, presented in the above manner, admits a matching complex structure $J$ (i.e. $IJ = - JI$) if and only if
\begin{gather}\label{nsc}
\mk{z} = I\mk{w} \subset \mk{o}, \qquad \dim(\bf U) \in 4 \ZZ .
\end{gather}

We introduce  notations for the entities related to the complex structure $I$:
\begin{df}  \label{Z_gamma,P_gamma} We denote
\begin{gather*}
Z_\gamma = IW_\gamma,\quad P_\gamma = W_\gamma - \ri Z_\gamma,\quad Q_\gamma = W_\gamma + \ri Z_\gamma.\\
  u_\gamma(2) = su_\gamma(2)\oplus \RR Z_\gamma,\quad  gl_\gamma(2) = sl_\gamma(2)\oplus \CC Z_\gamma.
\end{gather*}
\end{df}

In \cite{DimTsan10} we have shown that any invariant hypercomplex structure $I,J$ on  ${\bf U}$ splits the Lie algebra $\mk{g}$ in the following form
\begin{gather}\label{dkh}
\mk{g} = \mk{j} \oplus \bigoplus_{\gamma \in \Gamma } ( gl_\gamma (2) \oplus {\mc V}_\gamma),\quad \mk{u} = \mk{j}_u \oplus \bigoplus_{\gamma \in \Gamma } ( u_\gamma (2) \oplus {\mc V}_\gamma^u),
\end{gather}
where $\mk{o}_u = \mk{j}_u \oplus I(\mk{w}_u) ,\quad  dim_\RR(\mk{j}_u ) \in  4\ZZ $. In terms of this decomposition the hypercomplex structure $(I,J)$ is  defined    for $A\in \mc{V}_\gamma^u\ $ by:
\begin{gather} \label{csj1}
 I A  = 2 [W_\gamma,A],\quad  J A =- 2[Y_\gamma,A],
 \end{gather}
 for each $\gamma \in \Gamma$  it is determined   on $u_\gamma(2)$ by:
\begin{gather}\label{eee3}
IX_\gamma =  Y_\gamma,\quad IW_\gamma = Z_\gamma;\qquad JX_\gamma =  W_\gamma ,\quad JZ_\gamma = Y_\gamma,
\end{gather}
and the hypercomplex structure on $\mk{j}$ is defined by identifying it with $\HH^k$. It is obvious that the decomposition subspaces $\mk{j}_u$, $\mk{u}_\gamma(2)$, ${\mc V}_\gamma^u$ are $(I,J)$-invariant.

We shall obtain a left-invariant hypercomplex structure on  a coset space ${\bf U}/{\bf K}_u$, when $\mk{k}$ is a partial sum (of vector subspaces)
$$
\mk{k} = \mk{j}_\mk{k} \oplus \bigoplus_{\gamma \in \Gamma_\mk{k} } ( gl_\gamma (2) \oplus {\mc V}_\gamma),
$$
where the subset $\Gamma_\mk{k} \subset \Gamma$ and the subspace $\mk{j}_\mk{k} \subset \mk{j}$ are properly chosen, so that the above $\mk{k}$ is a subalgebra\footnote{such that the subgroup generated by $\mk{k}\cap \mk{u}$ is closed}.

Taking a complement  $\mk{j}_\mk{p}\subset \mk{j}$ so that $\mk{p} = \mk{j}_\mk{p} \oplus \mk{j}_\mk{k}$ and denoting $\Gamma_\mk{p} =\Gamma \setminus \Gamma_\mk{k}$ we obtain a complement (see the decomposition \eqref{dkh}) of $\mk{k}$ in $\mk{g}$:
$$
\mk{p} = \mk{j}_\mk{p} \oplus \sum_{\gamma \in \Gamma \setminus \Gamma_{\mk{k}}} (\mc{V}_{\gamma} \oplus gl_{\gamma}(2)).
$$
Then we use the same formulas as in \eqref{csj1}, \eqref{eee3} to define a hypercomplex structure on $\mk{p}$.
In this section we carry out this program in detail. Further we prove that all LIHCS on  our kind of coset spaces are obtained in this way\footnote{It is easy to see from the examples in \cite{DimTsan10} and Section \ref{shcp0} of this paper, that even if we discard the "toral" part $\mk{j}$, the action of ${\bf U}$ will not be effective in general}. We start by studying subalgebras $\mk{k}$ as above.   We start by studying subalgebras $\mk{k}$ as above.

    Obviously, for $\mk{k}$ we have $\mk{o}\cap \mk{k}=\mk{j}_{\mk{k}}+ span_\CC\{I(W_\gamma)\}_{\gamma \in \Gamma_{\mk{k}}}$ and we can represent  $ \mk{k}$ as follows
    \begin{gather} \label{pdp1}
     \mk{k} = (\mk{o}\cap \mk{k}) + \sum_{\gamma \in \Gamma_{\mk{k}} } ( sl_\gamma (2) + {\mc V}_\gamma).
     \end{gather}
From property \eqref{pdep 7} of the stem $\Gamma$, it follows easily that if $\gamma, \delta  \in \Gamma$, then
\be \label{property of Gamma_mk{k}}
 ( \gamma \in \Gamma_{\mk{k}} \ \mbox{and} \ \delta \succ \gamma ) \Rightarrow \delta \in \Gamma_{\mk{k}}.
 \ee
Furthermore, we have
\ben \mk{h} \cap \mk{k} &=& \mk{j}_{\mk{k}}\oplus span_\CC\{W_\gamma\}_{\gamma \in \Gamma_\mk{k}}\oplus span_\CC\{I(W_\gamma)\}_{\gamma \in \Gamma_{\mk{k}}};; \\
 \mk{o}  &=& \mk{j}_{\mk{k}} \oplus \mk{j}_{\mk{p}} \oplus span_\CC \{I(W_\gamma)\}_{\gamma \in \Gamma}\een
 therefore
\begin{gather}\label{pdu6}
 rank(\mk{g}) + srank(\mk{k}) - rank(\mk{k}) - srank(\mk{g}) = dim(\mk{j}_{\mk{p}})\geq 0.
 \end{gather}
It is convenient to introduce
\begin{df}
A subset  $\Gamma_{\mk{k}} \subset \Gamma$, which satisfies \eqref{property of Gamma_mk{k}}, will be called a \textbf{substem of} $(\Gamma, \prec)$.
\end{df}
We have

\begin{lemma} \label{pdp3}

Let $\Gamma_{\mk{k}} \subset \Gamma $  be a substem of $(\Gamma, \prec)$ and let $\Delta_{\mk{k}}^+ = \bigcup_{\gamma \in \Gamma_{\mk{k}}} \Phi_{\gamma}^+ \cup \{\gamma\},$ $\Delta_{\mk{k}} = \bigcup_{\gamma \in \Gamma_{\mk{k}}} \Phi_{\gamma} \cup \{\gamma,-\gamma \}$.
Then there exists a subset ${\mc M}\subset \Gamma$,  which consists of pairwise incomparable with respect to $\prec$ elements  and such that
\be \label{pdp2}
 \Delta_{\mk{k}} = \bigcup_{\gamma \in {\mc M}} \Theta_\gamma, \qquad  \Delta_{\mk{k}}^+ = \bigcup_{\gamma \in {\mc M}} \Theta_\gamma^+ \qquad \quad \mbox{disjoint unions}\ee
(for the definition of $\{ \Theta_\gamma \}_{\gamma \in \Gamma}$
see Proposition 2.15 in \cite{DimTsan10}).

 The stem of  $\Delta_{\mk{k}}^+ $ is $\Gamma_{\mk{k}}$.
\end{lemma}
\begin{proof}  We denote by ${\mc M}$  the set of all minimal elements in $\Gamma_{\mk{k}}$ (with respect to the ordering $\prec $). Let us take any element $\alpha \in \Delta_{\mk{k}}^+$ then by $\Delta_{\mk{k}}^+ = \bigcup_{\gamma \in \Gamma_{\mk{k}}} \Phi_{\gamma}^+ \cup \{\gamma\}$ it follows that $\alpha \in \Phi_{\gamma}^+ \cup \{\gamma\}$ for some $\gamma \in \Gamma_{\mk{k}}$. Since $\Gamma_{\mk{k}}$ is finite then there exists a minimal element $\delta\in {\mc M}$ such that $\delta \preceq  \gamma$ and therefore  $\gamma \in \Theta_{\delta}^+$ and $\Phi_{\gamma}^+ \subset  \Theta_{\delta}^+$. So we see that
$  \Delta_{\mk{k}}^+ \subset \bigcup_{\gamma \in {\mc M}} \Theta_\gamma^+. $

Let us take any $\gamma \in {\mc M}$. We recall  that the stem of $\Theta_\gamma^+$ is  $\Gamma \cap \Theta_\gamma$ and
\be \label{pdp33} \Theta_\gamma^+= \bigcup_{\delta \in \Gamma \cap \Theta_\gamma} \Phi_{\delta}^+ \cup \{\delta\}.\ee
  Since all the elements in  $\Gamma \cap \Theta_\gamma$ are greater than $\gamma \in \Gamma_{\mk{k}}$ and by \eqref{property of Gamma_mk{k}}  it follows that $\Gamma \cap \Theta_\gamma \subset \Gamma_{\mk{k}}$. Therefore (see  \eqref{pdp33}) it follows that $\Delta_{\mk{k}}^+ \supset \Theta_\gamma^+$, hence   $  \Delta_{\mk{k}}^+ \supset \bigcup_{\gamma \in {\mc M}} \Theta_\gamma^+.$

 Thus far, we proved   \eqref{pdp2}. Obviously that for any pair of different elements $\gamma, \delta \in {\mc M}$ neither $\gamma \preceq  \delta$ nor $\gamma \succeq \delta$. Hence in \eqref{pdp2} we have disjoint unions.

For $\gamma \in {\mc M}$ the stem of $(\Theta_\gamma,\Theta_\gamma^+)$ is $\Gamma \cap \Theta_\gamma$, then the stem of $(\Delta_{\mk{k}}, \Delta_{\mk{k}}^+) $ is $\cup_{\gamma \in {\mc M}} \Gamma \cap \Theta_\gamma = \Gamma \cap \Delta_{\mk{k}}= \Gamma_{\mk{k}}$ and the corollary is completely proved.
\end{proof}
For a subalgebra $\mk{k}$, which satisfies \eqref{pdp1}, we give a definition
\begin{df} Let $(\Gamma, \prec)$ be a stem of a reductive algebra $\mk{g}$.  Let  $\Gamma_{\mk{k}} \subset \Gamma$ be a   substem of $(\Gamma, \prec)$. A subalgebra $\mk{k} \subset \mk{g}$ of the type \eqref{pdp1} will be called a \textbf{$\Gamma_{\mk{k}}$-stemmed subalgebra of $\mk{g}$}.

%If $\mk{k}$ is a $\Gamma_{\mk{k}}$-stemmed subalgebra of $\mk{g}$ for some substem $\Gamma_{\mk{k}}$ of  a stem $(\Gamma,\prec)$ of $\mk{g}$ we %call $\mk{k}$ a \underline{stemmed subalgebra of $\mk{g}$}.
\end{df}

We come to the most important notion in this section
\begin{df} \label{shp} Let   $\mk{u}\supset \mk{k}_\mk{u}$ be a pair of compact Lie algebras. Let  $\mk{g}\supset \mk{k}$ be the complexifications of $(\mk{u},\mk{k}_u)$.

We  call   $(\mk{u},\mk{k}_u)$  a \textbf{hypercomplex pair} if the following conditions are satisfied:
\begin{itemize}
\item $dim(\mk{u}) - dim(\mk{k}_u)>0$ is divisible by four,
    \item  There exist a Cartan subalgebra $\mk{h}_u \subset \mk{u}$ with complexification $\mk{h}\subset \mk{g}$, set of positive roots $\Delta^+ \subset \Delta$  w. r. to $\mk{h}$, and  stem  $\Gamma \subset \Delta^+$,  such that   $\mk{k}$ is a $\Gamma_{\mk{k}}$-stemmed subalgebra of $\mk{g}$ for some  substem $\Gamma_{\mk{k}} \subset \Gamma $.
    \item $rank(\mk{g}) +  srank( \mk{k}) \geq rank( \mk{k}) + srank(\mk{g})$.
\end{itemize}

 Let $(\mk{u}, \mk{k}_u)$ be a hypercomplex pair. A pair of groups $({\bf U}, {\bf K}_u)$ is said to be {\bf associated to the hypercomplex pair} $(\mk{u}, \mk{k}_u)$, if ${\bf U}$ is a compact, connected Lie group, $\mk{u}$ is its Lie algebra, ${\bf K}_u \subset {\bf U}$ is a closed subgroup with Lie algebra $\mk{k}_u$.
\end{df}

Throughout the rest of this section  $(\mk{u}, \mk{k}_u)$ is a  hypercomplex pair and $(\mk{g}, \mk{k})$ is its complexification.

We  assume also  that  we are given a pair of groups $({\bf U},
{\bf K}_u)$,  associated with $(\mk{u}, \mk{k}_u)$, where  ${\bf
K}_u$ is a connected subgroup of ${\bf U}$ \footnote{So, by
assumption, for the Lie algebra $\mk{k}_u$   we have  that the
subgroup of ${\bf U}$ generated by $\mk{k}_u$ is a closed
subgroup, in addition to the properties listed in definition
\ref{shp}.}. In this section we obtain a hypercomplex structure on
any coset space $M={\bf U}/{\bf K}_u$ of this type. We shall
denote by $o \in M$ the coset ${\bf K}_u  \in {\bf U}/{\bf K}_u$.
Let $\mk{h}$ and $\Gamma$ be  as in Definition \ref{shp}, applied
to $(\mk{g},\mk{k})$, so we have \be
\mk{k}= \mk{o}_{\mk{k}}\oplus \bigoplus_{\gamma \in \Gamma_\mk{k} } ( sl_\gamma (2) \oplus {\mc V}_\gamma), \ \  \mbox{where} \quad \mk{o}_{\mk{k}} = \mk{k} \cap \mk{o} \nonumber \\[-2mm] \label{the mk{k}} \\[-2mm] rank(\mk{g}) + srank(\mk{k}) - rank(\mk{k})- srank(\mk{g}) \geq 0, \nonumber
\ee
where $\Gamma_{\mk{k}}$ is a substem of $\Gamma$.

\subsection{Construction of a hypercomplex structure on \texorpdfstring{$T_o(M)$}{\space} }
\label{chcs}
\subsubsection{The subspace $\mk{p}_u\subset \mk{u}$}

Let $\mk{p}_u$ be the orthogonal complement of $\mk{k}_u$ with respect to $\scal{,}$ and let $\mk{p}$ be the complexification of $\mk{p}_u$. We shall denote
\begin{gather}
  \mk{h}_{\mk{p}_u} =\mk{h} \cap \mk{\mk{p}}_u \qquad   \mk{h}_{\mk{p}} = \mk{h} \cap \mk{\mk{p}} \qquad  \mk{h}_{\mk{k}_u} =\mk{h} \cap \mk{\mk{k}}_u  \qquad \mk{h}_{\mk{k}} = \mk{h} \cap \mk{\mk{k}} \nonumber \\[-2mm]
 \label{mk{h}_{mk{p}} etc} \\[-2mm]
 \mk{o}_{\mk{p}_u} =  \mk{o}\cap \mk{p}_u \qquad \mk{o}_{\mk{p}}=\mk{o} \cap \mk{p} \qquad  \mk{o}_{\mk{k}_u} =  \mk{o}\cap \mk{k}_u. \nonumber
  \end{gather}

In order to describe $\mk{p}$ in terms of the stem, it is convenient to introduce  the following notatins
\begin{gather} \label{wcb2}
\Gamma_{\mk{p}} = \Gamma \setminus \Gamma_{\mk{k}} \qquad \mk{w}_{\mk{p}} = span_\RR\{W_\gamma\vert \gamma \in \Gamma_{\mk{p}}\}  \qquad \mk{w}_{\mk{k}} = span_\RR\{W_\gamma\vert \gamma \in \Gamma_{\mk{k}}\} \\
\label{wcb3} \Delta_{\mk{p}}^+ = \bigcup_{\gamma \in \Gamma_{\mk{p}}} \Phi_{\gamma}^+ \cup \{\gamma\} \quad \Delta_{\mk{p}}= \Delta_{\mk{p}}^+ \cup (-\Delta_{\mk{p}}^+) \quad \Delta_{\mk{k}}^+ = \bigcup_{\gamma \in \Gamma_{\mk{k}}} \Phi_{\gamma}^+ \cup \{\gamma\} \quad \Delta_{\mk{k}}= \Delta_{\mk{k}}^+ \cup (-\Delta_{\mk{k}}^+).
\end{gather}

From \eqref{the mk{k}} we see that $\mk{k}=\mk{h}_{\mk{k}} + \sum_{\alpha \in \Delta_{\mk{k}}} \mk{g}(\alpha) $ and $\mk{g}=\mk{k}\oplus \mk{p}$. Hence we get
\begin{lemma} \label{lemma for mk{p}} The subspace $\mk{p}$ may be decomposed as follows:
\begin{gather} \label{decomposition of mk{m}}
\mk{p} = \mk{h}_{\mk{p}} \oplus \bigoplus_{\alpha \in \Delta_{\mk{p}}} \mk{g}(\alpha).
\end{gather}
We have $\mk{h}=\mk{h}_{\mk{k}} \oplus \mk{h}_{\mk{p}}$.
\end{lemma}

From \eqref{property of Gamma_mk{k}} and $\Gamma_{\mk{k}}\cap \Gamma_{\mk{p}} = \emptyset$ it follows that if $\gamma \in \Gamma_{\mk{p}},\quad  \delta  \in \Gamma_{\mk{k}}$, then either $\gamma \prec \delta$ or $\gamma,\ \delta$ are not comparable, so using Proposition \ref{prop of Delta+},   we have
\begin{gather} \label{gmg2}
 \gamma \in \Gamma_{\mk{p}}, \ \beta \in \Delta_{\mk{k}}, \ \alpha \in \Phi_{\gamma}, \ \alpha + \beta \in \Delta \quad \Rightarrow \quad \alpha+\beta \in \Phi_{\gamma}, \gamma \pm \beta \not \in \Delta
 \end{gather}

\begin{lemma} \label{ggg1}  $\mk{w}_{\mk{p}} \subset  \mk{p}_{\mk{u} } $.
\end{lemma}
\bpr Let $\gamma \in \Gamma_{\mk{p}}$. We have to prove that for all $X \in \mk{k}$ we have $\scal{W_{\gamma}, X} = 0$. To that end we recall \eqref{the mk{k}} and the properties:

a) For each $\alpha \in \Delta$ the subspaces $\mk{g}(\alpha)$ and $\mk{h} $ are orthogonal  with respect to $\scal{,}$.

b)  $W_\gamma $ is orthogonal to $ \mk{o}$ with respect to $\scal{,}$.

c) For  each $\delta \in \Gamma_{\mk{k}}$ $\gamma \pm \delta \not \in \Delta$, in particular $\scal{W_\gamma,W_{\delta}} = 0$.
\epr

\begin{coro} \label{mk{o}=mk{o}_{mk{k}} oplus mk{o}_{mk{p}}} The following decompositions $\mk{h}_{\mk{p}}=\mk{o}_{\mk{p}} \oplus \mk{w}_{\mk{p}}^{\CC}$,  $\mk{o}=\mk{o}_{\mk{k}}  \oplus \mk{o}_{\mk{p}}$ are valid. In particular $dim(\mk{o}) = dim(\mk{o}_{\mk{k}}) + dim(\mk{o}_{\mk{p}})$. Furthermore, $dim(\mk{h})=dim(\mk{o}) + \#(\Gamma)$, $dim(\mk{h}_{\mk{k}} )=dim(\mk{o}_{\mk{k}}) + \#(\Gamma_{\mk{k}})$.
\end{coro}
\bpr
Since for $\alpha \in \Delta $ and $H \in \mk{h}$ we have $\alpha(\tau(H))=-\ol{\alpha(H)}$,then $\tau(\mk{o}) = \mk{o}$, hence
\ben
\mk{o} = \mk{o}_u \oplus \ri \mk{o}_u  \qquad \mk{o}_{\mk{k}} = \mk{o}_{\mk{k}_u} \oplus \ri \mk{o}_{\mk{k}_u} \qquad
\mk{o}_{\mk{p}} =  \mk{o}_{\mk{p}_u} \oplus \ri \mk{o}_{\mk{p}_u}.
\een
The orthogonal complement of $span_{\RR}\{W_{\gamma}\vert \gamma \in \Gamma \}$ in $\mk{h}_u$ is $\mk{o}_u$, hence $span_{\RR}\{W_{\gamma}\vert \gamma \in \Gamma \} \oplus \mk{o}_u = \mk{h}_u$ and therefore $dim(\mk{h}_u)= dim(\mk{o}_u) + \#(\Gamma)$ and $ dim_\CC(\mk{h})= dim_\CC(\mk{o}) + \#(\Gamma)$. In Lemma \ref{ggg1} we proved that $W_{\gamma} \in \mk{p}_u $ for $\gamma \in \Gamma_{\mk{p}}$, besides by \eqref{the mk{k}} we have  $W_{\gamma} \in \mk{k}_u $ for $\gamma \in \Gamma_{\mk{k}}$. Since $\mk{p}_u$ is orthogonal to $\mk{k}_u$, in particular $\mk{p}_u$ is orthogonal to $\mk{h}_{\mk{k}_u}$ as well(we recall that $\mk{h}_{\mk{k}_u} = \mk{k}_u
\cap \mk{h}_u $),   hence $\mk{o}_{\mk{k}_u}=\mk{o} \cap \mk{k}_u=\mk{o} \cap \mk{h}_{\mk{k}_u}$  is the orthogonal complement of $span\{W_{\gamma}\vert \gamma \in \Gamma_{\mk{k}}\}$ in $\mk{h}_{\mk{k}_u}$, which implies  $\mk{w}_{\mk{k}} \oplus \mk{o}_{\mk{k}_u} = \mk{h}_{\mk{k}_u}$, hence $ dim(\mk{h}_{\mk{k}})= dim(\mk{o}_{\mk{k}}) + \#(\Gamma_{\mk{k}})$. Besides $\mk{o}_{\mk{p}_u}=\mk{o} \cap \mk{p}_u=\mk{o} \cap \mk{h}_{\mk{p}_u}$ is the orthogonal complement of $span\{W_{\gamma}\vert \gamma \in \Gamma_{\mk{p}}\}$ in $\mk{h}_{\mk{p}_u}$.
Thus far, we showed  that
\begin{gather*} span_{\RR}\{W_{\gamma}\vert \gamma \in \Gamma \} \oplus \mk{o}_u = \mk{h}_u,\\  \mk{w}_{\mk{k}} \oplus \mk{o}_{\mk{k}_u} = \mk{h}_{\mk{k}_u}, \quad \mk{w}_{\mk{p}} \oplus \mk{o}_{ \mk{p}_u } = \mk{h}_{\mk{p}_u}. \end{gather*}
 These  equations together with $\mk{h}_u=\mk{h}_{\mk{k}_u} \oplus \mk{h}_{\mk{p}_u}$ imply $\mk{o}_u = \mk{o}_{\mk{k}_u} \oplus \mk{o}_{ \mk{p}_u}$. Therefore  $\mk{o}=\mk{o}_{\mk{k}}  \oplus \mk{o}_{\mk{p}}$.\epr

\begin{lemma} \label{edm1} The following integers are equal to each other:
\begin{itemize}
    \item[a)] $rank(\mk{g}) + srank(\mk{k}) - rank(\mk{k}) - srank(\mk{g})$
    \item[b)] $ dim(\mk{h}_{\mk{p}} ) - 2 \# (\Gamma_{\mk{p}})$
    \item[c)] $ dim(\mk{o}_{\mk{p}}) -  \# (\Gamma_{\mk{p}})$.
\end{itemize}

\end{lemma}
\bpr By Corollary \ref{mk{o}=mk{o}_{mk{k}} oplus mk{o}_{mk{p}}}  we have $ dim(\mk{o})= dim(\mk{o}_{\mk{k}})+ dim(\mk{o}_{\mk{p}})$, besides $ dim(\mk{h})= dim(\mk{o}) + \#(\Gamma)$,   $ dim(\mk{h}\cap\mk{k})= dim(\mk{o}_{\mk{k}}) + \#(\Gamma_{\mk{k}})$. In the  calculations that follow we take into account these equalities
\begin{gather*}  dim(\mk{o}_{\mk{p}})-\#(\Gamma_{\mk{p}})=  dim(\mk{o})-  dim(\mk{o}_{\mk{k}})-\#(\Gamma_{\mk{p}})=  dim(\mk{h})- \#(\Gamma) \\-( dim(\mk{h}\cap\mk{k})- \#(\Gamma_{\mk{k}}))  -(\#(\Gamma) -\#(\Gamma_{\mk{k}}))
= rank(\mk{g})- rank(\mk{k}) - 2 \#(\Gamma) + 2 \#(\Gamma_{\mk{k}}) \\
=rank(\mk{g})- rank(\mk{k}) - srank(\mk{g}) + srank(\mk{k})
\end{gather*}
and the equality between $a)$ and $c)$ is proved. The equality between  a) and b) follows from
\begin{gather*}
dim(\mk{h}\cap\mk{p})-2  \#( \Gamma_{\mk{p}})= rank(\mk{g})-rank(\mk{k}) -2 ( \#(\Gamma) - \#(\Gamma_{\mk{k}}) )\nonumber \\ =
 rank(\mk{g})-rank(\mk{k})+ srank(\mk{k})-srank(\mk{g}).
 \end{gather*}
 \epr

Since on one hand $\mk{h}_{\mk{p}}\cap \mk{u}$ and $\mk{h}_{\mk{k}}\cap\mk{u}$ are orthogonal to each other   and on the other hand $H_{\gamma} \in \mk{h}_{\mk{p}}$ for $\gamma \in \Gamma_{\mk{p}}$ (see Lemma \ref{ggg1}), $H_{\alpha} \in \mk{h}_{\mk{k}}$ for $\alpha \in \Delta_{\mk{k}}$ then one can easily verify that for any choice of $\gamma \in \Gamma_{\mk{p}},\ \alpha \in \Delta_{\mk{k}}$ we have
\be \label{gmg8}
\gamma_{\vert \mk{h}_{\mk{k}}} = 0; \qquad \qquad  \alpha_{\vert \mk{h}_{\mk{p}}} = 0.
\ee

We recall that   $ rank(\mk{g}) + srank(\mk{k}) - rank(\mk{k})- srank(\mk{g}) \geq 0 $, hence, by Lemma \ref{edm1}    we may write
\be \label{dim(mk{h}_{mk{m}})=2...}  dim(\mk{h}_{\mk{p}}) \geq 2 \#(\Gamma_{\mk{p}}) \qquad  dim(\mk{o}_{\mk{p}}) \geq \#(\Gamma_{\mk{p}}). \ee

Let us give some comments on the dimension of $\mk{o}_{\mk{p}}$. First of all  by \eqref{dim(mk{h}_{mk{m}})=2...}  we have  $ dim(\mk{o}_{\mk{p}}) \geq \#(\Gamma_{\mk{p}})$, therefore $ dim_\RR (\mk{o}_{\mk{p}_u}) \geq \#(\Gamma_{\mk{p}})$.
 By \eqref{wcb3} we see that  $\# (\Delta_{\mk{p}}^+) + \# (\Gamma_{\mk{p}}) $ is  even (for any $\gamma \in \Gamma$ the set $\Phi_{\gamma}^+$ contains of even number of elements).
From \eqref{decomposition of mk{m}} it follows that $ dim(\mk{o}_{\mk{p}}) - \#(\Gamma_{\mk{p}})=  dim(\mk{p}) - 2 \#(\Gamma_{\mk{p}}) - 2 \#(\Delta_{\mk{p}}^+)= dim(\mk{p}) - 2 ( \#(\Gamma_{\mk{p}}) + \#(\Delta_{\mk{p}}^+) )$. Since $ dim(\mk{p})$ is divisible by four as well as $ 2 ( \#(\Gamma_{\mk{p}}) + \#(\Delta_{\mk{p}}^+) )$, then  $ dim(\mk{o}_{\mk{p}}) - \#(\Gamma_{\mk{p}})$ is divisible by four. Therefore $ dim(\mk{o}_{\mk{p}_u}) - \#(\Gamma_{\mk{p}})$ is divisible by four and we can decompose $\mk{o}_{\mk{p}_u} $ as follows
\be \label{mk{o}_{mk{p}_0} = mk{o}_{mk{p}_0}^{a}} \mk{o}_{\mk{p}_u} = \mk{j}_{\mk{p}_u} \oplus  \mk{z}_{\mk{p}} \qquad \qquad  dim_\RR (\mk{z}_{\mk{p}})= \#(\Gamma_{\mk{p}}), \ \ \   dim_\RR(\mk{j}_{\mk{p}_u})/4 \in \NN. \ee We complexify this decomposition and obtain
\be \label{mk{o}=mk{j}_{mk{p}} oplus mk{o}} \mk{o}_{\mk{p}} = \mk{j}_{\mk{p}} \oplus  \mk{z}_{\mk{p}}^\CC \qquad \qquad  dim_\CC(\mk{z}_{\mk{p}}^\CC)= \#(\Gamma_{\mk{p}}), \ \ \   dim(\mk{j}_{\mk{p}})/4 \in \NN. \ee

The following conditions for integrability are obvious (and well known):

\begin{prop} \label{lfim} Let $C : \mk{p}_u \rightarrow \mk{p}_u$ be a complex structure on $\mk{p}_u$. Let $C: \mk{p} \rightarrow \mk{p}$ be the complexification of $C$. Let $\mk{a}^+$ and $\mk{a}^-$ be the $\ri$ and $(-\ri)$-eigen subspaces of $C$ in $\mk{p}$. Then the following conditions are equivalent
\begin{gather*}  \forall X, Y \in \mk{p}_u \ \ [C(X),C(Y)]_{\mk{p}_u}-[X,Y]_{\mk{p}_u}-C([C(X),Y]_{\mk{p}_u})-C([X,C(Y)]_{\mk{p}_u})=0;\\
  \forall X, Y \in \mk{a}^\pm \qquad
 [X,Y]_{\mk{p}} \in \mk{a}^\pm .
 \end{gather*}
\end{prop}

\begin{prop} \label{mkm4} Let $\kappa \in {\bf Aut}(\mk{g})$ be an automorphism of Lie algebras, such that $\kappa(\mk{p})=\mk{p}$, $\kappa(\mk{k})=\mk{k}$. Let $\mk{a} \subset \mk{p}$ be a subspace, such that
 for $ X, Y \in \mk{a}$ we have $
 [X,Y]_{\mk{p}} \in \mk{a} $.  Then
 \begin{gather*} \forall X, Y \in \kappa(\mk{a}) \qquad
 [X,Y]_{\mk{p}} \in \kappa(\mk{a}). \end{gather*}
\end{prop}

\subsubsection{The complex structure $I$}  We define  a complex structure $I: \mk{p}_u  \rightarrow \mk{p}_u$ (with complexification $I: \mk{p} \rightarrow \mk{p}$) by \eqref{mk{o}_{mk{p}_0} = mk{o}_{mk{p}_0}^{a}} and \eqref{wcb2})
\begin{gather}   I( \mk{w}_{\mk{p}}) = \mk{z}_{\mk{p}} \qquad
\label{afi4} I(\mk{j}_{\mk{p}_u})= \mk{j}_{\mk{p}_u} \qquad I(\mk{j}_{\mk{p}})= \mk{j}_{\mk{p}}\\
 \forall \alpha \in \Delta_{\mk{p}}^+ \qquad   I(E_{\alpha}) =  \ri E_{\alpha}, \quad I(E_{-\alpha}) =-  \ri E_{-\alpha}.  \nonumber \end{gather}
%\mk{p}_u^n &=& \mk{o}_{\mk{p}_u}' + \sum_{\gamma \in \Gamma_{\mk{p}}} \RR \ri H_{\gamma} + \sum_{\alpha \in \Delta_{\mk{p}}^+} \left ( %\mk{g}(\alpha) + \mk{g}(-\alpha) \right )\cap\mk{u} \nonumber

We can rewrite  \eqref{decomposition of mk{m}} and \eqref{the mk{k}}, using Corollary \ref{mk{o}=mk{o}_{mk{k}} oplus mk{o}_{mk{p}}} and \eqref{mk{o}=mk{j}_{mk{p}} oplus mk{o}},  as follows:
\be \label{gmg6} & & \mk{p} = \mk{j}_{\mk{p}} \oplus \sum_{\gamma \in \Gamma_{\mk{p}}} (\mc{V}_{\gamma} \oplus gl_{\gamma}(2)) \qquad \qquad \mk{k} = \mk{o}_{\mk{k}} \oplus \sum_{\gamma \in \Gamma_{\mk{k}}} (\mc{V}_{\gamma} \oplus sl_{\gamma}(2)), \ee
where  $gl_\gamma(2)$ is as in Definition \ref{Z_gamma,P_gamma}.

We have
 \be \label{mk{m}_1^n}
 \mk{p}_s^+ &=&  \sum_{\gamma \in \Gamma_{\mk{p}}}(\CC P_\gamma + \CC E_\gamma + {\mc V}_\gamma^+ )  \\ \label{mk{m}_2^n}
   \mk{p}_s^- &=&   \sum_{\gamma \in \Gamma_{\mk{p}}}(\CC Q_\gamma + \CC E_{-\gamma} + {\mc V}_\gamma^- ) = \tau(\mk{p}_s^+) \\
 \label{mk{m}^n} \mk{p}_s&=& \mk{p}_s^+ \oplus \mk{p}_s^- \ \   \Rightarrow   \ \ \mk{p} = \mk{j}_{\mk{p}} \oplus \mk{p}_s, \ \  \mk{g} = \mk{j}_{\mk{p}} \oplus \mk{p}_s \oplus \mk{k}.  \ee
 Obviously for any choice of $X, Y \in \mk{p}_s^\pm $ we have $[X,Y]_{\mk{p}}  \in \mk{p}_s^\pm$. We have proved:
 \begin{prop} \label{prop for mk{p}_s oplus mk{a}} Let $\mk{a} \subset \mk{j}_{\mk{p}}$ be an arbitrary subspace of $\mk{j}_{\mk{p}}$. Then for any $X, Y \in \mk{p}_s^\pm + \mk{a}$ we have $[X,Y]_{\mk{p}} \in \mk{p}_s^\pm + \mk{a}$.
  \end{prop}
If we denote by $\mk{p}^+$ and $\mk{p}^-$ the $ \ri$ and $(-\ri)$-eigen spaces of $I$, then  (see \eqref{afi4})
$ \mk{p}^\pm = \mk{p}_s^\pm + (\mk{p}^\pm \cap \mk{j}_{\mk{p}})$ and from Proposition \ref{prop for mk{p}_s oplus mk{a}} it follows
\be \label{[X,Y]_{mk{m}} in mk{m}_iI} \forall X, Y \in \mk{p}^\pm \qquad [X,Y]_{\mk{p}}  \in \mk{p}^\pm.  \ee

\subsubsection{The complex structure $J$}
For each $\gamma\in\Gamma$ we have an inner automorphism\footnote{Rather a circle of automorphisms depending on the parameter $\rho_\gamma$ (see Definition \ref{X_gamma etc})} of $\mk{g}$
\be \label{phi_p}
 {\bf c}_\gamma = \exp\left(\frac{\pi}{2} ad X_\gamma \right ) \qquad \qquad \gamma \in \Gamma,
\ee
where $X_\gamma$ is given in Definition \ref{X_gamma etc}.

We recall (see \cite{DimTsan10}) some properties of the inner automorphisms ${\bf c}_\gamma$.
\begin{lemma} \label{props of phi_gamma} Let $\gamma, \delta \in \Gamma$. Let $\alpha \in \Phi_{\delta}^+$ and $H \in \mk{h}$.  Then

\begin{gather}
\nonumber  {\bf c}_\gamma \circ \tau = \tau \circ {\bf c}_\gamma,\quad {\bf c}_\gamma \circ {\bf c}_\delta = {\bf c}_\delta \circ {\bf c}_\gamma; \\
\label{prop of phi_gamma 1}    \delta \neq \gamma \ \ \   \Rightarrow \ \ \ \ {\bf c}_\gamma({\mc V}_{\delta}^+)= {\mc V}_{\delta}^+  \\
 \label{prop of phi_gamma 2}    \delta = \gamma \ \ \ \Rightarrow \ \ \  {\bf c}_\gamma(E_{\alpha})= \frac{\sqrt{2}}{2} \left ( E_{\alpha} + \overline{\rho_\gamma} N_{\gamma, - \alpha} E_{\alpha -\gamma} \right )  \\
\label{prop of phi_gamma 4}   {\bf c}_\gamma(E_{\gamma}) =
E_{\gamma}+ \ri  \overline{\rho_\gamma}(Y_\gamma-W_{\gamma})= \overline{\rho_\gamma}(X_\gamma- \ri W_{\gamma})\\
\label{prop of phi_gamma 5}  {\bf c}_\gamma(H) =
H+ \ri \gamma(H) (Y_\gamma+W_{\gamma}).
\end{gather}
\end{lemma}

Let us define the {\bf Cayley transform}
\be \label{theta_{mk{m}}}
 {\bf c}= \prod_{\gamma \in \Gamma} {\bf c}_\gamma,
\ee
where multiplication is the composition of automorphisms (which commute).

By an easy computation, using Lemma \ref{props of phi_gamma}, one can derive (we recall that for $\gamma \in \Gamma_{\mk{p}}$ we denote  $  P_\gamma = W_\gamma - \ri I W_\gamma $, $  Q_\gamma = W_\gamma + \ri I W_\gamma $ and that $IW_\gamma \in \mk{o}$)
\be
 \forall \gamma \in \Gamma \qquad {\bf c}(\mc{V}_\gamma^+) = \sum_{\alpha \in \Phi_{\gamma}^+} \CC(E_{\alpha} + \overline{\rho_\gamma} N_{\gamma, - \alpha} E_{\alpha -\gamma} )  \qquad  {\bf c}(sl_\gamma(2)) = sl_\gamma(2) \nonumber\\[-2mm] \label{gmg7}   \\[-2mm]
\forall \gamma \in \Gamma_{\mk{p}} \qquad  {\bf c}(span\{ P_\gamma, E_\gamma \}) =span\{ E_\gamma - \ri \ol{\rho_\gamma}Q_\gamma, P_\gamma - \ri \ol{\rho_\gamma} E_{-\gamma}\},  \nonumber
 \ee
  hence by \eqref{gmg6} we obtain immediately
\be \label{gmg5}
{\bf c}(\mk{k}) = \mk{k}, \qquad {\bf c}(\mk{p}) = \mk{p}.
\ee
Furthermore, using \eqref{gmg7} we obtain (see \eqref{mk{m}_1^n} also)

\be
{\bf c}(\mk{p}_s^+)&=& \sum_{\gamma \in \Gamma_{\mk{p}}}\sum_{\alpha \in \Phi_{\gamma}^+} \CC(E_{\alpha} + \overline{\rho_\gamma} N_{\gamma, - \alpha} E_{\alpha -\gamma} ) \nonumber \\[-2mm] \label{theta_{mk{m}}(mk{m}_1^n)} \\[-2mm] & &+ \sum_{\gamma \in \Gamma_{\mk{p}}} \CC(E_\gamma - \ri \ol{\rho_\gamma}Q_\gamma) + \sum_{\gamma \in \Gamma_{\mk{p}}}\CC(P_\gamma - \ri \ol{\rho_\gamma} E_{-\gamma}). \nonumber \ee
Now using $\mk{p}_s^- = \tau(\mk{p}_s^+)$ and $\tau \circ {\bf c} = {\bf c} \circ \tau $ we see that
\be {\bf c}(\mk{p}_s^-)&=& \sum_{\gamma \in \Gamma_{\mk{p}}}\sum_{\alpha \in \Phi_{\gamma}^+} \CC(E_{-\alpha} + \rho_\gamma N_{\gamma, - \alpha} E_{-\alpha +\gamma} ) \nonumber \\[-2mm] \label{theta_{mk{m}}(mk{m}_2^n)} \\[-2mm] & &+ \sum_{\gamma \in \Gamma_{\mk{p}}} \CC(E_{-\gamma} - \ri \rho_\gamma P_\gamma) + \sum_{\gamma \in \Gamma_{\mk{p}}}\CC(Q_\gamma - \ri \rho_\gamma E_{\gamma}). \nonumber \ee

Now we  define  a complex structure $J: \mk{p}_u  \rightarrow \mk{p}_u$ with  complexification $J: \mk{p} \rightarrow \mk{p}$  such that (we recall that $ dim_\RR(\mk{j}_{\mk{p}_u})/4 \in \NN$ and $I(\mk{j}_{\mk{p}_u})=\mk{j}_{\mk{p}_u}$)
\begin{gather}
J(\mk{j}_{\mk{p}_u})=\mk{j}_{\mk{p}_u}  \qquad J_{\vert \mk{j}_{\mk{p}_u}} \circ I_{ \vert \mk{j}_{\mk{p}_u}}=- I_{ \vert \mk{j}_{\mk{p}_u}} \circ  J_{\vert \mk{j}_{\mk{p}_u}} \nonumber\\
 \label{afi3}  J(E_{\gamma}) = \ol{\rho_{\gamma}} Q_\gamma  \qquad J(E_{-\gamma}) =-\rho_{\gamma}P_\gamma \qquad\gamma \in \Gamma_{\mk{p}}    \\
 J(E_\alpha) = \ri N_{\gamma,- \alpha}\ol{\rho_{\gamma}} E_{\alpha-\gamma} \quad  J(E_{-\alpha}) =-\ri N_{\gamma,- \alpha} \rho_{\gamma} E_{\gamma-\alpha} \qquad \alpha \in \Phi_{\gamma}^+,\gamma \in \Gamma_{\mk{p}}. \nonumber
 \end{gather}
 It is easy to check that the complex structures given in \eqref{csj1}, \eqref{eee3} define the same hypercomplex structure as \eqref{afi3}.

We see that $J(\mk{p}_s)=\mk{p}_s$ and by \eqref{afi4} we compute $J_{\vert \mk{p}_s} \circ I_{ \vert \mk{p}_s}=- I_{ \vert \mk{p}_s} \circ  J_{\vert \mk{p}_s}$. Therefore
$J \circ I = - I \circ J$.

Let $\mk{p}^+_J$ and $\mk{p}^-_J$ be  the $ \ri$ and $(-\ri)$-eigen spaces of $J$. Comparing \eqref{afi3} and \eqref{theta_{mk{m}}(mk{m}_1^n)}, \eqref{theta_{mk{m}}(mk{m}_2^n)}, we see that
$$
\mk{p}^\pm_J = {\bf c}(\mk{p}_s^\pm) + (\mk{p}^\pm_J \cap \mk{j}_{\mk{p}})= {\bf c}(\mk{p}_s^\pm + (\mk{p}^\pm_J \cap \mk{j}_{\mk{p}}) ).
$$
Now Propositions \ref{mkm4}, \ref{prop for mk{p}_s oplus mk{a}} and formula \eqref{gmg5} imply
\be \label{gmg3} \forall X, Y \in \mk{p}^\pm_J \qquad [X,Y]_{\mk{p}} \in \mk{p}^\pm_J.
 \ee

\subsubsection{$ad(\mk{k})$-invariance of $I$ and $J$}
From the decomposition \eqref{gmg6} and the definitions of $I$ and
$J$ (formulas \eqref{afi4} and \eqref{afi3}) we see that
$\mk{j}_{\mk{p}}$, $\mc{V}_{\gamma}$ and $gl_{\gamma}(2)$  for
$\gamma \in \Gamma_{\mk{p}}$ are  $I$ and $J$ invariant. We have
(see \eqref{csj1} also)
\begin{gather}\label{mkm5}
I_{\vert \mc{V}_{\gamma}} = ad(2 W_{\gamma})_{\vert \mc{V}_{\gamma}},\quad
 J_{\vert \mc{V}_{\gamma}} = - ad(  2 Y_\gamma )_{\vert \mc{V}_{\gamma}}, \qquad \gamma \in \Gamma_{\mk{p}}.
 \end{gather}

It turns out that the subspaces $\mc{V}_{\gamma}$ and $gl_{\gamma}(2)$ are $ad(\mk{k})$-invariant as well, which is a straightforward consequence of  \eqref{gmg2} and \eqref{gmg8}, more precisely  one may compute:
\be \label{ad(X)(V_gamma) subset V_gamma} \forall \gamma \in \Gamma_{\mk{p}},\,   \forall X \in \mk{k} \ \ \qquad ad(X)(\mc{V}_{\gamma}) \subset \mc{V}_{\gamma}, \quad ad(X)(gl_{\gamma}(2))=\{0\}.
\ee

 \begin{lemma} \label{lijm0} The complex structures  $I$, $J$ give a hypercomplex structure on $\mk{p}_u$, i.e. for $X, Y \in \mk{p}_u$ we have
\begin{gather*}
[J(X),J(Y)]_{\vert \mk{p}_u } - [X,Y]_{\vert \mk{p}_u} -J([X, J(Y)]_{\vert \mk{p}_u}) - J([J(X),Y]_{\vert \mk{p}_u})=0; \\
[I(X),I(Y)]_{\vert \mk{p}_u } - [X,Y]_{\vert \mk{p}_u} -I([X, I(Y)]_{\vert \mk{p}_u}) - I([I(X),Y]_{\vert \mk{p}_u})=0.
\end{gather*}
Moreover for $X \in \mk{k}_u, \ Y \in \mk{p}_u$ we have
\begin{gather}
 \label{dec 5}    I \circ ad(X)(Y) = ad(X) \circ I (Y), \quad J \circ ad(X)(Y) =  ad(X) \circ J (Y).
 \end{gather}
 \end{lemma}
\begin{proof}
Vanishing of the Nijenhuis tensors follows from Lemma \ref{lfim} and formulas  \eqref{[X,Y]_{mk{m}} in mk{m}_iI}, \eqref{gmg3}.

 Now we go to the proof of \eqref{dec 5}. Using \eqref{mkm5} and \eqref{ad(X)(V_gamma) subset V_gamma} we readily deduce  that for $\gamma \in \Gamma_{\mk{p}}$, $X \in \mk{k}$ we have
\begin{gather}  (ad(X) \circ I )_{\vert \mc{V}_{\gamma}} = (I \circ ad(X))_{\vert \mc{V}_{\gamma}} \quad (ad(X) \circ J )_{\vert \mc{V}_{\gamma}} = (J \circ ad(X))_{\vert \mc{V}_{\gamma}}, \nonumber \\  \qquad (ad(X) \circ I )_{\vert gl_{\gamma}(2)} = (I \circ ad(X))_{\vert gl_{\gamma}(2)}= (ad(X) \circ J )_{\vert gl_{\gamma}(2)} = (J \circ ad(X))_{\vert gl_{\gamma}(2)} = 0. \nonumber
\end{gather}
Finally, from \eqref{gmg8} and since $\mk{j}_{\mk{p}} \subset
\mk{h}_{\mk{p}}$ we see that
$$
  \forall \alpha \in \Delta_{\mk{k}} \ \ \ \alpha_{\vert \mk{j}_{\mk{p}}} = 0 \ \ \ \Rightarrow \ \ \ \forall X \in \mk{k} \ \ \  ad(X)_{\vert \mk{j}_{\mk{p}}} = 0.
$$
Hence, $\mk{j}_{\mk{p}}$ being $I$ and $J$ invariant, for each $X \in \mk{k}$ we get
\begin{gather*}
(ad(X) \circ I )_{\vert \mk{j}_{\mk{p}}} = (I \circ ad(X))_{\vert \mk{j}_{\mk{p}}}= (ad(X) \circ J )_{\vert \mk{j}_{\mk{p}}} = (J \circ ad(X))_{\vert \mk{j}_{\mk{p}}} = 0.
\end{gather*}
Using the decomposition \eqref{gmg6} of $\mk{p}$, we see that for
all $X \in \mk{k},\quad Y \in \mk{p}$ we have
\begin{gather*}
 I \circ ad(X)(Y) = ad(X) \circ I (Y) \quad J \circ ad(X)(Y) = ad(X) \circ J (Y).
\end{gather*}
Recalling that $I$ and $J$ commute with $\tau$, we obtain \eqref{dec 5}.
\end{proof}

So we are given a pair of groups $({\bf U}, {\bf K}_u)$ with connected ${\bf K}_u$ associated with a hypercomplex pair $(\mk{u}, \mk{k}_u)$. We constructed anti-commuting complex structures $I$ and $J$ on $\mk{p}_u$, which have the properties listed in Lemma \ref{lijm0} and also we have the reductivity condition $[\mk{k}_u, \mk{p}_u] \subset \mk{p}_u$. Then, it is known (see e.g. \cite{Kobayashi69}, Ch. X), that $I$ and $J$ determine  left-invariant, integrable complex structures on the coset space ${\bf U}/{\bf K}_u$, which thus becomes a compact, homogeneous, hypercomplex space.

\begin{df} \label{hypercomplex space associated with a hyp} Let $(\mk{u}, \mk{k}_u)$ be a hypercomplex pair and let $({\bf U},{\bf K}_u)$ be a pair of groups associated with it. Let ($I_M, J_M$) be a LIHCS on $M = {\bf U}/{\bf K}_u$ and let $\pi : {\bf U} \rightarrow M$ be the natural projection. We shall say that the  homogeneous hypercomplex space $(M, I_M, J_M)$ {\bf is associated with the hypercomplex pair $(\mk{u}, \mk{k}_u)$} if the complex structures $I, J$ on $\mk{p}_u$ constructed above satisfy:
\begin{gather*}
 I = ((\rd
\pi )_{\vert \mk{p}_u} )^{-1}\circ I_M(o) \circ (\rd \pi
)_{\vert \mk{p}_u}, \ \ J = ((\rd \pi )_{\vert \mk{p}_u}
)^{-1}\circ J_M(o) \circ (\rd \pi )_{\vert \mk{p}_u}.
\end{gather*}
\end{df}

In the next section we prove that any ${\bf U}$- left-invariant, hypercomplex structure on a coset space ${\bf U}/{\bf K}_u$ is associated with the hypercomplex pair $(\mk{u}, \mk{k}_u)$.

\section{Necessary conditions } \label{nc0}
\subsection{Introduction}
\label{introduction nc0} In this section we shall discuss  a coset space  $M={\bf U}/{\bf K}_u$, where ${\bf U}$ is a compact, connected Lie group
and ${\bf K}_u$ is a closed subgroup. We shall assume that ${\bf U}$ acts effectively on $M$. Furthermore we shall assume that there are   two left-invariant
complex structures   $ I_M $ and $J_M$, which anti-commute, on $M$. By left-invariant we mean that all the left translations $l_u : M \rightarrow M$,
where $u$ varies through ${\bf U}$, are holomorphic with respect to both $I_M$ and $J_M$.

We are going to prove that $M$ is associated to a hypercomplex pair.

\subsection{\texorpdfstring{$M$}{\space} as a coset space of a complex Lie group.} \label{coset space of a complex Lie group}
So, we are given a  compact coset space  $M={\bf U}/{\bf K}_u$ and two complex structures $I_M$ and $J_M$ on it. Let ${\bf G}_{I0}$ be the identity component of the group of
all biholomorphisms of the complex manifold $(M, I_M)$. By results of Bochner
and Montgomery in \cite{MontgomeryBoch1} and \cite{MontgomeryBoch} it follows that one can define a
complex structure on ${\bf G}_{I0}$, such that ${\bf G}_{I0}$ becomes a
complex Lie group and the action on the complex manifold $(M, I_M)$
\be \label{action of {bf G}_I} {\bf G}_{I0} \times M \rightarrow M
\qquad (f, m) \in {\bf G}_I \times M \mapsto f(m) \in M \ee is
holomorphic.

Let $\pi$ be the natural projection
\be \label{projection} \pi: {\bf U} \rightarrow {\bf U}/{\bf K}_u=M \qquad \quad \pi(e)=o. \ee
We are given, that ${\bf U}$ is a compact, connected subgroup of ${\bf G}_{I0}$ which acts effectively and transitively on $M$.
Let $\mk{u}$ and $\mk{g}_{I0}$ be the Lie algebras of ${\bf U}$ and ${\bf G}_{I0}$, respectively. Since $\mk{u}$ is compact, then we have the decomposition
 $ \mk{u} = \mk{c}_u \oplus \mk{u}_s $ (see \eqref{mk{u} = mk{z}(mk{u}) oplus mk{u}_s} ). We recall \cite[page 132]{Helgason78} that the semi-simple part  $\mk{u}_s$ is  a compact Lie algebra also. We denote by ${\bf U}_s$ the connected subgroup of ${\bf U}$ with Lie algebra $\mk{u}_s$. Since $\mk{u}_s$ is a compact, semi-simple Lie algebra, then (see \cite[page 133]{Helgason78}) ${\bf U}_s$ is a compact Lie group, hence it is a closed, compact, semi-simple  subgroup of ${\bf G}_{I0}$.

\begin{lemma} \label{lemma about complexification of bf u_s}
The real subalgebra $\mk{u}_s$  of the complex Lie algebra $\mk{g}_{I0}$
satisfies  \be \label{lemma about complexification of bf u_seq}
\mk{u}_s \cap \ri \mk{u}_s = \{0\}.\ee In particular the subspace
$\mk{u}_s+ \ri \mk{u}_s\subset \mk{g}_{I0}$ is a complex
semi-simple subalgebra of $\mk{g}_{I0}$, which is a
complexification of $\mk{u}_s$, we shall denote it by $\mk{g}_s$.
\end{lemma}
\begin{proof}

Let ${\bf G}_m$ be a maximal connected semi-simple
subgroup of ${\bf G}_{I0}$, which contains ${\bf U}_s$. Actually (see the proof of Theorem 2.3 in \cite{Wang54}),
${\bf G}_m$ is generated by the semi-simple part of a
Levi-decomposition of the complex Lie algebra of ${\bf G}_{I0}$.
Therefore ${\bf G}_m$ is a complex semi-simple  Lie group and
${\bf U}_s$ is its connected, compact, real Lie subgroup. Let
$\mk{g}_m$ be  the complex Lie algebra of ${\bf G}_m$. Then the
Lie algebra $\mk{u}_s$ of ${\bf U}_s$ is a real
subalgebra of $\mk{g}_m$.

Now ${\bf G}_m$ is a connected complex semi-simple Lie
group, therefore it is semi-simple as a real Lie group also. Let
us take any compact real form  $\mk{b}$ of $\mk{g}_m$. Then
the subgroup ${\bf B}$ of ${\bf G}_m$ generated by $\mk{b}$ is a maximal
compact subgroup. Let ${\bf B}_u$ be
a maximal compact subgroup of ${\bf G}_m$, containing ${\bf U}_s$,
let the Lie algebra of ${\bf B}_u$ be $\widetilde{\mk{u}}$. The
subgroup ${\bf B}_u$ is connected (all maximal compact subgroups
in a semi-simple connected group are connected  and conjugate to each other- see
\cite[page 356]{Helgason78}).  Therefore
there exists an element $g \in {\bf G}_m$ such that $\alpha_g({\bf
B})=g {\bf B} g^{-1} = {\bf B}_u $, where by $\alpha_g: {\bf G}_m
\rightarrow {\bf G}_m$ we denote the conjugation of ${\bf G}_m$
with the element $g$. Since ${\bf G}_m$ is a complex group then
$(\rd \alpha_g)_e = Ad(g): \mk{g}_m \rightarrow \mk{g}_m$ is
an automorphism of the complex Lie algebra $\mk{g}_m$. Since
$\alpha_g({\bf B})= {\bf B}_u $ then, obviously, $Ad(g)(\mk{b})
= \widetilde{\mk{u}}$. Since $\mk{b}$ is a compact form, then
$\widetilde{\mk{u}}$ is a compact form of $\mk{g}_m$ and
$\widetilde{\mk{u}} \cap \ri \widetilde{\mk{u}} = \{0\}$. On
the other hand $\mk{u}_s \subset \widetilde{\mk{u}}$ and
\eqref{lemma about complexification of bf u_seq} follows.
\end{proof}

Let us denote \be \label{mk{u}} \mk{g}_I= \mk{u} + \ri \mk{u}\subset \mk{g}_{I0}. \ee
If we denote $ \mk{c}_I =\mk{c}_u +\ri \mk{c}_u \subset \mk{g}_I$ then
$\mk{c}_I$ is abelian and obviously $[\mk{c}_I, \mk{g}_s] \subset
\{0\}$. Therefore, since $\mk{g}_s$ is semi-simple, we  have $
\mk{c}_I \cap \mk{g}_s = \{0\}$ and we may write \be
\label{mk{g} = mk{c} oplus mk{g}_s} \mk{g}_I = \mk{c}_I \oplus \mk{g}_s \qquad
\qquad\mk{c}_I=\mk{c}_u + \ri \mk{c}_u.\ee

Let ${\bf G}_I$ be the connected subgroup of ${\bf G}_{I0}$ generated by the subalgebra $\mk{g}_I \subset \mk{g}_{I0}$.
Obviously ${\bf G}_I \supset {\bf U}$ and therefore ${\bf G}_I$ acts transitively on $M = {\bf U}/{\bf K}_u$.
Hence the complex manifold $(M,I_M)$ is biholomorphic with the complex coset space ${\bf G}_I/{\bf L}$, where (we recall that $\pi(e)=o$)
\be {\bf L}=\{g  \in {\bf G}_I \vert g(o) = o \}. \ee
Obviously $ {\bf K}_u = {\bf U} \cap {\bf L}. $
Let us denote by $\mk{l}$ the complex subalgebra of $\mk{g}_I$, corresponding to the complex subgroup ${\bf L}$ and by $\mk{k}_u$ the subalgebra of $\mk{u}$,
corresponding to the subgroup ${\bf K}_u$.

Let $p$ be the natural projection
\ben p: {\bf G}_I \rightarrow M \simeq {\bf G}_I/{\bf L}.\een
Since $p \circ in = \pi $, where $in: {\bf U} \rightarrow {\bf G}_I$ is the embedding of ${\bf U}$ in ${\bf G}_I$, then
the restriction of $\rd p$ to $\mk{u}$ is equal to $\rd \pi : \mk{u} \rightarrow T_o(M)$, i. e.
\be \label{rd p(X) = rd pi (X)} \forall X \in \mk{u}  \quad \rd p(X) = \rd \pi (X). \ee

 Furthermore, since the projection $p$ is holomorphic then  \be
\label{rd p(i X) = I_o (rd p(X) )} \forall X \in \mk{g}_I \qquad
\rd p (\ri X) = I_M (\rd p (X) )= \ri \rd p (X).\ee Since $ {\bf K}_u = {\bf U} \cap {\bf L} $
then
$ \mk{k}_u = \mk{u} \cap \mk{l}.$ Furthermore
 \be \label{ker rd pi} \ker(\rd \pi) = \mk{k}_u \qquad  \ker(\rd p) = \mk{l} \qquad  \mk{u} \cap \mk{l} =
\mk{k}_u.\ee

\noindent

Therefore the  subspace \be \label{frak{m}_0 = X in
frak{u};forall Y in frak{k}_0} \mk{p}_u = \{X \in \mk{u}\vert \ \
\forall Y \in \mk{k}_u \scal{X, Y}=0 \}\ee satisfies
 \be
\label{properties of frak{m}_0} \mk{u}= \mk{k}_u \oplus \mk{p}_u.
\ee
Since $\mk{k}_u$
is $Ad({\bf K}_u)$-invariant  then
obviously $\mk{p}_u$ is $Ad({\bf K}_u)$-invariant:
\be
\label{properties of frak{m}_0 1} Ad({\bf K}_u) (\mk{p}_u)
\subset \mk{p}_\mk{u} \ \ \Rightarrow \ \ [\mk{k}_u, \mk{p}_u]
\subset \mk{p}_u.
\ee

 Let $\mk{g}$ and $\mk{k}$   be the complexifications of $\mk{u}$, $\mk{k}_u$ and let $\tau $ be the conjugation of $\mk{g}$, such that $\mk{u}=\mk{g}^{\tau}$.  We shall denote the  center of $\mk{g}$ by $\mk{c}$.

Obviously, we have
\be \label{mk{k}} \mk{k} = \mk{k}_u + \ri \mk{k}_u \subset \mk{g} \qquad \mk{k}_u \cap \ri \mk{k}_u=\{0 \}.\ee

\begin{remark} \label{the complexification of mk{u}} Let us summarize. We have $\mk{g}_{I0}$ - the complex Lie algebra of the complex Lie group ${\bf G}_{I0} $ (the identity component of the group of
all biholomorphisms of the complex manifold $(M, I_M)$). In
$\mk{g}_{I0}$ live our main actors $\mk{u} \subset \mk{g}_{I0}$,
$\mk{k}_u \subset \mk{g}_{I0}$ which are the real Lie algebras  of
${\bf U}$ and ${\bf K}_u$. And we denote by $\mk{g}_I$ the complex
subalgebra of $\mk{g}_{I0}$ generated by $\mk{u}$, i. e.
$\mk{g}_I= \mk{u} + \ri \mk{u}\subset \mk{g}_{I0}$, ${\bf G}_{I}
\subset {\bf G}_{I0}  $ is the corresponding complex Lie group. By
$\mk{g}$ and $\mk{k}\subset \mk{g}$ we denote the usual
complexifications of the real Lie algebras $\mk{u}$ and
$\mk{k}_u$, and $\tau$ is the conjugation.

Since $\mk{u}$ is compact we can decompose it as direct sum of  center and semisimple part: $\mk{u} = \mk{c}_u \oplus \mk{u}_s $.
From Lemma \ref{lemma about complexification of bf u_s}  we see that the complex subalgebra of $\mk{g}$ generated by $\mk{u}_s$ is isomorphic to
the subalgebra $\mk{g}_s = \mk{u}_s+\ri \mk{u}_s \subset \mk{g}_I$ of $\mk{g}_I$. Therefore we can identify these two algebras (the identification is: for $X, Y \in \mk{u}_s$ $X + \ri Y \in \mk{g} \mapsto X + \ri Y \in \mk{g}_I$), and we shall   denote them  by the common notation: $\mk{g}_s$. With this convention we have  $\mk{g}=\mk{c} \oplus \mk{g}_s$,
where $\mk{c}$ is the  center of $\mk{g}$.

Denoting by  $\tau_s$ the restriction of $\tau$ to $\mk{g}_s$, we have
\begin{gather}\label{tau new}
 \forall Y, Z \in \mk{u}  \quad  \tau_s(Y+ \ri Z) =  Y - \ri Z, \quad  \tau(\mk{c})=\mk{c}.  \end{gather}
 We represented our homogeneous space $M$ as complex factor $ {\bf G}_I/{\bf L}$ and denoted by $\mk{l}$  the complex subalgebra of $\mk{g}_I$, corresponding to the complex subgroup ${\bf L} \subset {\bf G}_I$.

The Lie algebras $\mk{g}_{I} \supset \mk{l}$ are auxiliary and will be used to prove the  regularity of  $\mk{k}, \mk{p}$.

\end{remark}

\subsection{Construction of I, J} \label{some notations}
Let  us  denote by $\mk{p}$ the complex subspace of $\mk{g}_I$, generated by $\mk{p}_{\mk{u}}$:
\be \label{mk{k} and mk{m}}
 \mk{p} = \mk{p}_u + \ri \mk{p}_u \subset \mk{g}_I.
 \ee
We shall prove soon (Corollary \ref{mkm6}) that the sum above is actually direct sum of real vector spaces, i. e. that $\mk{p}$ is a complexification of $\mk{p}_u$, therefore  we may regard $\mk{p}$ as a subspace of $\mk{g}$.

We define  complex structures $I, J$ on $\mk{p}_u$, using the complex
structures $I_M(o), J_M(o)$ on $T_o(M)$, as follows
\begin{gather}
\label{wcb5}
I = ((\rd
\pi )_{\vert \mk{p}_u} )^{-1}\circ I_M(o) \circ (\rd \pi
)_{\vert \mk{p}_u}, \ \ J = ((\rd \pi )_{\vert \mk{p}_u}
)^{-1}\circ J_M(o) \circ (\rd \pi )_{\vert \mk{p}_u}.
\end{gather}

\begin{lemma} \label{extension of I} The complex structure $I$ on the real vector space $\mk{p}_u$ can be uniquely extended to a complex structure on the complex vector space $\mk{p}$. We shall denote this extension by the same letter $I : \mk{p} \rightarrow \mk{p}$.
\end{lemma}
\begin{proof} Let us take any element $Z \in \mk{p}$. Let us assume that we have two representations of $Z$ of the type
\be \label{extension of I equation}
 Z = X_1 + \ri Y_1 = X_2 + \ri Y_2, \qquad X_1, X_2, Y_1, Y_2 \in \mk{p}_u
 \ee
 therefore
\begin{gather*}
\rd p(X_1 + \ri Y_1) = \rd p(X_2 + \ri Y_2) \ \ \Rightarrow \ \  \rd p(X_1) + \ri \rd p (Y_1) = \rd p(X_2) + \ri \rd p(Y_2) \ \ \Rightarrow \  \\
  \rd \pi(X_1) + \ri \rd \pi (Y_1) = \rd \pi (X_2) + \ri \rd \pi(Y_2) \ \ \Rightarrow \ \  \rd \pi(X_1 +I(Y_1)) = \rd \pi (X_2 + I(Y_2)).\end{gather*}
  The last equation implies $X_1 +I(Y_1) = X_2 + I(Y_2)$ (taking into account that $X_1 +I(Y_1) \in \mk{p}_u$ and $X_2 +I(Y_2) \in \mk{p}_u$), hence
  \be \label{extension of I1} X_1 - X_2 = I(Y_2 - Y_1) \ \ \Rightarrow \ \ I(X_1-X_2)= Y_1 - Y_2. \ee
  Furthermore, we have   $\ri (Y_2 - Y_1) = X_1 - X_2$ (see \eqref{extension of I equation}), therefore $\ri (Y_1 - Y_2) \in \mk{p}_u$, besides
 \ben \rd \pi (\ri (Y_1 - Y_2)) = \rd p (\ri (Y_1 - Y_2)) = \ri \rd p (Y_1 - Y_2) = \ri \rd \pi (Y_1 - Y_2) = \rd \pi (I(Y_1 - Y_2)),\een
 which implies
 \be \label{extension of I2} \ri (Y_1 - Y_2) = I(Y_1 - Y_2) \ \ \Rightarrow \ \   Y_1 - Y_2 + \ri I(Y_1 - Y_2)=0.\ee From \eqref{extension of I1} it follows that
 \ben I(X_1 - X_2) + \ri I(Y_1 - Y_2) = Y_1 - Y_2 +\ri I(Y_1 - Y_2)\een and now from \eqref{extension of I2}  we infer
 \ben  I(X_1 - X_2) + \ri I(Y_1 - Y_2) =0 \qquad \Rightarrow \qquad  I(X_1) + \ri I(Y_1) = I(X_2) + \ri I(Y_2).\een
 Hence we see that the operator on $\mk{p}$ defined by the rule
 \be X + \ri Y \mapsto I(X) + \ri I(Y) \qquad \quad X,Y \in \mk{p}_u \nonumber \ee
 is well defined and this is the  extension of $I$ from $\mk{p}_u$ to $\mk{p}$.
\end{proof}

 Let
$\mk{p}^+$ and $\mk{p}^-$ be the $\ri$ and
$(-\ri)$-eigenspaces of $I$, respectively, i. e. \be \label{frak{m}_1 and frak{m}_2} \mk{p}^+=\{X \in \mk{p}\vert \ \ I(X)= \ri X \} \quad \quad
\mk{p}^-=\{X \in \mk{p}\vert \ \ I(X)= -\ri X \}.  \ee
In addition to \eqref{rd p(i X) = I_o (rd
p(X) )} we have \be \label{rd p( I(X)) = I_o (rd p (X)) frak{m}}
\forall X \in \mk{p} \qquad  \rd p( I(X)) = \ri  \rd p (X). \ee
Indeed, let $X_1 , X_2 \in \mk{p}_u$ then
\begin{gather*}  \rd p( I(X_1 + \ri X_2)) = \rd p( I(X_1) + \ri I(X_2))= \rd \pi (I(X_1)) + \ri \rd \pi (I(X_2))  \\
=\ri \rd \pi (X_1) + \ri \ri  \rd \pi (X_2) = \ri (\rd \pi (X_1) +  \ri  \rd \pi (X_2)) = \ri (\rd p (X_1) +   \rd p (\ri X_2)) \\
=\ri \rd p (X_1 + \ri X_2).  \end{gather*}
\begin{lemma} \label{mkm3} The following equality is valid
$$
\mk{l} \cap \mk{p} = \mk{p}^-.\qquad  \mk{g}_I = \mk{l} \oplus \mk{p}^+.
$$
\end{lemma}
\begin{proof} Let $Z = X + \ri Y \in \mk{l}\cap \mk{p}$, where $X, Y \in \mk{p}_u$, then we have
$$
 0=\rd p(X + \ri Y) = \rd p  (X) + \ri \rd p (Y) = \rd \pi (X) + \rd \pi (I(Y)) = \rd \pi (X + I(Y)).
$$
This equality, and $X + I(Y) \in \mk{p}_u$, imply $X + I(Y) = 0$. Therefore $Y = I(X)$ and
$Z =X + \ri Y  =X + \ri I(X)$, hence $I(Z)=I(X + \ri I(X))=-\ri ( X + \ri I(X))=-\ri Z$. Thus we proved that
$ \mk{l} \cap \mk{p} \subset \mk{p}^-.$
Conversely, let $Z \in \mk{p}^-$, then $I(Z) = -\ri Z$ and $Z = \ri I(Z)$, therefore
\ben
\rd p (Z) = \rd p (\ri I(Z)) = \ri \ri \rd p (Z) = - \rd p(Z) \ \ \Rightarrow \rd p (Z) = 0,
\een
hence $Z \in \mk{l}$ and we see that $\mk{p}^- \subset \mk{l} \cap \mk{p}$.
So we proved the first equality.  Now we have $\mk{l} \cap \mk{p}^+ \subset (\mk{l} \cap \mk{p}) \cap \mk{p}^+ = \mk{p}^- \cap \mk{p}^+ = \{0\}$ and we see that $\mk{l} \cap \mk{p}^+ = \{0\}$. Since $\mk{u}=\mk{k}_u + \mk{p}_u$, then obviously $\mk{g}_I=\mk{k} + \mk{p}$, hence $\mk{g}_I=\mk{k} + \mk{p}^+ + \mk{p}^- \subset \mk{l} + \mk{p}^+$. The Lemma follows.

\end{proof}

\begin{coro} \label{mkm6} The complex subspace $\mk{p}\subset \mk{g}_I$ is a complexification of $\mk{p}_u \subset \mk{u}$, i. e. we have
$$
 \mk{p} = \mk{p}_u \oplus \ri \mk{p}_u \qquad \mbox{direct vector space sum}.
 $$
In particular the complex structure $J$ on $\mk{p}_u$  may be extended by complex linearity to the complex subspace $\mk{p}$, {\bf we shall denote this extension by the same letter} $J$.
\end{coro}
\bpr
By the previous lemma we have $\dim_\CC(\mk{p}^+) = \dim_\CC(\mk{g}_I) - \dim_\CC(\mk{l}) = \dim_\CC(M)$, hence $\dim_\CC(\mk{p}) = 2 \dim_\CC(\mk{p}^+) = \dim_\RR(M) = \dim_\RR(\mk{p}_u)$.  So, we see that
$$
  \dim_\CC(\mk{p}) = \dim_\RR(\mk{p}_u) \ \  \Rightarrow \ \ \dim_\RR(\mk{p}) = 2 \dim_\RR(\mk{p}_u).
$$
On the other hand
\ben \dim_{\RR}(\mk{p}) = \dim_\RR(\mk{p}_u) + \dim_\RR(\ri \mk{p}_u) - \dim_\RR(\mk{p}_u \cap \ri \mk{p}_u) = 2 \dim_\RR(\mk{p}_u) - \dim_\RR(\mk{p}_u \cap \ri \mk{p}_u)\een
and we infer (using $\dim_\RR(\mk{p}) = 2 \dim_\RR(\mk{p}_u)$) that $\dim_\RR(\mk{p}_u \cap \ri \mk{p}_u)=0$, i. e.
$\mk{p}_u \cap \ri \mk{p}_u =\{0\}$ and the corollary follows.
\epr

\begin{remark}\label{coro for wt{mk{m}} as a subspace of mk{g}} From Corollary \ref{mkm4} it follows that the subspace of the complexification $\mk{g}$, generated by $\mk{p}_u$, is isomorhic to the subspace $\mk{p}$ of $\mk{g}_I$. We shall identify these two subspaces and  shall denote them by a common notation $\mk{p}$ (the isomorphism, that identifies them maps a sum $X + \ri Y $ in $\mk{g}_I$, where $X, Y \in \mk{p}_u$, to the sum $X + \ri Y $ in $\mk{g}$). Of course the complex structures $I$, $J$ and the subspaces $\mk{p}^+$,  $\mk{p}^-$ can be regarded as objects related to the subspace of $\mk{g}$ generated by $\mk{p}_u$ so, they satisfy the relations
\begin{gather}    I \circ \tau = \tau \circ I \quad J \circ \tau = \tau \circ J \quad I \circ J = -J \circ I \nonumber \\[-2mm]
\label{coro for wt{mk{m}} as a subspace of mk{g}eq} \\[-2mm]
 \quad \mk{p}^- = \tau(\mk{p}^+) \quad J(\mk{p}^+) = \mk{p}^- . \nonumber \end{gather}

  Since $\mk{g}$ is the complexification of $\mk{u}$, the relations $\mk{u} = \mk{k}_u + \mk{p}_u$, $\mk{k}_u \cap \mk{p}_u = \{0\} $, $[\mk{k}_u, \mk{p}_u] \subset \mk{p}_u$ imply
\be  \label{properties of frak{m}} \mk{g}  = \mk{k} \oplus \mk{p},\quad [\mk{k}, \mk{p}] \subset \mk{p}. \ee
\end{remark}

\subsection{Fixing the Cartan subalgebra}
\label{fixing the Cartan}

The coset space ${\bf G}_I/{\bf L}$ is compact. From the
normalizer theorem (A. Borel - R. Remmert) (see e.g.
\cite{Akhiezer} (page 80)), it follows that the normalizer
$\mk{n}(\mk{l})$ of $\mk{l}$ in $\mk{g}_I$ is a parabolic
subalgebra of $\mk{g}_I$, i. e. \be \label{mk{n}{mk{l}}=mk{c}
oplus mk{p}} \mk{n}(\mk{l}) = \mk{c}_I \oplus \mk{s}\ee where
$\mk{s}$ is a parabolic subalgebra of $\mk{g}_s$ and $\mk{c}_I$ is
the center of $\mk{g}_I$ (see \eqref{mk{g} = mk{c} oplus
mk{g}_s}).  Let us recall that $\tau_s$ is the conjugation of
$\mk{g}_s$, which corresponds to the compact real form $\mk{u}_s
\subset \mk{g}_s$.  Since $\mk{s}$ is a parabolic subalgebra of
$\mk{g}_s$ then (see \cite{Wolf69}, Theorem 2.6 ) there exists a
$\tau_s$-invariant Cartan subalgebra $\mk{h}_s$ of $\mk{g}_s$,
such that $\mk{h}_s \subset \mk{s}$. Let us denote (recall that we
identify the semisimple part of $\mk{g}$ with the semisimple part
of $\mk{g}_I$ and denote them by the common notation $\mk{g}_s$):
\be \label{mk{h}} \mk{h}_I= \mk{c}_I \oplus \mk{h}_s \subset
\mk{g}_I, \qquad \mk{h}= \mk{c} \oplus \mk{h}_s \subset \mk{g} \ee
then $\mk{h}_I$ and $\mk{h}$ are Cartan subalgebras of the
reductive algebras $\mk{g}_I$ and $\mk{g}$, respectively, and we
obtain  that $\mk{h}_I \subset \mk{n}(\mk{l})$, i. e. \be
\label{[mk{h}, mk{l}] subset mk{l}}
 [\mk{h}_I, \mk{l}] \subset \mk{l}.\ee
We racall that  $\tau_{\vert \mk{g}_s} = \tau_s $, $\tau(\mk{c})=\mk{c}$ (see \eqref{tau new}), hence
\be \label{tau(mk{h}) = mk{h}}   \tau(\mk{h}) = \mk{h}.\ee

\begin{remark} Further in this text, we treat the  algebra $\mk{g}$ and its Cartan subalgebra  $\mk{h}$, the algebras $\mk{g}_I$ and $\mk{h}_I \subset \mk{g}_I$ are used only for the proof of the regularity of $\mk{k}$ and $\mk{p}$.

So everywhere $\Delta$ is the root system w.r. to the Cartan subalgebra $\mk{h} \subset \mk{g}$ and we fix a Weyl-Chevalley basis of $\mk{g}$ as in \eqref{wcb}.

The choice of the positive roots  $\Delta^+ \subset \Delta$ will be explained in subsection \ref{the system Gamma}.

\end{remark}

\subsection{Regularity of \texorpdfstring{$\mk{k}$}{\space} and \texorpdfstring{$\mk{p}$}{\space}} \label{regularity of mk{k} and mk{m}}

We have shown that $\mk{l}$ is a  regular subalgebra of $\mk{g}_I$. In this subsection  we shall prove that $\mk{k}$  and $\mk{p}$ are regular in $\mk{g}$ (for the notations $\mk{k}$ and $\mk{p}$ see \eqref{mk{k}} and  Remark \ref{coro for wt{mk{m}} as a subspace of mk{g}}).

Let us  recall first that $\mk{h}_s$ is $\tau_s$-invariant and therefore we have
\be \label{mk{h}_s = (mk{h}_s cap mk{u}_s ) oplus...} \mk{h}_s = (\mk{h}_s \cap \mk{u}_s ) \oplus \ri (\mk{h}_s \cap \mk{u}_s ), \ee
for the definition of $\mk{u}_s$ see \eqref{mk{u} = mk{z}(mk{u}) oplus mk{u}_s} and Lemma \ref{lemma about complexification of bf u_s}.

\begin{lemma} \label{lemma for mk{k}} The subalgebra $\mk{k} \subset \mk{g}$ is regular w. r. to $\mk{h}$, i. e.
\be \label{frak{k} is regular} \mk{k} = (\mk{k} \cap \mk{h} ) \oplus
\bigoplus_{  \alpha \in \Delta_{\mk{k}} } \mk{g}(\alpha),  \quad \mbox{where} \quad \Delta_{\mk{k}}=\{ \alpha \in \Delta\vert
\mk{g}(\alpha) \subset \mk{k} \}.
\ee
The set $\Delta_{\mk{k}}$ is symmetric with respect to $0$, i. e.  $-\Delta_{\mk{k}}=\Delta_{\mk{k}}$.

Consequently $\mk{p}$ is a regular subspace of $\mk{g}$.
\end{lemma}

\bpr
Since  $[\mk{h}_I,\mk{l}] \subset \mk{l}$ then $[\mk{h}_s \cap \mk{u}_s,\mk{l}] \subset \mk{l}$, besides since $\mk{u}$ is a subalgebra then $[\mk{h}_s \cap \mk{u}_s, \mk{u}] \subset \mk{u}$,   hence  $[\mk{h}_s \cap \mk{u}_s , \mk{u} \cap \mk{l}] \subset \mk{u}\cap \mk{l}$. On the other hand  $\mk{k}_u = \mk{l}\cap \mk{u}$ (see \eqref{ker rd pi}), which implies
\be [\mk{h}_s\cap \mk{u}_s, \mk{k}_u] \subset \mk{k}_u,\ee
which, by \eqref{mk{h}_s = (mk{h}_s cap mk{u}_s ) oplus...} and $\mk{k}=\mk{k}_u + \ri \mk{k}_u \subset \mk{g}$, implies $[\mk{h}_s,\mk{k}] \subset \mk{k}$. Now \eqref{mk{h}}  gives $[\mk{h},\mk{k}] \subset \mk{k}$, i. e. $\mk{k}$ is regular and we obtain \eqref{frak{k} is regular}.

The set $\Delta_{\mk{k}}$ is symmetric with respect to $0$, i. e.  $-\Delta_{\mk{k}}=\Delta_{\mk{k}}$, since $\tau(\mk{k})=\mk{k}$ and $\tau(\mk{g}(\alpha)) = \mk{g}(-\alpha)$ for $\alpha \in \Delta$.

Since $ \mk{p}_u$ is the orthogonal complement of $\mk{k}_u$ in $\mk{u}$ with respect to the Killing form, hence the regularity of $ \mk{p} $ follows from the regularity of $\mk{k}$.
 \epr
Now we shall prove the regularity of $\mk{p}^+$ and $\mk{p}^-$

\begin{lemma} \label{lemma wcb8}   The subspaces $\mk{p}^+$ and $\mk{p}^-$ are regular in $\mk{g}$  w. r. to $\mk{h}$, i. e.
\be
 \label{wcb8}  \mk{p}^+ = (\mk{p}^+ \cap \mk{h}) \oplus \bigoplus_{\alpha \in \Delta_{\mk{p}}^+ } \mk{g}(\alpha), \ \   \mbox{where} \quad \Delta_{\mk{p}}^+=\{ \alpha\vert  \mk{g}(\alpha) \subset \mk{p}^+ \} \\
 \label{wcb9} \mk{p}^- =(\mk{h} \cap \mk{p}^-) \oplus \bigoplus_{\alpha \in \Delta_{\mk{p}}^+} \mk{g}(-\alpha) \qquad \mk{h} \cap \mk{p}^- = \tau(\mk{h} \cap \mk{p}^+). \ee
Furthermore,
\be
\mk{h}\cap \mk{p} = (\mk{h} \cap \mk{p}^+) \oplus (\mk{h} \cap \mk{p}^-) \qquad  \mk{h} = (\mk{h} \cap \mk{p}) \oplus (\mk{h} \cap \mk{k}) \\ \label{wcb7} \Delta \setminus \Delta_{\mk{k}}= \Delta_{\mk{p}}^+ \cup (-\Delta_{\mk{p}}^+) \qquad \Delta_{\mk{p}}^+ \cap  (-\Delta_{\mk{p}}^+)=\emptyset.
\ee
  \end{lemma}
\bpr
We have proved that $\mk{l}$ and $\mk{p}$ are regular subalgebras of $\mk{g}_I$ with respect to the Cartan subalgebra $\mk{h}_I$, i. e. $[\mk{h}_I,\mk{l}] \subset \mk{l}$ and $[\mk{h}_I,\mk{p}] \subset \mk{p}$. Therefore $[\mk{h}_I, \mk{l}\cap \mk{p}] \subset \mk{l} \cap \mk{p}$. On the other hand $\mk{l}\cap \mk{p} = \mk{p}^-$ (see  Lemma \ref{mkm3}). Therefore $[\mk{h}_I,\mk{p}^-] \subset \mk{p}^-$, i. e. $\mk{p}^-$ is regular in $\mk{g}_I$ with respect to $\mk{h}_I$. Now obviously $[\mk{h}_u,\mk{p}^-] \subset \mk{p}^-$ in $\mk{g}_I$, hence one can verify, taking into account $[\mk{h}_u, \mk{p}_u] \subset \mk{p}_u$ and Remark \ref{coro for wt{mk{m}} as a subspace of mk{g}}, that  $[\mk{h}_u,\mk{p}^-] \subset \mk{p}^-$ in $\mk{g}$ as well. Therefore $[\mk{h},\mk{p}^-] \subset \mk{p}^-$ in $\mk{g}$ and  $\mk{p}^-$ is regular. From \eqref{coro for wt{mk{m}} as a subspace of mk{g}eq} in  Remark \ref{coro for wt{mk{m}} as a subspace of mk{g}} we know that $\mk{p}^+ = \tau(\mk{p}^-)$, besides $\mk{h}$ is $\tau$-invariant, therefore $\mk{p}^+$ is regular as well, so we obtain \eqref{wcb8}.

From  \eqref{wcb8} and $\tau(\mk{g}(\alpha))=\mk{g}(-\alpha)$, it follows  that
$$
\mk{p}^- = \tau(\mk{p}^+\cap \mk{h}) + \sum_{\alpha \in \Delta_{\mk{p}}^+} \mk{g}(-\alpha).
$$
We see that $\mk{p}^-$ is regular in $\mk{g}$ and $\mk{h}\cap \mk{p}^- = \tau(\mk{h}\cap \mk{p}^+) $, therefore we have proved \eqref{wcb9}.

Since $ \mk{p} = \mk{p}^+ + \mk{p}^-$, $\mk{p}^+ \cap \mk{p}^- = \{0\}$, then
\ben
\Delta_{\mk{p}}^+ \cap (- \Delta_{\mk{p}}^+) = \{0\} \qquad \qquad  \mk{h} \cap \mk{p} = (\mk{h}\cap \mk{p}^+) \oplus (\mk{h}\cap \mk{p}^-).
\een
 Since we have a direct vector sum $\mk{g} = \mk{k} + \mk{p}^+ + \mk{p}^-$  and Lemma \eqref{lemma for mk{k}} then obviously $ \Delta \setminus \Delta_{\mk{k}}= \Delta_{\mk{p}}^+ \cup (-\Delta_{\mk{p}}^+)$.
\epr

\subsection{The set \texorpdfstring{$\Delta_{\mk{p}}^+ $}{\space}} \label{the system Gamma}

We recall that the definitions of $I$ and $J$ are in \eqref{wcb5}.
 Integrability of the
complex structure $I_M$ on $M$ implies (see \cite{Kobayashi69}, ch. X) that for $X, Y \in \mk{p}_u$ we have
$$
  [I X , I Y]_{\mk{p}_u} - [X,Y]_{\mk{p}_u} - I([X, I Y]_{\mk{p}_u}) - I([I X,
Y]_{\mk{p}_u})=0.
$$
 After complexification  we can write
$$
 \forall X, Y \in
\mk{p}  \ \ [I X , I Y]_{\mk{p}} - [X,Y]_{\mk{p}} - I([X, I
Y]_{\mk{p}}) - I([I X, Y]_{\mk{p}})=0.
$$
The integrability of I implies
\be \label{wcb6}
 \forall X,Y \in \mk{p}^\pm \qquad [X,Y]_{\mk{p}} \in \mk{p}^\pm.
 \ee

 The elements in ${\bf U}$
act as hypercomplex diffeomorphisms of $M$, whence for $X \in \mk{k}$ we have
\be \label{afi1}
 I \circ ad(X) = ad(X) \circ I, \quad J \circ ad(X) = ad(X) \circ J.
\ee
We recall that \eqref{properties of frak{m}} is valid, therefore the equations above make sense.
Relation \eqref{afi1} implies
\be \label{afi2}
[\mk{k},\mk{p}^+] \subset \mk{p}^+, \qquad [\mk{k}, \mk{p}^-] \subset \mk{p}^-.
\ee
Indeed, let $X \in \mk{k}$, $Y \in
\mk{p}^+$ then
$$
[X, Y]= -\ri ad(X)(I(Y)) = -\ri I(ad(X)(Y)) = -\ri I([X,Y]) \Rightarrow I([X,Y]) =  \ri [X,Y]
$$
 hence $[X, Y] \in \mk{p}^+$. The proof for $\mk{p}^-$ is the same.

From \eqref{afi2}, \eqref{wcb6} and since $\mk{k}$ is a subalgebra of $\mk{g}$ it follows that $\mk{k}+\mk{p}^+$ is a subalgebra of $\mk{g}$. Furthermore, both $\mk{k}$ and $\mk{p}^+$ are regular, therefore  $\mk{k}+\mk{p}^+$ is regular and obviously by \eqref{wcb8} and \eqref{frak{k} is regular} we have
\begin{gather}\label{sred}
\mk{k}+\mk{p}^+ = (\mk{k}+\mk{p}^+)\cap \mk{h} + \sum_{\alpha \in \Delta_{\mk{k}} \cup \Delta_{\mk{p}}^+} \mk{g}(\alpha).
\end{gather}

From   \eqref{wcb7} we see,  that $\Delta =\Delta_{\mk{k}} \cup \Delta_{\mk{p}}^+ \cup(-\Delta_{\mk{p}}^+)=(\Delta_{\mk{k}} \cup \Delta_{\mk{p}}^+) \cup \left ( -(\Delta_{\mk{k}} \cup \Delta_{\mk{p}}^+) \right )$.  From  (\cite{Bourbaki75}, ChVII, \S 3, n.4) we see that \eqref{sred}, $\Delta =(\Delta_{\mk{k}} \cup \Delta_{\mk{p}}^+) \cup \left ( -(\Delta_{\mk{k}} \cup \Delta_{\mk{p}}^+) \right )$ and the fact that $\mk{k}+\mk{p}^+$ is a subalgebra imply that  there exists a basis $\Pi$ of the root system $\Delta$, such that  the set of positive roots $\Delta^+$, for this basis $\Pi$, satisfies $\Delta^+ \subset \Delta_{\mk{k}} \cup \Delta_{\mk{p}}^+$. Since $\Delta^+ \cup (-\Delta^+) = \Delta $, then for any $\alpha \in \Delta_{\mk{p}}^+$ we have $\alpha \in \Delta^+$ or $-\alpha \in \Delta^+$, but the relation $-\alpha \in \Delta^+$   contradicts  $\Delta^+ \subset \Delta_{\mk{k}} \cup \Delta_{\mk{p}}^+$ and \eqref{wcb7}. Thus we proved the existence of such a basis $\Pi$, that the corresponding set of positive roots satisfies
\be \label{Delta_{mk{m}}^+=}
\Delta_{\mk{p}}^+ \subset \Delta^+.
\ee

\begin{remark} From now on this subset of positive roots  $\Delta^+$ will be fixed throughout section \ref{nc0}.
It is useful to denote also \be \Delta_{\mk{k}}^+ = \Delta^+ \cap \Delta_{\mk{k}} .\ee
Then from \eqref{wcb7} it follows
\be \Delta^+ = \Delta_{\mk{k}}^+ \cup \Delta_{\mk{p}}^+ \ \ \mbox{disjoint union}.\ee

In the sequel we shall use also the following notations, related to the stem $(\Gamma,\prec)$ of $\Delta^+$:
\begin{gather}
\label{prop of Delta+ def of B_{gamma}^{frak{m}}} \Phi_{\zeta}^{\mk{p}} = \{ \beta \in \Delta_{\mk{p}}^+ \vert \  \zeta- \beta \in \Delta_{\mk{p}}^+ \} \qquad \zeta \in \Delta_{\mk{p}}^+\\
\label{def of Gamma_{frak{m}},...} \Gamma_{\mk{p}}= \Gamma \cap \Delta_{\mk{p}}^+, \qquad
 \Gamma_{\mk{k}}=\Gamma \cap \Delta_{\mk{k}}^+ =\Gamma \setminus \Gamma_{\mk{p}}, \qquad \mk{w}_{\mk{p}} = span_\RR\{W_\gamma\vert \gamma \in \Gamma_{\mk{p}}\}.
   \end{gather}
\end{remark}

\subsection{Decomposition of the Cartan subalgebra: \texorpdfstring{$\mk{h}_u$, $\mk{h}_{\mk{k}_u}$, $\mk{h}_{\mk{p}_u}$, $\mk{n}_{\mk{p}_u}$}{\space}}  We study the structure of $\Delta_{\mk{k}}^+$, $\Delta_{\mk{p}}^+$ to show that $\mk{k}$ is a $\Gamma_{\mk{k}}$-stemmed subalgebra. We also show that $J$ is of the type \eqref{afi3}.
First we introduce some more notations.
Let us denote
\begin{gather}\label{csj2}
\mk{h}_u = \mk{u} \cap \mk{h} =  \mk{c}_u \oplus (\mk{h}_s \cap \mk{u}_s),\qquad
  \mk{h}_{\mk{k}_u} = \mk{k}_u \cap \mk{h},\quad
\mk{h}_{\mk{p}_u} =  \mk{p}_u \cap \mk{h}.
\end{gather}
Since $\mk{u} = \mk{k}_u \oplus \mk{p}_u$, then obviously
\be \label{hhh1} \mk{h}_u = \mk{h}_{\mk{k}_u} \oplus \mk{h}_{\mk{p}_u}.
 \ee

Furthermore, let us denote
\begin{gather} \label{frak{n}_{frak{m}}}
\mk{n}_{\mk{p}_u}
=   span_\RR\{(E_{\alpha} -
E_{-\alpha}), \ri (E_{\alpha} + E_{-\alpha}); \alpha \in \Delta_{\mk{p}}
\} \\
\mk{n}_{\mk{p}}^+ = \sum_{\alpha \in \Delta_{\mk{p}}^+} \mk{g}(\alpha) \qquad \mk{n}_{\mk{p}}^- = \sum_{\alpha \in \Delta_{\mk{p}}^+} \mk{g}(-\alpha) \qquad \mk{n}_{\mk{p}} = \mk{n}_{\mk{p}}^+ +\mk{n}_{\mk{p}}^-\end{gather} then  obviously

\begin{gather} \label{hhh2} \mk{p}_u = \mk{h}_{\mk{p}_u} \oplus \mk{n}_{\mk{p}_u},\quad
\mk{k}_u = \mk{h}_{\mk{k}_u} \oplus span_{\RR}\{(E_{\alpha}-
E_{-\alpha}), \ri (E_{\alpha} + E_{-\alpha} ); \alpha \in  \Delta_{\mk{k}}\}.
 \end{gather}

 With these notations the following lemmas are valid
\begin{lemma} \label{mkm9} If $\alpha \in \Delta$ and $\alpha\sco{\mk{h}_{\mk{k}_u}} = 0$, then $W_{\alpha} \in
\mk{p}_u$ (recall that $W_{\alpha}=\frac{i}{2} H_{\alpha}$).
\end{lemma}
\begin{proof} Let $X \in \mk{h}_{\mk{k}_u}$. Let us decompose $X = X_{\mk{c}} + X_s$, $X_{\mk{c}} \in \mk{c}_u$, $X_s \in \mk{u}_s \cap \mk{h}_s$ in accordance with  the decomposition $\mk{h}_u
= \mk{c}_u \oplus (\mk{h}_s \cap \mk{u}_s)$. Since $\alpha(X)=0$ and $\alpha$ vanishes on $\mk{c}_u$, then $\alpha(X_s)=0$. Hence  $Kill(h_{\alpha},X_s)= \alpha(X_s) = 0$ and
 $\scal{W_{\alpha},X}=\scal{W_{\alpha},X_{\mk{c}} + X_s} = \scal{W_{\alpha}, X_s}=Kill(W_{\alpha},X_s)=0$.

Using \eqref{hhh2} and the obvious fact that $\scal{W_\alpha,\mk{n}^\pm} = 0$ we see that
$\scal{\mk{k}_\mk{u}, W_{\alpha}}=0$, hence  $W_{\alpha} \in
\mk{p}_u$.\end{proof}

\begin{lemma} \label{lemma if... then alpha(X)=0} Let $X \in \mk{h}_{\mk{p}_u}$ and $\alpha \in
\Delta_{\mk{k}}$ then $\alpha(X)=0$.
\end{lemma}
\begin{proof}

Since $E_{\alpha} \in \mk{k}$ and $-\Delta_{\mk{k}} = \Delta_{\mk{k}}$ then  $H_{\alpha} = [E_{\alpha}, E_{-\alpha}] \in \mk{k}\cap \mk{h}$. Therefore $\ri H_{\alpha} \in \mk{h}_{\mk{k}_u}$.
 From the definition of $\mk{p}_u$ we have $\scal{\ri
H_{\alpha}, X} = 0$.  Let us decompose $X = X_{\mk{c}} + X_s$, $X_{\mk{c}} \in \mk{c}_u$, $X_s \in \mk{u}_s \cap \mk{h}_s$. Since  $\alpha$ vanishes on $\mk{c}_u \subset \mk{c}$, then $\alpha(X_s)=\alpha(X)$.  Since $\scal{\ri
H_{\alpha}, X} = 0$ then $\scal{\ri H_{\alpha}, X_s} = \alpha(X_s)= 0$, therefore $\alpha(X)=0$.
\end{proof}

\subsection{Preliminary lemmas about \texorpdfstring{$\mk{k}$, $\mk{p}^+$, $\mk{p}^-$}{\space}}

Here we prove that $\Gamma_{\mk{k}}$ is a substem of $(\Gamma, \prec)$.

\begin{lemma} \label{pdu2} If $\gamma \in \Gamma_{\mk{k}}$ and $\widetilde{\gamma} \succeq \gamma$, then $
\Phi_{\widetilde{\gamma}} \cup \{\widetilde{\gamma}, - \widetilde{\gamma} \} \subset  \Delta_{\mk{k}}$.
In particular $\Phi_{\widetilde{\gamma}}^+ \cup \{\widetilde{\gamma}\} \subset \Delta_{\mk{k}}^+$.
\end{lemma}
\begin{proof} Since $-\Delta_{\mk{k}} = \Delta_{\mk{k}}$ then $E_{-\gamma} \in \mk{k}$. Let $\alpha \in \Phi_{\gamma}^+$. Let us assume that $\alpha \not \in
\Delta_{\mk{k}}$. Then $\alpha \in
\Delta^+ \setminus \Delta_{\mk{k}}$, therefore  $\alpha \in
\Delta_{\mk{p}}^+$ and $E_{\alpha} \in \mk{p}^+$. By \eqref{afi2}
$[E_{-\gamma}, E_{\alpha}] \in \mk{p}^+$, but $[E_{-\gamma},
E_{\alpha}] = N_{-\gamma, \alpha} E_{\alpha - \gamma}$ and $\alpha
- \gamma \not \in \Delta^+$ (see \eqref{pdep 2}) , which contradicts $\Delta_{\mk{p}}^+ \subset
\Delta^+$ (see  \ref{Delta_{mk{m}}^+=}).

Thus, we see that if $\alpha \in \Phi_{\gamma} \cup  \{-\gamma, \gamma\}$, then $E_{\alpha} \in \mk{k}.$ Now
the lemma follows from \eqref{pdep 7} and the
fact that $\mk{k}$ is a subalgebra.
\end{proof}

\begin{coro}\label{coro for the property of Gamma_{frak{k}}} The subset $\Gamma_{\mk{k}} \subset \Gamma$ is a substem of $(\Gamma,\prec)$.
\end{coro}

\begin{lemma} \label{help lemma for H_gamma not in frak{k}} Let $\gamma \in \Gamma_{\mk{p}}$   and  $H_{\gamma} \in \mk{k}$. Let
\be J(E_{\gamma}) = H + \sum_{\alpha \in \Delta_{\mk{p}}^+ } c_{\alpha} E_{-\alpha} \qquad H \in \mk{h} \cap \mk{p}^-, \ \ c_{\alpha} \in \CC.\ee Then  $H=0$ and
\be \label{help lemma for H_gamma not in frak{k}1} \forall \alpha \in \Delta_{\mk{p}}^+  \qquad   \alpha(H_{\gamma}) \geq 0 \Rightarrow c_{\alpha} = 0. \ee
\end{lemma}
\begin{proof}

Since $H_{\gamma} \in \mk{k}$ then by \eqref{afi1} it follows
\be J([H_{\gamma}, E_{\gamma}]) = [H_{\gamma}, J(E_{\gamma})] \qquad \Rightarrow \qquad \gamma(H_{\gamma}) J(E_{\gamma}) =\sum_{\alpha \in \Delta_{\mk{p}}^+} c_{\alpha} [H_{\gamma},E_{-\alpha}] \nonumber  \ee
hence \be  \gamma(H_{\gamma}) \left ( H + \sum_{\alpha \in \Delta_{\mk{p}}^+ } c_{\alpha} E_{-\alpha} \right ) = -\sum_{\alpha \in \Delta_{\mk{p}}^+} c_{\alpha} \alpha(H_{\gamma})E_{-\alpha}  \nonumber \ee
and  we obtain
\be \gamma(H_{\gamma}) H + \sum_{\alpha \in \Delta_{\mk{p}}^+ } c_{\alpha}\ (\gamma + \alpha)(H_{\gamma})  \ E_{-\alpha}  = 0.\ee
Now we have $\gamma(H_{\gamma})=2$ and therefore  $(\gamma + \alpha)(H_{\gamma}) \neq 0$, if $\alpha(H_\gamma) \geq 0$, therefore the last equation above implies $H=0$ and \eqref{help lemma for H_gamma not in frak{k}1}.
\end{proof}

\begin{lemma} \label{pdu3} Let $\gamma \in \Gamma_{\mk{p}}$ be a maximal root.
Then $H_{\gamma} \not \in \mk{k}$.
\end{lemma}
\begin{proof} Let us recall that $J(\mk{p}^+) =  \mk{p}^-$, besides by $\gamma \in \Gamma_{\mk{p}}$  we have  $E_{\gamma} \in \mk{p}^+$ (see \eqref{def of Gamma_{frak{m}},...} and \eqref{wcb8}), therefore
\ben J(E_{\gamma}) = H + \sum_{\alpha \in \Delta_{\mk{p}}^+ } c_{\alpha} E_{-\alpha} \qquad H \in \mk{h}\cap \mk{p}^-, c_{\alpha} \in \CC.\een
Let us assume that $H_{\gamma} \in \mk{k}$.
Now, since $\gamma$ is a maximal root, we have  $\forall \alpha \in \Delta_{\mk{p}}^+$, $ \alpha(H_{\gamma}) \geq 0$, therefore Lemma \ref{help lemma for H_gamma not in frak{k}} implies
\be \left ( H=0, \quad \forall \alpha \in \Delta_{\mk{p}}^+ \ \ \ c_{\alpha} =0 \right ) \ \ \Rightarrow \ \ J(E_{\gamma})=0\ee
which contradicts $J^2 = - Id$ .
\end{proof}
\begin{lemma} \label{lemma about J(E_{gamma}) in  frak{h} then H_gamma}Let $\gamma \in \Delta_{\mk{p}}^+$ be such that $J(E_{\gamma}) \in  \mk{h}$, then $W_{\gamma} \in \mk{p}_u$, i. e. $H_{\gamma} \in \mk{p}$.
\end{lemma}
\begin{proof} By \eqref{afi1} and $J(E_{\gamma}) \in \mk{h}$ it follows  that
\be \forall H \in \mk{h} \cap \mk{k} \ \ \ \  J([H,E_{\gamma}])= [H, J(E_{\gamma})]=0. \nonumber \ee
Therefore $ \forall H \in \mk{h} \cap \mk{k} \ \ \  \gamma(H) J(E_{\gamma}) = 0 $, hence
$  \forall H \in \mk{h} \cap \mk{k} \ \ \  \gamma(H) = 0.  $
The lemma follows from Lemma \ref{mkm9}.
\end{proof}
\begin{lemma} \label{lemma about J(E_{gamma}) in  frak{h} then B_{gamma}} Let $\gamma \in \Gamma_{\mk{p}}$ be such that $J(E_{\gamma}) \in  \mk{h}$, then $\Phi_{\gamma}^+ \subset \Delta_{\mk{p}}^+$.
\end{lemma}
\begin{proof} Let $\alpha \in \Phi_{\gamma}^+$. Let us assume that $\alpha \not \in \Delta_{\mk{p}}^+$. Then $\alpha \in \Delta_{\mk{k}}^+$ and $E_{\alpha}, E_{- \alpha} \in \mk{k}$, besides $E_{\gamma} \in \mk{p}^+ $ (since $\gamma \in \Gamma_{\mk{p}} = \Gamma \cap \Delta_{\mk{p}}^+$). Since $[\mk{k}, \mk{p}^+] \subset \mk{p}^+$, then (we recall that $\gamma - \alpha \in \Delta_{+}$)
\be [E_{\gamma}, E_{-\alpha}]= N_{\gamma, -\alpha} E_{\gamma - \alpha} \in \mk{p}^+  \ \ \Rightarrow \ \ E_{\gamma - \alpha} \in \mk{p}^+.\ee
There exist $H \in \mk{h}\cap \mk{p}^-$ and constants $c_{\beta} \in \CC$ for $\beta \in \Delta_{\mk{p}}^+$, such that
\be J(E_{\gamma - \alpha}) = H + \sum_{\beta \in \Delta_{\mk{p}}^+} c_{\beta} E_{-\beta}.\ee
From \eqref{afi1} it follows
\be [E_{\alpha}, J(E_{\gamma - \alpha})] = J([E_{\alpha}, E_{\gamma-\alpha}]) \ \ \Rightarrow \ \ [E_{\alpha}, J(E_{\gamma - \alpha})] = N_{\alpha, \gamma - \alpha} J(E_{\gamma}) \in \mk{h}\setminus \{0\} \nonumber\ee
therefore
\be [E_{\alpha}, H + \sum_{\beta \in \Delta_{\mk{p}}^+} c_{\beta} E_{-\beta}] = -\alpha(H) E_{\alpha} + \sum_{\beta \in \Delta_{\mk{p}}^+} c_{\beta} [E_{\alpha}, E_{-\beta}] = N_{\alpha, \gamma - \alpha} J(E_{\gamma})\in \mk{h}\setminus \{0\}. \nonumber \ee
However, for $\beta \in \Delta_{\mk{p}}^+$ we have $\alpha \pm  \beta \neq 0$ (since $\alpha \in \Delta_{\mk{k}}$), therefore
\ben -\alpha(H) E_{\alpha} + \sum_{\beta \in \Delta_{\mk{p}}^+} c_{\beta} [E_{\alpha}, E_{-\beta}] \not \in \mk{h}. \nonumber  \een
 This contradicts our assumption $\alpha \not \in \Delta_{\mk{p}}^+$ whence the lemma follows.
\end{proof}

\subsection{The integrability condition}

 \label{The integrability condition}
Throughout section \ref{nc0} the notations $\mk{p}_u$, $I$, $J$, $\mk{p}^+$, $\mk{p}^-$  have the same meaning (see    Section  \ref{some notations})  as well as $\mk{h}_{\mk{p}_u}$ (see \eqref{csj2}).
We recall that $\mk{k}_u \oplus \mk{p}_u = \mk{u}$  and that the reductivity condition $[\mk{k}_u, \mk{p}_u] \subset \mk{p}_u$ is satisfied.  In this section we shall vary $\mk{h}_{\mk{p}_u}$ to another $\mk{q}$, with corresponding $\wt{\mk{p}}_u=\mk{q}+\mk{n}_{\mk{p}_u}$, observing  the condition $\mk{k}_u \oplus \wt{\mk{p}}_u = \mk{u}$, but we discard the reductivity condition.

\subsubsection{Construction of $(\wt{\mk{p}}, \wt{\mk{p}}^+, \wt{J}, \wt{\xi}, \wt{\eta}, \wt{b})$ from $\mk{q}$}

 Let    $\mk{q}$ be any subspce of $\mk{h}_u$, such that
 \be \label{wtj2}
  \mk{h}_u = \mk{q}
\oplus \mk{h}_{\mk{k}_u} \ee
then we denote $\wt{\mk{p}}_u = \mk{q} + \mk{n}_{\mk{p}_u}$.

We denote by $\wt{\mk{p}}$ the complexification of  $\wt{\mk{p}}_u$. Obviously
\be \label{widetilde{frak{m}}}
 \wt{\mk{p}} = (\mk{h} \cap \wt{\mk{p}}) + \mk{n}_{\mk{p}}.
\ee

 From \eqref{hhh2} we see that  $\wt{\mk{p}}_u$ satisfies $\mk{u} = \wt{\mk{p}}_u \oplus \mk{k}_u$, therefore we can  define  complex structures on $\wt{\mk{p}}_u$, using the complex
structures $I_M(o) : T_o(M)  \rightarrow T_o(M)$,  $J_M(o) :
T_o(M)  \rightarrow T_o(M)$ on
$T_o(M)$, as follows \begin{gather}\wt{I}: \wt{\mk{p}}_u \rightarrow \wt{\mk{p}}_u \qquad \wt{J}: \wt{\mk{p}}_u \rightarrow \wt{\mk{p}}_u \nonumber\\ \wt{I} = ((\rd
\pi )_{\vert \wt{\mk{p}}_u} )^{-1}\circ I_M(o) \circ (\rd \pi
)_{\vert \wt{\mk{p}}_u}, \ \ \wt{J} = ((\rd \pi )_{\vert \wt{\mk{p}}_u}
)^{-1}\circ J_M(o) \circ (\rd \pi )_{\vert \wt{\mk{p}}_u}. \nonumber
\end{gather}

 We shall denote the complexifications of $\wt{I}$  and  $\wt{J}$  by the same letters $\wt{I} : \wt{\mk{p}} \rightarrow \wt{\mk{p}}$, $\wt{J} : \wt{\mk{p}} \rightarrow \wt{\mk{p}}$. Let
$\wt{\mk{p}}^+$ and $\wt{\mk{p}}^-$ be the $\ri$ and
$(-\ri)$-eigenspaces of $\wt{I}$, respectively.

The following lemma relates $\wt{J}$ to $J$.
\begin{lemma}\label{wtj1}
  Let $X \in \mk{n}_{\mk{p}_u}$ and let $J(X) =
H+Y$, $\widetilde{J}(X) = \widetilde{H}+\widetilde{Y}$, where $H
\in \mk{h}_{\mk{p}_u}, \widetilde{H} \in
\mk{q} $ and $Y,\widetilde{Y} \in
\mk{n}_{\mk{p}_u}$ then $ Y =  \widetilde{Y}$. Furthermore, $
H = 0 \iff \widetilde{H} = 0 $.

The same statement is valid if we replace $J$ with $I$
 and   $\widetilde{J}$  with $\widetilde{I}$.
\end{lemma}
\begin{proof} By the definitions of $J$ and $\widetilde{J}$, we
have \be  ((\rd \pi )_{\vert \wt{\mk{p}}_u} )^{-1}\circ
J_M(o) \circ (\rd \pi )_{\vert \wt{\mk{p}}_u} (X) =
\widetilde{H} + \widetilde{Y}  \nonumber \\
((\rd \pi )_{\vert \mk{p}_u} )^{-1}\circ J_M(o) \circ (\rd \pi
)_{\vert \mk{p}_u} (X) = H + Y\nonumber \ee hence \ben \rd \pi
(\widetilde{H} + \widetilde{Y} ) = \rd \pi (H + Y) = J_M(o) ( \rd
\pi  (X) ).\een On the other hand $\ker (\rd \pi) = \mk{k}_u$,
therefore  $ (H-\widetilde{H}) + (Y-\widetilde{Y}) \in
\mk{k}_u$,  hence $ Y-\widetilde{Y} \in (\mk{k}_u +
\mk{h}_u) \cap \mk{n}_{\mk{p}_u}$. Since $(\mk{k}_u +
\mk{h}_u) \cap \mk{n}_{\mk{p}_u} = \{0\}$, we see that
$Y-\widetilde{Y} = 0 $ and \be (H-\widetilde{H}) \in \mk{k}_u.
\nonumber  \ee Now if $H=0$ then $\widetilde{H} \in \mk{k}_u
\cap \mk{q}= \{0\}$, hence
$\widetilde{H} = 0$.
 On the other hand if $\widetilde{H}=0$ then $H \in
\mk{k}_u \cap \mk{h}_{\mk{p}_u}= \{0\}$, hence $H = 0$.
\end{proof}

\begin{lemma} \label{wid0} The complex structure $\widetilde{I}$ satisfies the following conditions
\begin{gather}
\label{wid1}  \wt{I}(\mk{n}_{\mk{p}_u}) =\mk{n}_{\mk{p}_u}
 ,\quad
\widetilde{I}(\mk{q})
=\mk{q} \\
\label{wid03}
\wt{\mk{p}}^+ = (\wt{\mk{p}}^+ \cap \mk{h}) +
\mk{n}^+_{\mk{p}} \qquad  \quad \wt{\mk{p}}^- = (\wt{\mk{p}}^- \cap \mk{h}) +
\mk{n}^-_{\mk{p}}=\tau(\wt{\mk{p}}^+ \cap \mk{h}) +
\mk{n}^-_{\mk{p}}.
\end{gather}
\end{lemma}
\begin{proof} The equalities   $I(E_{\alpha} - E_{-\alpha}) = \ri (E_{\alpha} +
E_{-\alpha})$, $I(\ri(E_{\alpha}+ E_{-\alpha})) = -(E_{\alpha} -
E_{-\alpha})$  are   valid for $\alpha \in \Delta_{\mk{p}}$  and from
\eqref{frak{n}_{frak{m}}} we see that
$I(\mk{n}_{\mk{p}_u}) = \mk{n}_{\mk{p}_u}$.

Hence, by the previous lemma
$ \widetilde{I}(\mk{n}_{\mk{p}_u}) =I(X) = \mk{n}_{\mk{p}_u}. $

 Let us recall  the fact that $I$ restricted to $\mk{h} \cap \mk{p}^+$ is $ \ri$ and restricted to
$\mk{h} \cap \mk{p}^-$ is $-\ri$, therefore
$$
I(\mk{h}\cap \mk{p}) = \mk{h} \cap \mk{p} \Rightarrow
I(\mk{h}_{\mk{p}_u}) = \mk{h}_{\mk{p}_u}.
$$
 Now we prove the second statement in \eqref{wid1}. Let us take $H \in
\mk{q}$, $H \neq 0$. Let us decompose
$\widetilde{I}(H)=H' + B$, where $H'\in
\mk{q}$, $B \in \mk{n}_{\mk{p}_u}$.
We shall show that $B=0$. The definition of $\widetilde{I}$
implies $((\rd \pi )_{\vert \wt{\mk{p}}_u} )^{-1}\circ
I_M(o) \circ (\rd \pi )_{\vert \wt{\mk{p}}_u}(H) = H' +
B$, therefore
\be \label{help equality} (\rd \pi )(H' + B)=I_M(o)(
\rd \pi (H)).
\ee
Now since $H, H' \in \mk{h}_u$ and since
$\mk{h}_u = \mk{h}_{\mk{p}_u} \oplus \mk{h}_{\mk{k}_u}$,
then we may decompose $H$ and $H'$ as follows
\begin{gather*}
H = H_1 + H_2, \
\ H' = H_1'+ H_2' \qquad H_1, H_1' \in \mk{h}_{\mk{k}_u}, \ \
H_2, H_2' \in \mk{h}_{\mk{p}_u}.
\end{gather*}
The equality
\eqref{help equality} can be rewritten as follows
\be \label{help
equality1} (\rd \pi )(H_2' + B)=I_M(o)( \rd \pi (H_2))
\ee
where we
used $\ker (\rd \pi) = \mk{k}_u$ and $\mk{h}_{\mk{k}_u}
\subset \mk{k}_u$. Since $H_2' + B \in \mk{h}_{\mk{p}_u}+
\mk{n}_{\mk{p}_u} = \mk{p}_u$ and $H_2 \in \mk{p}_u$, then
\be \label{help equality2} (H_2' + B)=((\rd \pi )_{\vert
\mk{p}_u} )^{-1}\circ I_M(o) \circ (\rd \pi )_{\vert \mk{p}_u}(H_2)
= I(H_2).
 \ee
 We have already shown that $I(\mk{h}_{\mk{p}_u})
= \mk{h}_{\mk{p}_u}$, which implies that $B=0$ and
\eqref{wid1} is proved.

From \eqref{wid1} and Lemma \ref{wtj1} it
follows that for each $X \in \mk{n}_{\mk{p}}$ we have $\widetilde{I}(X) = I(X) =  \ri X $,
therefore $$\mk{n}_{\mk{p}}^+ \subset \wt{\mk{p}}^+
\qquad \mk{n}_{\mk{p}}^- \subset \wt{\mk{p}}^-.
$$ Let us take any element $X \in
 \wt{\mk{p}}^+ $. By \eqref{widetilde{frak{m}}} we may decompose it as follows  $X=A + H + B$, $A \in \mk{n}_{\mk{p}}^-$,
  $H \in \mk{h}\cap \wt{\mk{p}}$, $B \in
  \mk{n}_{\mk{p}}^+$. Since $A \in \mk{n}_{\mk{p}}^-\subset \wt{\mk{p}}^-
  $  and $ B \in \mk{n}_{\mk{p}}^+\subset \wt{\mk{p}}^+$, then
\begin{gather*}
\widetilde{I}( A + H + B) =  -\ri A + \widetilde{I}( H )+ \ri B.
\end{gather*}
On the other hand since $A + H + B \in
\wt{\mk{p}}^+$, then $\widetilde{I}( A + H + B) = \ri
( A +  H + B ) $ hence
\begin{gather*}   -\ri A + \widetilde{I}( H )+ \ri B =
\ri  A + \ri  H + \ri B
\end{gather*}
this equality together with
$\wt{I}(\mk{h}\cap \wt{\mk{p}})
=\mk{h}\cap \wt{\mk{p}} $ implies that $A=0$, therefore
$X = H + B$. On the other hand $B \in \mk{n}_{\mk{p}^+}
\subset \wt{\mk{p}}^+ $ therefore  $X-B = H \in
\wt{\mk{p}}^+ \cap \mk{h}$. And we proved $\wt{\mk{p}}^+ = (\wt{\mk{p}}^+ \cap \mk{h}) +
\mk{n}^+_{\mk{p}} $. The other equality in \eqref{wid03} follows from $\wt{I} \circ \tau = \tau \circ \wt{I}$.
\end{proof}
Let us  choose  a basis $\{ U_1, U_2, \dots, U_m \}$ of
$\wt{\mk{p}}^+ \cap \mk{h}$   , then   $\{ V_1=\tau(U_1), V_2=\tau(U_2), \dots, V_m=\tau(U_m)\}$  is a basis of $\wt{\mk{p}}^- \cap \mk{h}$. Then by \eqref{wid03} the subspaces $\wt{\mk{p}}^+$, $\wt{\mk{p}}^-$ decompose as follows
\ben  \nonumber
\wt{\mk{p}}^+ = span_\CC \{U_i\}_{i=1}^m + \sum_{\alpha \in \Delta_{\mk{p}}^+} \mk{g}(\alpha) \qquad  \wt{\mk{p}}^- = span_\CC \{V_i\}_{i=1}^m + \sum_{\alpha \in \Delta_{\mk{p}}^+} \mk{g}(-\alpha).
\een
In the same way as in \cite{DimTsan10}( Definition 1.12, Proposition 1.13) one can prove, using  the conditions $\widetilde{J}(\wt{\mk{p}}^+) = \wt{\mk{p}}^-$, $\widetilde{J}^2 = - Id$,  $\wt{J} \circ \tau = \tau \circ \wt{J}$,   the following lemma

\begin{lemma} \label{lemma about a,b,xi,eta} There  exist functions \be \wt{\xi} : \{1,2,\dots,m\}
\times \Delta_{\mk{p}}^+ \rightarrow \CC \qquad \wt{{\bf a}} :
\Delta_{\mk{p}}^+ \times \Delta_{\mk{p}}^+ \rightarrow \CC
\nonumber \\ \wt{{\bf b}}: \{1,2,\dots, m \} \times \{1,2,\dots,m\}
\rightarrow \CC \qquad \wt{\eta}:\Delta_{\mk{p}}^+ \times
\{1,2,\dots,m\} \rightarrow \CC \nonumber \ee  such that
\be
\widetilde{J}(E_{\alpha}) & = & \sum_{t=1}^m \wt{\xi}_{t \alpha} V_t +
\sum_{\beta \in \Delta_{\mk{p}}^+} \wt{a}_{\beta \alpha} E_{-\beta}   \qquad \alpha \in \Delta_{\mk{p}}^+ \nonumber\\[-2mm] \label{def of a,b,xi,eta} \\[-2mm]
  \widetilde{J}(U_q) & =& \sum_{t=1}^m
\wt{b}_{t q} V_t + \sum_{\beta \in \Delta_{\mk{p}}^+}
\wt{\eta}_{\beta q} E_{-\beta} \qquad q\in \{ 1, \dots ,m \}. \nonumber
 \ee

The functions $\wt{{\bf a}}, \wt{{\bf b}}, \wt{\xi}, \wt{\eta}$ satisfy
\begin{gather}
  \ol{\wt{{\bf b}}}\wt{\xi} - \overline{\wt{\xi}}  \wt{{\bf a}}  = 0,  \quad  \ol{\wt{\eta}} \xi  - \ol{ \wt{{\bf a}}} { \wt{{\bf a}}}  = \rI_{\Delta_{\mk{p}}^+},
\quad  \ol{ \wt{{\bf a}}} \eta -  \ol{\wt{\eta} }{ \wt{{\bf b}}}  = 0, \quad
\ol{\wt{\xi}}  \wt{\eta} -  \ol{ \wt{{\bf b}}} { \wt{{\bf b}}}= \rI_m. \nonumber
 \end{gather}\end{lemma}
\begin{remark} \label{remark for a,b,xi,eta and a,wt{b},wt{xi},wt{eta}}
The  functions  $\wt{{\bf a}}, \wt{{\bf b}}, \wt{\xi}, \wt{\eta}$  depend on the choice of the subspace $\mk{q}$, which in turn determines $\widetilde{I}, \widetilde{J}$. For the case $\mk{q}=\mk{h}_{\mk{p}_u}$, which gives the complex structures $I$, $J$, fixed in subsection \ref{some notations}, we shall denote these function  by  ${\bf a}, {\bf b}, \xi,\eta $ throughout section \ref{nc0}.
\end{remark}
 From Lemma \ref{wtj1} it follows
\begin{lemma} \label{the matrix a is fixed} The matrix $\wt{{\bf a}}$, assigned to $\widetilde{J}$ in accordance with Lemma \ref{lemma about a,b,xi,eta}, does not change as $\mk{q}$ varies constrained to \eqref{wtj2}, i.e. $\wt{{\bf a}} = {\bf a}$ for any choice of $\mk{q}$.
\end{lemma}
This lemma allows us to write ${\bf a}$ instead of $\wt{{\bf a}}$, without paying attention to the choice of  $\mk{q}$ and this matrix ${\bf a}$ is  fixed throughout section \ref{nc0}.

\begin{coro} \label{J(E_{alpha}) in frak{h}}  Let $\alpha \in \Delta_{\mk{p}}^+$ then
\be \widetilde{J}(E_{\alpha}) \in \mk{h} \qquad  \iff \qquad \forall \zeta \in \Delta_{\mk{p}}^+ \ \ a_{\zeta \alpha} = 0. \nonumber\ee
\end{coro}

\begin{remark} \label{remark about extention of a}

It is convenient to extend the definitions of the functions ${\bf a}$,  $\wt{\eta}$.
Let $\alpha, \beta \in \mk{h}^*,\quad k\in \NN$. If $\alpha \not \in \Delta_{\mk{p}}^+$ or $\beta \not \in \Delta_{\mk{p}}^+$ , then we set
$
a_{\alpha,\beta} =\wt{\eta}_{\alpha,k} = 0.
$
\end{remark}

\begin{df}\label{def Cartan variation of mk{m}_0} In this subsection we started from any subspace $\mk{q} \subset \mk{h}_u$, satisfying \eqref{wtj2} and constructed  a 6-uple $(\wt{\mk{p}}, \wt{\mk{p}}^+, \wt{J}, \wt{\xi}, \wt{\eta}, \wt{{\bf b}})$, depending on $\mk{q}$ (we recall that the matrix ${\bf a}$ does not depend on $\mk{q}$). For the sake of brevity we shall call such a 6-uple $(\wt{\mk{p}}, \wt{\mk{p}}^+, \wt{J}, \wt{\xi}, \wt{\eta}, \wt{{\bf b}})$ Cartan variation of $\mk{p}_u$.  Of course Lemma \ref{wid0} and Corollary \ref{J(E_{alpha}) in frak{h}} are valid for any Cartan variation of $\mk{q}$.
Furthermore,  $(\mk{p}, \mk{p}^+, J, \xi, \eta, {\bf b})$ is a Cartan variation of $\mk{p}_u$ (see \eqref{hhh1}).
\end{df}

\subsubsection{Equations for \texorpdfstring{${\bf a}, \wt{{\bf b}} , \wt{\xi}, \wt{\eta}$}{\space}.}
The following theorem is well known (see e.g. \cite{Kobayashi69}, Ch. X)
\begin{theorem} \label{theorem for Nij} Let ${\bf U}$ be a connected Lie group with Lie algebra $\mk{u}$. Let ${\bf K}_u$ be a closed subgroup of ${\bf U}$ with Lie algebra $\mk{k}_u$. Let $J_M$ be a left-invariant integrable  complex structure  on the coset space $M = {\bf U}/{\bf K}_u$. Let $\wt{\mk{p}}_u$ be  any subspace of $\mk{u}$ such that
\ben \mk{u} = \mk{k}_u \oplus \wt{\mk{p}}_u.\een
Then the complex structure on $\wt{\mk{p}}_u $, defined  as follows
\ben \widetilde{J} = ((\rd
\pi )_{\vert \wt{\mk{p}}_u} )^{-1}\circ J_M(o) \circ (\rd
\pi )_{\vert \wt{\mk{p}}_u} \een  satisfies the following equation
\begin{gather} \label{Nijenhuis 1} \forall X, Y \in \wt{\mk{p}}_u \qquad  [\widetilde{J}(X),\widetilde{J}(Y)]_{\vert \wt{\mk{p}}_u } - [X,Y]_{\vert \wt{\mk{p}}_u} -\widetilde{J}([X, \widetilde{J}(Y)]_{\vert \wt{\mk{p}}_u}) - \widetilde{J}([\widetilde{J}(X),Y]_{\vert \wt{\mk{p}}_u})=0. \end{gather}
\end{theorem}

Let us remark that in \cite{Kobayashi69} the reductivity condition $[\mk{k}_u, \wt{\mk{p}}_u ] \subset \wt{\mk{p}}_u$ is presumed, but it is used there for other statements. One can verify that for the proof of the theorem above, this condition is not necessary.

We apply this theorem to the  complex structure $\widetilde{J}$ in any Cartan variation $(\wt{\mk{p}}, \wt{\mk{p}}^+, \wt{J}, \wt{\xi}, \wt{\eta}, \wt{{\bf b}})$ of $\mk{p}_u$  and from \eqref{Nijenhuis 1}  we infer
 \begin{gather} \label{Nijenhuis 2} \forall X, Y \in \widetilde{\mk{p} } \qquad [\widetilde{J}(X),\widetilde{J}(Y)]_{\vert \widetilde{\mk{p}} } - [X,Y]_{\vert \widetilde{\mk{p}}} -\widetilde{J}([X, \widetilde{J}(Y)]_{\vert \widetilde{\mk{p}}}) - \widetilde{J}([\widetilde{J}(X),Y]_{\vert \widetilde{\mk{p}}})=0. \end{gather}

Using this equation, we prove:

\begin{prop} \label{theorem union of the res from Nij}

Let $\alpha,\beta \in  \Delta_{\mk{p}}^+,\quad q = 1,\dots,m$. Then
\begin{gather}\label{NijEq2}
a_{\beta,\alpha}(\alpha + \beta)(U_q) = pr_{\mk{g}(-\beta)}\sco{\sum_{\nu\in\Delta_{\mk{p}}^+} \widetilde{\eta}_{\nu,q}[E_{-\nu},E_\alpha]}.
\end{gather}

%Let $\alpha,\beta,\gamma \in \Delta_{\mk{p}}^+,\ \widetilde{J}E_\gamma\in\mk{h}$. Then
%\begin{gather} \label{NijEq2}
%a_{\beta,\alpha}(\beta + \alpha)(\widetilde{J}(E_{-\gamma})) = pr_{\mk{g}(-\beta)}([E_\alpha,E_{-\gamma}]).
% \end{gather}
\end{prop}

\bpr
We  put in \eqref{Nijenhuis 2} $X=U_q$, $Y=E_\alpha$, apply  $\widetilde{J}$ on the equality  and  expand, using \eqref{def of a,b,xi,eta} and keeping explicit only  the coefficient before $E_{-\beta}$.  The computation is the same  as in \cite{DimTsan10}, Proposition 4.1.
 \epr

\begin{prop} \label{theorem union of the res from Nij new}  Let $\alpha,\beta,\gamma \in \Delta_{\mk{p}}^+,\ \widetilde{J}E_\gamma\in\mk{h}$. Then
\begin{gather} \label{NijEq3}
a_{\beta,\alpha}(\beta + \alpha)(\widetilde{J}(E_{-\gamma})) = pr_{\mk{g}(-\beta)}([E_\alpha,E_{-\gamma}]).
 \end{gather}
\end{prop}
\bpr We  put in \eqref{Nijenhuis 2} $X=E_{-\gamma}$, $Y=E_\alpha$  and  expand, using \eqref{def of a,b,xi,eta} and keeping explicit only   terms with component in $\mk{n}^-$.  The computation is the same  as in \cite{DimTsan10}, Proposition 4.9.
\epr

\begin{prop} \label{theorem union of the res from Nij 2} Let $\alpha,\beta,\gamma \in \Delta_{\mk{p}}^+,\ \widetilde{J}E_\gamma\in\mk{h}$.  Let $\forall \zeta \in \Delta_{\mk{p}}^+ $  $a_{\gamma \zeta} =a_{\zeta \gamma}= 0$  and $ \beta + \alpha  \neq \gamma $, then
\be
\label{NijEq4} \sum_{\zeta \in \Phi_{\gamma}^{\mk{p}} } a_{\zeta \beta} N_{\gamma ,-\zeta }  a_{\alpha, \gamma- \zeta}=0.
\ee

\end{prop}
\bpr We just put in \eqref{Nijenhuis 2} $X=E_\gamma$, $Y=E_\beta$ and use \eqref{def of a,b,xi,eta}.  Keeping explicit only terms with component in $\mk{n}^-$,  by  the same computation as in  \cite{DimTsan10}, Lemma 4.11  we obtain\footnote{ now in the conditions of the proposition  we have  $\forall \zeta \in \Delta_{\mk{p}}^+ $  $a_{\gamma, \zeta} =a_{\zeta, \gamma}= 0$, because in the computation we need $a_{\gamma, \beta}=0$ and we have not proved yet that the vanishing of $a_{\beta, \gamma }$ implies  $a_{\gamma, \beta}=0$}:
\begin{gather*}
a_{\alpha,\beta}(\alpha+\beta)(\widetilde J E_\gamma) + \sum_{\nu \in \Phi_{\gamma}^{\mk{p}} } N_{\gamma,-\nu} a_{\nu,\beta}a_{\alpha,\gamma - \nu}=0.
\end{gather*}
We recall that the notation $\Phi_{\gamma}^{\mk{p}}$  is in \eqref{prop of Delta+ def of B_{gamma}^{frak{m}}}.

When $\beta +\alpha \neq \gamma$, Proposition \ref{theorem union of the res from Nij new} gives\footnote{Because $(\alpha+\beta)(\widetilde J E_\gamma) = \ol{(\alpha+\beta)(\widetilde J E_{-\gamma}) }$.} $a_{\alpha,\beta}(\alpha+\beta)(\widetilde J E_\gamma) = 0$, whence the lemma.

 \epr

From these propositions we derive some useful corollaries.

\begin{coro} \label{lemma about J(E_{gamma}) in frak{h} and a_{alpha beta}} Let $(\wt{\mk{p}}, \wt{\mk{p}}^+, \wt{J}, \wt{\xi}, \wt{\eta}, \wt{{\bf b}})$  be any Cartan variation of $\mk{p}_u$.   Let $\gamma \in \Delta_{\mk{p}}^+$ be such that $\widetilde{J} (E_{\gamma}) \in \mk{h}$. Let $\alpha, \beta \in \Delta_{\mk{p}}^+$ be such that $\alpha + \beta = \gamma$. Then  $\gamma\left (\widetilde{J}(E_{\gamma})\right ) \neq 0$, $a_{\alpha \beta} \neq 0$, $a_{ \beta \alpha} \neq 0$ and
\be \label{lemma about J(E_{gamma}) in frak{h} and a_{alpha beta} a_{alpha beta} = ...} a_{\alpha \beta} = \frac{N_{\gamma, - \beta}}{\overline{\gamma\left (\widetilde{J}(E_{\gamma})\right )} } \qquad   a_{ \beta \alpha} = \frac{N_{\gamma, - \alpha}}{\overline{\gamma \left (\widetilde{J}(E_{\gamma})\right )} }.\ee \end{coro}
\begin{proof} Since  $\alpha,\beta,\gamma \in \Delta_{\mk{p}}^+,\ \widetilde{J}E_\gamma\in\mk{h}$ and $\alpha+\beta = \gamma $ then by \eqref{NijEq3} in Proposition  \ref{theorem union of the res from Nij new} $$ a_{\beta \alpha} \gamma(\widetilde{J} E_{-\gamma}) = N_{\alpha,-\gamma }=N_{\gamma,-\alpha} \qquad a_{ \alpha \beta} \gamma(\widetilde{J} E_{-\gamma}) = N_{\beta,-\gamma }=N_{\gamma,-\beta}$$
and the corollary follows.
\end{proof}

\begin{coro} \label{lemma a_{alpha beta}=0 if alpha + beta not in delta}  Let $(\wt{\mk{p}}, \wt{\mk{p}}^+, \wt{J}, \wt{\xi}, \wt{\eta}, \wt{{\bf b}})$  be any Cartan variation of $\mk{p}_u$.   Let $\alpha, \beta \in  \Delta_{\mk{p}}^+$ and  $\alpha + \beta $ does not vanish on $\widetilde{\mk{p}} \cap \mk{h}$. Then
\begin{itemize}
    \item[(i)] $a_{\alpha \beta}=0 \iff a_{\beta \alpha} = 0$,
    \item[(ii)] If $\alpha + \beta \not \in \Delta_{\mk{p}}^+$, then $a_{\alpha \beta} = a_{\beta \alpha} = 0$.
\end{itemize}
 \end{coro}
\begin{proof} Now  $\{U_1,\dots, U_m, V_1, \dots, V_m\}$ is a basis of $\wt{\mk{p}}\cap \mk{h}$ and we have $\delta(V_i)=-\ol{\delta(U_i)}$ for any $\delta \in \Delta$. Since $\alpha + \beta$ does not vanish on $\widetilde{\mk{p}} \cap \mk{h}$ then $(\alpha + \beta) (U_q) \neq 0$ for some $q \in \{1,\dots,m\}$. Then the lemma follows from \eqref{NijEq2} and the formula
\ben pr_{\mk{g}(-\beta)}\sco{\sum_{\nu\in\Delta_{\mk{p}}^+} \widetilde{\eta}_{\nu,q}[E_{-\nu},E_\alpha]} = \left \{
\begin{array}{c c c} 0 & \mbox{if} & \alpha + \beta \not \in \Delta_{\mk{p}}^+ \\
\widetilde{\eta}_{\alpha+\beta,q} N_{-\alpha - \beta, \alpha} & \mbox{if}  & \alpha + \beta  \in \Delta_{\mk{p}}^+.
    \end{array}
  \right. \een
  In particular we see that $pr_{\mk{g}(-\beta)}\sco{\sum_{\nu\in\Delta_{\mk{p}}^+} \widetilde{\eta}_{\nu,q}[E_{-\nu},E_\alpha]} = 0 $ iff $pr_{\mk{g}(-\alpha)}\sco{\sum_{\nu\in\Delta_{\mk{p}}^+} \widetilde{\eta}_{\nu,q}[E_{-\nu},E_\beta]} = 0 $.
\end{proof}

\subsection{Results for J. Determining of the subsets \texorpdfstring{$\Delta_{\mk{k}}^+$}{\space}, \texorpdfstring{$\Delta_{\mk{p}}^+$}{\space}}

In  subsection \ref{The integrability condition} we obtained some
results for $\wt{J}$, related to any Cartan variation
$(\wt{\mk{p}}, \wt{\mk{p}}^+, \wt{J}, \wt{\xi}, \wt{\eta},
\wt{{\bf b}})$ of $\mk{p}_u$. In this subsection we resume to
analyse   the complex structures $I$ and $J$ defined in subsection
\ref{some notations}. In particular  we calculate the coefficients
$\{a_{\alpha \beta}\}_{\{\alpha, \beta \in \Delta_{\mk{p}}^+\}}$,
defined in subsection \ref{The integrability condition}, Remark
\ref{remark for a,b,xi,eta and a,wt{b},wt{xi},wt{eta}}.
Furthermore, we describe the structure of  $\Delta_{\mk{k}}^+$
 and $\Delta_{\mk{p}}^+$, using the stem $\Gamma$.

We start with a Lemma
\begin{lemma}\label{pdu4} Let $\gamma \in \Gamma_{\mk{p}}$ be a maximal root.
Then \begin{gather} \label{pdu4 1} \forall \zeta \in \Delta_{\mk{p}}^+  \qquad a_{\zeta \gamma} = a_{\gamma \zeta} = 0  \\
W_{\gamma} \in \mk{p}_u, \qquad \Phi_{\gamma}^+ \subset \Delta_{\mk{p}}^+ \\
\label{pdu4 forall} \forall \alpha \in \Phi_{\gamma}^+ \ \ \forall \beta \in \Delta_{\mk{p}}^+ \qquad (a_{\alpha \beta } \neq 0 \ \mbox{or} \ a_{\beta \alpha } \neq 0)\iff \alpha + \beta = \gamma. \end{gather}
\end{lemma}
\begin{proof} We have $\alpha(H_{\gamma}) >0$ for $\alpha \in \Phi_{\gamma}^+$ (see Proposition \ref{prop of Delta+}, \eqref{pdep 2}) and  $\gamma(H_{\gamma})=2$.  Since $\gamma$ is a maximal root then  $\alpha(H_{\gamma})=0$ for $\alpha \in \Delta^+ \setminus ( \Phi_{\gamma}^+ \cup \{\gamma\})$ (Proposition \ref{prop of Delta+}, \eqref{pdep 1} and \eqref{pdep 6}). Hence we see that
 \be  \label{pdu4 eq1}  \forall \alpha \in \Delta_{\mk{p}}^+ \qquad (\gamma + \alpha)(H_{\gamma}) \neq 0. \ee

 Now, since $\gamma \in \Gamma_{\mk{p}}$ is a maximal root,  Lemma  \ref{pdu3} implies that $H_{\gamma} \not \in \mk{k}$.
 We recall that $\mk{h}_u = \mk{h}\cap \mk{u}$ and $\mk{h}_{\mk{k}_u} = \mk{h}_u \cap \mk{k}_u$  (see \eqref{csj2}).   Since $H_{\gamma} \not \in \mk{k}$ then $W_\gamma \not \in \mk{h}_{\mk{k}_u}$, therefore we can choose  $\mk{q}$
such that    \be \label{ri H_{gamma} in  mk{h}_{widetilde{mk{m}_0}}} \mk{h}_u = \mk{q}
\oplus \mk{h}_{\mk{k}_u} \qquad W_{\gamma} \in  \mk{q}. \ee

From this $\mk{q}$ we obtain  a Cartan variation of $\mk{p}_u$ $(\wt{\mk{p}}, \wt{\mk{p}}^+, \wt{J}, \wt{\xi}, \wt{\eta}, \wt{{\bf b}})$ as in Definition \ref{def Cartan variation of mk{m}_0}.
Now \eqref{pdu4 eq1} and \eqref{ri H_{gamma} in  mk{h}_{widetilde{mk{m}_0}}}  show that for any $\zeta \in \Delta_{\mk{p}}^+$ the functional $(\zeta + \gamma)$ does not vanish on $\mk{h} \cap \widetilde{\mk{p}}$, hence Lemma \ref{lemma a_{alpha beta}=0 if alpha + beta not in delta}, (ii) is applicable. Therefore, $\gamma$ being  a maximal root, we proved \eqref{pdu4 1}. Therefore by Corollary \ref{J(E_{alpha}) in frak{h}} $J(E_{\gamma}) \in \mk{h}$ and Lemma \ref{lemma about J(E_{gamma}) in  frak{h} then H_gamma} implies that $W_{\gamma} \in \mk{p}_u$. On the other hand Lemma \ref{lemma about J(E_{gamma}) in  frak{h} then B_{gamma}} implies that $\Phi_{\gamma}^+ \subset \Delta_{\mk{p}}^+$.

Since  $\Phi_{\gamma}^+ \subset \Delta_{\mk{p}}^+$, then\footnote{for the notation $\Phi_{\gamma}^{\mk{p}}$ see \eqref{prop of Delta+ def of B_{gamma}^{frak{m}}}} $\quad \Phi_{\gamma}^{\mk{p}} = \Phi_{\gamma}^+$.
To prove \eqref{pdu4 forall}, we   first show that
\be \label{a_{zeta delta } neq 0 iff zeta + delta = gamma 1} \forall \zeta, \delta \in \Phi_{\gamma}^+ \cup \{\gamma \}  \ \ \ \ \ a_{\zeta \delta } \neq 0 \iff \zeta + \delta = \gamma.\ee
Indeed, for any $\zeta , \delta \in \Phi_{\gamma}^+ \cup \{\gamma \}$ we have $\zeta(H_{\gamma}) >0,  \delta(H_{\gamma})>0$ and then Corollary \ref{lemma a_{alpha beta}=0 if alpha + beta not in delta}, (ii), Corollary \ref{lemma about J(E_{gamma}) in frak{h} and a_{alpha beta}} in conjunction with the well known property (see Proposition \ref{prop of Delta+}, \eqref{pdep 0}):
\ben \label{a_{zeta delta } neq 0 iff zeta + delta = gamma 0} \forall \zeta, \delta \in \Phi_{\gamma}^+ \cup \{\gamma\}  \ \ \ \ \ \zeta + \delta \in \Delta^+ \iff \zeta + \delta = \gamma  \een imply formula \eqref{a_{zeta delta } neq 0 iff zeta + delta = gamma 1}.

Now we  take any $\alpha \in \Phi_{\gamma}^+$ and $\beta \in \Delta_{\mk{p}}^+$. Since we have proved \eqref{a_{zeta delta } neq 0 iff zeta + delta = gamma 1}, we may think that $\beta \not \in \Phi_{\gamma}^+ \cup \{ \gamma \}$.
 Now we apply Theorem \ref{theorem union of the res from Nij 2}, \eqref{NijEq4}, as we take into account \eqref{pdu4 1} and $\alpha + \beta \neq \gamma$, since $\beta \not \in \Phi_{\gamma}^+$, and obtain
\ben
 \sum_{\zeta \in \Phi_{\gamma}^+ } a_{\zeta \beta} N_{\gamma ,-\zeta }  a_{\alpha, \gamma- \zeta}=0.
\een
From \eqref{a_{zeta delta } neq 0 iff zeta + delta = gamma 1} it follows that for $\zeta \in \Phi_{\gamma}^+$ $a_{\alpha, \gamma- \zeta} \neq 0$ if and only if $\zeta = \alpha$. Therefore $a_{\alpha \beta} = 0$.

To prove $a_{\beta \alpha } = 0$ we  apply  Corollary \ref{lemma a_{alpha beta}=0 if alpha + beta not in delta} (i) (we may apply it since $\alpha(H_{\gamma}) > 0$, $\beta(H_{\gamma})\geq 0$).  The lemma is completely proved.
\end{proof}

In order that we prove the main theorem of this subsection: Theorem \eqref{th for J(E_gamma) in frak{h}}, we need of some preliminary discussion on $\Delta_{\mk{p}}^+$ and $\Gamma_{\mk{p}}$.
\begin{lemma} \label{if Gamma_{mk{m}} is nonempty} The subset $\Delta_{\mk{p}}^+$ is nonempty iff $\Gamma_{\mk{p}}$ (for the notation $\Gamma_{\mk{p}}$ see  \eqref{def of Gamma_{frak{m}},...}) contains a maximal root. Besides, for any $\gamma \in \Gamma$ and $\alpha \in \Phi_{\gamma}^+$ the inclusion $\alpha \in \Delta_{\mk{p}}^+$ implies $\gamma \in \Gamma_{\mk{p}}$, i. e.
\be \label{dec of Gamma_{mk{m}}}  \Delta_{\mk{p}}^+ =  \bigcup_{\gamma \in \Gamma_{\mk{p}}} (\Phi_{\gamma}^+ \cup \{ \gamma \} ) \cap \Delta_{\mk{p}}^+.  \ee
\end{lemma}
\bpr Let us recall that we have disjoint union $\Delta^+ = \Delta_{\mk{p}}^+ \cup \Delta_{\mk{k}}^+$ (see \eqref{wcb7}).

Let $\gamma \in \Gamma$, $\alpha \in \Phi_{\gamma}^+$ and $\alpha \in \Delta_{\mk{p}}^+$. By   Lemma \ref{pdu2} $\gamma \in \Delta^+\setminus \Delta_{\mk{k}}^+ =\Delta_{\mk{p}}^+$, hence $\gamma \in \Gamma_{\mk{p}}$ (recall that $\Gamma_{\mk{p}}= \Gamma \cap \Delta_{\mk{p}}^+$).

Thus we see that for any $\gamma \in \Gamma$ and $\alpha \in \Phi_{\gamma}^+$ the inclusion $\alpha \in \Delta_{\mk{p}}^+$ implies $\gamma \in \Gamma_{\mk{p}}$.

Now \eqref{dec of Gamma_{mk{m}}} follows from the partition \eqref{pdu1}.

If  $\Gamma_{\mk{p}}$ does not contain a maximal root, then all the maximal roots are in $\Gamma_{\mk{k}} = \Gamma \setminus \Gamma_{\mk{p}}$. On the other hand from \eqref{pdep 3} in Proposition \ref{prop of Delta+}  we know that  for any root $\gamma \in \Gamma$ there exists a maximal root $\widetilde{\gamma}\in \Gamma $, such that $\gamma \succeq \widetilde{\gamma} $.  All the maximal roots being in $\Gamma_{\mk{k}}$, we see that  for any root $\gamma \in \Gamma$ there exists a maximal root $\widetilde{\gamma}\in \Gamma_{\mk{k}} $, such that $\gamma \succeq \widetilde{\gamma} $. Hence
 Lemma \ref{pdu2} implies that $ \Delta^+ = \Delta_{\mk{k}}^+ $, therefore $\Delta_{\mk{p}}^+ = \emptyset $.
\epr
This lemma shows that if $\Gamma_{\mk{p}}$ is non empty, then the elements of $\Gamma_{\mk{p}}$ are a subsequence of  $\gamma_1, \dots, \gamma_d$, whose first element is a maximal root (we recall that the maximal roots are in the beginning of $\gamma_1, \dots, \gamma_d$ - see Corollary \ref{gss1}). Let this subsequence be $\gamma_{i_1}, \dots, \gamma_{i_{d_{\mk{p}}}}$  with length $d_{\mk{p}}$    (here $1 \leq i_1 < i_2 < \dots < i_{d_{\mk{p}}} \leq d$). For the sake of brevity, we shall denote
$$
\mu_{1}= \gamma_{i_1}, \ \  \mu_2= \gamma_{i_2}, \ \ \dots, \ \ \mu_{d_{\mk{p}}}= \gamma_{i_{d_{\mk{p}}}}.
$$
With these notations we represent $\Gamma_{\mk{p}}$, if it is nonempty, as a sequence $\Gamma_{\mk{p}}=$ $\{\mu_1, \mu_2, \dots, \mu_{d_{\mk{p}}}\}$, such that $\mu_1$ is a maximal root and if $\mu_i, \mu_j \in \Gamma_\mk{p}$, then
\be \label{phf1}  \mu_i \prec \mu_j \Rightarrow i<j.\ee

\begin{theorem} \label{th for J(E_gamma) in frak{h}}   Let $\gamma \in \Gamma_{\mk{p}}$, then $J(E_{\gamma}) \in \mk{h}$ and $W_{\gamma} \in \mk{p}_u$. Furthermore
\begin{gather} \Delta_{\mk{p}}^+ = \bigcup_{\gamma \in \Gamma_{\mk{p}}} \Phi_{\gamma}^+ \cup \{\gamma\} \\
 \label{th for J(E_gamma) in frak{h} matrix a}  \forall \alpha \in \Delta_{\mk{p}}^+ \forall \beta \in \Delta_{\mk{p}}^+ \qquad a_{\alpha \beta } \neq 0 \iff \alpha + \beta \in  \Gamma_{\mk{p}}.  \end{gather}
\end{theorem}
\begin{proof}
We represented $\Gamma_{\mk{p}}$ as a sequence $\mu_1, \mu_2, \dots, \mu_{d_{\mk{p}}}$, where $\mu_1$ is a maximal root.

We shall prove the theorem by induction with respect to the integer $1 \leq i \leq d_{\mk{p}}$. Let us assume that for some $1 \leq i \leq d_{\mk{p}}-1$ the following  holds
\begin{gather}  \forall j\in \{1,\dots,i\} \qquad \qquad \left ( H_{\mu_j} \in \mk{p} \ \ \mbox{and} \ \ \Phi_{\mu_j}^{\mk{p}}=\Phi_{\mu_j}^+ \subset \Delta_{\mk{p}}^+ \ \ \mbox{and} \right. \nonumber \\[-2mm] \label{pdu5} \\[-2mm]
\left . \left ( \forall \alpha \in \Phi_{\mu_j}^+\cup \{ \mu_j \} \ \forall \beta \in \Delta_{\mk{p}}^+ \qquad (a_{\alpha \beta } \neq 0 \ \mbox{or} \ a_{\beta \alpha } \neq 0)\iff \alpha + \beta = \mu_j \right ) \right ). \nonumber
\end{gather}
If $i=1$, then \eqref{pdu5} holds by Lemma \ref{pdu4} (since  $\mu_1$ is a maximal root).

Now we consider the case  ${i+1}$.

First, we shall prove that $H_{\mu_{i+1}} \not \in \mk{k}$. Indeed, let us assume that $H_{\mu_{i+1}} \in \mk{k}$. Then,  since (see \eqref{phf1} and Proposition \ref{prop of Delta+})
 \be \label{phf2} \forall \alpha \in \bigcup_{j=i+1}^{d_{\mk{p}}} (\Phi_{\mu_j}^+ \cup \{\mu_j\}) \cap \Delta_{\mk{p}}^+  \quad \qquad \alpha(H_{\mu_{i+1}}) \geq 0,
  \ee
 by Lemma \ref{help lemma for H_gamma not in frak{k}}  we should obtain (recall the equality \eqref{dec of Gamma_{mk{m}}} in  Lemma \ref{if Gamma_{mk{m}} is nonempty} also) $$ J(E_{\mu_{i+1}}) = \sum_{\left \{ \alpha \in  \Phi_{\mu_j}^+ \cup \{\mu_j\}; 1 \leq j \leq i \right  \}  } a_{\alpha, \mu_{i+1}} E_{-\alpha}.$$
On the other hand the induction assumption \eqref{pdu5} shows that for $\alpha \in  \Phi_{\mu_j}^+ \cup \{\mu_j\}$,  $0\leq j \leq i$  we have $ a_{\alpha, \mu_{i+1}}=0$ and we see that  $H_{\mu_{i+1}} \in \mk{k}$ implies $J(E_{\mu_{i+1}})=0$, which contradicts $J^2 = - Id$.

 Therefore $H_{\mu_{i+1}} \not \in \mk{k}$ and we can choose  $\mk{q}$
such that
$$
\mk{h}_u = \mk{q}
\oplus \mk{h}_{\mk{k}_u} \qquad W_{\mu_{i+1}} \in  \mk{q}.
$$
From this $\mk{h}_{\wt{\mk{p}}_\mk{u} }$ we obtain  a Cartan variation of $\mk{p}_u$ $(\wt{\mk{p}}, \wt{\mk{p}}^+, \wt{J}, \wt{\xi}, \wt{\eta}, \wt{{\bf b}})$ as in Definition \ref{def Cartan variation of mk{m}_0}.

By \eqref{phf2} we see that for  $\alpha \in  (\Phi_{\mu_j}^+ \cup \{\mu_j\})\cap \Delta_{\mk{p}}^+$,  $i+1\leq j \leq d_{\mk{p}}$ the functional    $(\alpha + \mu_{i+1})$
 does not vanish on $\mk{h} \cap \widetilde{\mk{p}}$, besides for such type of $\alpha$ we have   $\alpha+\mu_{i+1} \not \in \Delta$, hence Corollary \ref{lemma a_{alpha beta}=0 if alpha + beta not in delta}, (ii) is applicable to $\alpha$ and $\mu_{i+1}$.  Therefore
 \be  \label{th for J(E_gamma) in frak{h} a_{alpha gamma}= a_{gamma alpha}=0}  \forall \alpha \in \bigcup_{j=i+1}^{d_{\mk{p}}} (\Phi_{\mu_j}^+ \cup \{\mu_j\}) \cap \Delta_{\mk{p}}^+ \qquad \quad a_{\alpha, \mu_{i+1}} = a_{\mu_{i+1}, \alpha}=0. \ee    The last relation together with the induction assumption \eqref{pdu5} and  equality \eqref{dec of Gamma_{mk{m}}} in  Lemma \ref{if Gamma_{mk{m}} is nonempty}   imply
 \be \label{th for J(E_gamma) in frak{h} a_{alpha gamma} = a_{gamma alpha}=0} \forall \alpha \in \Delta_{\mk{p}}^+ \qquad  a_{\alpha ,\mu_{i+1}} = a_{\mu_{i+1}, \alpha} = 0. \ee
  Therefore $J(E_{\mu_{i+1}}) \in \mk{h}$ and Lemmas \ref{lemma about J(E_{gamma}) in  frak{h} then H_gamma} and \ref{lemma about J(E_{gamma}) in  frak{h} then B_{gamma}} imply
  \be \label{th for J(E_gamma) in frak{h} H_{gamma} in frak{m}}  W_{\mu_{i+1}} \in \mk{p}_u \qquad \Phi_{\mu_{i+1}}^{\mk{p}} =\Phi_{\mu_{i+1}}^+ \subset \Delta_{\mk{p}}^+. \ee
Now since $H_{\mu_{i+1}} \in \mk{p}$ then  for  any pair $\zeta , \delta \in \Phi_{\mu_{i+1}}^+ \cup \{ \mu_{i+1} \}$ the functional $ \zeta + \delta$ does not vanish on $\mk{p} \cap \mk{h}$, besides we know that for such a pair the relation $\zeta + \delta \in \Delta$ holds iff  $\zeta + \delta = \mu_{i+1}$.   Therefore Corollary \ref{lemma a_{alpha beta}=0 if alpha + beta not in delta}, (ii) and  Corollary \ref{lemma about J(E_{gamma}) in frak{h} and a_{alpha beta}}  imply \be \label{a_{zeta delta } neq 0 iff zeta + delta = gamma in theorem} \forall \zeta, \delta \in \Phi_{\mu_{i+1}}^+ \cup \{\mu_{i+1}\}  \ \ \ \ \ a_{\zeta \delta } \neq 0 \iff \zeta + \delta = \mu_{i+1}.\ee
In order that we complete the induction, it remains to prove that for any $\alpha \in \Phi_{\mu_{i+1}}^+$ and $\beta \in \Delta_{\mk{p}}^+ \setminus (\Phi_{\mu_{i+1}}^+ \cup \{ \mu_{i+1} \} )$ the equalities $a_{\alpha \beta} = a_{\beta \alpha} = 0$ are valid. To that end we employ  Theorem \ref{theorem union of the res from Nij 2}, \eqref{NijEq4}, as we take into account \eqref{th for J(E_gamma) in frak{h} a_{alpha gamma} = a_{gamma alpha}=0} and $\alpha + \beta \neq \mu_{i+1}$, since $\beta \not \in \Phi_{\mu_{i+1}}^+$, and obtain
\be
 \sum_{\zeta \in \Phi_{\mu_{i+1}}^+ } a_{\zeta \beta} N_{\mu_{i+1} ,-\zeta }  a_{\alpha, \mu_{i+1}- \zeta}=0.
\ee
By \eqref{a_{zeta delta } neq 0 iff zeta + delta = gamma in theorem} it follows that, as $\zeta$ varies through $\Phi_{\mu_{i+1}}^+$,  $a_{\alpha, \mu_{i+1} - \zeta } \neq 0$ iff $\zeta = \alpha$ and therefore the equation above is reduced to $a_{\alpha \beta} = 0$. To prove that $a_{\beta \alpha} = 0$ we can use     Corollary \ref{lemma a_{alpha beta}=0 if alpha + beta not in delta}, (i), provided that $\alpha + \beta $ does not vanish on $\mk{p} \cap \mk{h}$. The nonvanishing of $\alpha + \beta $ for $\beta \in \Phi_{\mu_j}^+ \cup \{\mu_j\}$,  $j\geq i+1$ follows from $\alpha(H_{\mu_{i+1}}) >0$ and $\beta(H_{\mu_{i+1}}) \geq 0$ (here we use \eqref{th for J(E_gamma) in frak{h} H_{gamma} in frak{m}}).  The nonvanishing of $\alpha + \beta $ for $\beta \in \Phi_{\mu_j}^+ \cup \{\mu_j\}$,  $j\leq i$ follows from $\alpha(H_{\mu_j}) \geq 0$ and $\beta(H_{\mu_j}) > 0$ (here we use that by the induction assumption \eqref{pdu5} $H_{\mu_j} \in \mk{p}$).

Thus far, we proved by induction that for any $\gamma \in \Gamma_{\mk{p}}$, we have  $J(E_{\gamma}) \in \mk{h}$, $H_{\gamma} \in \mk{p}$ and
\begin{gather*} \Delta_{\mk{p}}^+ = \bigcup_{\gamma \in \Gamma_{\mk{p}}} \Phi_{\gamma}^+ \cup \{\gamma\} \\
\forall \alpha \in \Delta_{\mk{p}}^+ \forall \beta \in
\Delta_{\mk{p}}^+ \qquad (a_{\alpha \beta } \neq 0 \ \mbox{or} \
a_{\beta \alpha } \neq 0)\iff \alpha + \beta \in  \Gamma_{\mk{p}}.
\end{gather*} Now Corollary \ref{lemma about J(E_{gamma}) in
frak{h} and a_{alpha beta}} shows that the theorem follows.
\end{proof}

\begin{coro}  \label{coro abot mk{g} and mk{k}}      $\mk{k}$ is a $\Gamma_{\mk{k}}$-stemmed subalgebra of $\mk{g}$.

 \end{coro}
\bpr
By Theorem \ref{th for J(E_gamma) in frak{h}} and Corollary \ref{coro for the property of Gamma_{frak{k}}} it follows that the subsets $\Delta_{\mk{k}}$, $\Delta_{\mk{k}}^+$ are of the type
\be  \Delta_{\mk{k}}^+ = \bigcup_{\gamma \in \Gamma_{\mk{k}}} \Phi_{\gamma}^+ \cup \{\gamma\}, \nonumber  \\[-2mm] \label{pdp31} \\[-2mm]
\Delta_{\mk{k}} = \bigcup_{\gamma \in \Gamma_{\mk{k}}} \Phi_{\gamma} \cup \{\gamma\}, \nonumber \ee
where $\Gamma_{\mk{k}}$  is a substem of $(\Gamma, \prec)$. Now the decomposition \eqref{frak{k} is regular} shows that $\mk{k}$ is a $\Gamma_{\mk{k}}$-subalgebra.

\epr

Now, after having proved Theorem \ref{th for J(E_gamma) in frak{h}}, by the same ideas  as in \cite{DimTsan10} we shall describe the complex structures  $I$ and $J$ more explicitely. To that end now we  introduce  notaions (additional to \eqref{csj2}):
\begin{gather}
  \label{csj3}  \mk{o}_{\mk{p}} = \mk{p} \cap \mk{o} ,\quad \mk{o}_{\mk{k}}=\mk{k} \cap \mk{o},\quad \mk{w}_{\mk{p}} = span_\RR\{W_\gamma\vert \gamma \in \Gamma_{\mk{p}}\}\\
  \mk{h}_{\mk{p}}=\mk{h}\cap \mk{p}, \qquad  \mk{h}_{\mk{p}}^+=\mk{h}\cap \mk{p}^+, \qquad  \mk{h}_{\mk{p}}^-=\mk{h}\cap \mk{p}^-
   \\
   \mk{w}_{\mk{p}} = span_\RR\{W_\gamma\vert \gamma \in \Gamma_{\mk{p}}\}.
\end{gather}
By Theorem \ref{th for J(E_gamma) in frak{h}}, we have $\mk{w}_{\mk{p}}\subset \mk{h}_{\mk{p}}$.

From Lemma \ref{lemma if... then alpha(X)=0} and $\Gamma_{\mk{k}} \subset \Delta_{\mk{k}}$ it follows
 \be \label{formula for mk{0}_mk{p}}
 \mk{o}_{\mk{p}} = \mk{h}_{\mk{p}} \cap \bigcap_{\gamma \in \Gamma_{\mk{p}}} \ker(\gamma) = \{ X \in \mk{h}_{\mk{p}} \vert \ \forall \gamma \in \Gamma_{\mk{p}} \ \ \ \gamma(X) = 0 \}. \ee   Furthermore, by Theorem \ref{th for J(E_gamma) in frak{h}} for $\gamma \in \Gamma_{\mk{p}}$ we have that $J E_\gamma \in \mk{h}$. As in   \cite{DimTsan10}, we shall denote  for $\gamma \in \Gamma_{\mk{p}}$
\begin{gather}\label{ca7 vtora chast}
V_\gamma = JE_\gamma \in  \mk{h}_{\mk{p}}^-,\qquad U_\gamma = JE_{-\gamma}=-\tau (JE_\gamma) = - \tau (V_\gamma)\in \mk{h}_{\mk{p}}^+, \\
    P_\gamma = W_\gamma - \ri I W_\gamma,  \qquad   Q_\gamma = W_\gamma + \ri I W_\gamma.
\end{gather}

\begin{prop} \label{phf} The complex structure $J$ satisfies the following conditions:
\begin{gather} \label{the matrix xi in details 4}  \abs{\gamma(V_\gamma )} =\abs{\gamma(U_\gamma )}  =1\qquad \qquad\gamma \in \Gamma_{\mk{p}}
\\
  \label{the matrix xi in details 2}  JE_\gamma =- \ri\gamma(V_\gamma )Q_\gamma = \frac{1}{2} \gamma(V_\gamma )(H_\gamma + \ri I H_\gamma) \qquad \qquad\gamma \in \Gamma_{\mk{p}}
\\ \label{the matrix xi in details 1} JE_\alpha = N_{\gamma,- \alpha}\gamma(V_\gamma )E_{\alpha-\gamma}  \qquad \alpha \in \Phi_{\gamma}^+,\gamma \in \Gamma_{\mk{p}}. \end{gather}

 The complex structure $I$ satisfies    $I(\mk{w}_{\mk{p}}) \subset \mk{o}_{\mk{p}}$.
 \begin{remark} If we denote $\rho_\gamma = \ri \ol{\gamma(V_\gamma)}$, then the formulas in this proposition   are the same as the formulas in the second and the rhird row of \eqref{afi3}. We show further that the first formula in \eqref{afi3} is also necessarily valid.
\end{remark}

 \end{prop}
 \begin{proof} By Theorem \ref{th for J(E_gamma) in frak{h}} for $\gamma \in \Gamma_{\mk{p}}$ we have $J(E_\gamma)=V_\gamma \in \mk{h}, J(E_{-\gamma})=U_\gamma \in \mk{h}$, $W_{\gamma} \in \mk{p}_u$. Now\footnote{using the integrability condition in \eqref{Nijenhuis 2} for $J$, $\mk{p}$} by the same computations as in the proof of Proposition 4.7 \cite{DimTsan10} we obtain  \eqref{the matrix xi in details 4}, \eqref{the matrix xi in details 2} and \be \gamma (I H_\gamma)=0 \qquad \gamma \in \Gamma_{\mk{p}}.\ee
 If  $\gamma\neq \delta$, $\delta \in \Gamma_{\mk{p}}$, then we compute \eqref{Nijenhuis 2} with $\wt{J} = J$, $\wt{\mk{p}}=\mk{p}$, $ X=E_\gamma $,  $ Y=E_\delta $:
 \begin{gather*}
0 = [JE_\gamma,E_\delta]_{\mk{p}} + [E_\gamma, JE_\delta]_{\mk{p}}
= \delta(V_\gamma)E_\delta -\gamma(V_\delta)E_{\gamma}.
\end{gather*}
Hence $\delta(V_\gamma) = 0 $. By \eqref{the matrix xi in details 2}  $\delta(V_\gamma) = \frac{1}{2} \gamma(V_\gamma ) \delta(H_\gamma + \ri I H_\gamma) =\frac{\ri }{2} \gamma(V_\gamma ) \delta(I H_\gamma) $. Therefore $\delta (I H_\gamma)=0$ and we see that \be \forall \delta \in \Gamma_{\mk{p}} \qquad \delta(I H_\gamma)=0. \ee
If $\delta \in \Gamma_{\mk{k}}$ then $h_\delta \in \mk{k}$, $I H_\gamma \in \mk{p}$ and therefore $\delta (I H_\gamma ) = \scal{h_\delta, I H_\gamma}=0 $ (see \eqref{frak{m}_0 = X in
frak{u};forall Y in frak{k}_0} also). So far, we see that $I H_\gamma \in \mk{0}$, i. e. $I W_\gamma \in \mk{0}$ and we proved that $I(\mk{w}_{\mk{p}}) \subset \mk{o}_{\mk{p}_u }$.

We pass to the proof of \eqref{the matrix xi in details 1}\footnote{it its the same as the proof of proposition 4.18 in \cite{DimTsan10}}. Let us fix $\alpha \in \Phi_{\gamma}^+,\gamma \in \Gamma_{\mk{p}}  $ and denote $\beta = \gamma - \alpha$. From \eqref{th for J(E_gamma) in frak{h} matrix a} in Theorem \ref{th for J(E_gamma) in frak{h}}   and \eqref{def of a,b,xi,eta} we have
\begin{gather*}
\label{JE_alpha =a_{beta,alpha}E_{-beta} + H} JE_\alpha =a_{\beta,\alpha}E_{-\beta} + H, \qquad H \in \mk{h}_{\mk{p}}^-.
\end{gather*}
In the following computation we put in \eqref{Nijenhuis 2} $\wt{J} = J$, $\wt{\mk{p}}=\mk{p}$, $X= E_\gamma$, $Y=E_\alpha$ and keep   explicit only terms which nave nontrivial projection to $\mk{h}$.
 \begin{gather*} 0= [V_\gamma,a_{\beta,\alpha}E_{-\beta} + H]_{\vert \mk{p} }  - J[E_\gamma,a_{\beta,\alpha}E_{-\beta} + H]_{\vert \mk{p} }- J[V_\gamma,E_\alpha]_{\vert \mk{p} } \\
 = \gamma(H)V_\gamma  - N_{\gamma,-\beta}a_{\beta,\alpha}JE_\alpha - \alpha(V_\gamma)JE_\alpha +A\\
 = \gamma(H)V_\gamma - ( N_{\gamma,-\beta}a_{\beta,\alpha} + \alpha(V_\gamma))H +B,
 \end{gather*}
 where $A,B \in \mk{n}^-$. Now we use $a_{\beta,\alpha} = N_{\gamma,-\alpha}\gamma(V_\gamma)$
(see \eqref{lemma about J(E_{gamma}) in frak{h} and a_{alpha beta} a_{alpha beta} = ...} in Corollary \ref{lemma about J(E_{gamma}) in frak{h} and a_{alpha beta}}) and $\quad N_{\gamma, -\alpha}N_{\gamma, -\beta}= -1$ (Proposition 2.27 in \cite{DimTsan10}) to get
\begin{gather}\label{mxj4}
\gamma(H)V_\gamma + \beta(V_\gamma)H = 0
\end{gather}
We apply $\gamma$ to this equation and obtain  $\gamma(H)(\gamma + \beta)(V_\gamma) = 0 $. By   \eqref{the matrix xi in details 2} and $\beta\in\Phi^+_\gamma$ we have
   $$
\beta(V_\gamma) =  \gamma(V_\gamma)\sco{\frac{1}{2} + \mbox{ imaginary number }} \neq 0.$$
Thus $(\gamma + \beta)(V_\gamma) \neq 0 $, whence $\gamma(H) = 0$.
Now by \eqref{mxj4} and $\beta(V_\gamma)\neq 0$ we get $H = 0$. The formula \eqref{the matrix xi in details 1} follows from \eqref{JE_alpha =a_{beta,alpha}E_{-beta} + H} and $a_{\beta,\alpha} = N_{\gamma,-\alpha}\gamma(V_\gamma)$.
\end{proof}

\begin{df} As we show in the next two propositions, it is usefully to introduce some components of $ \mk{h}_{\mk{p}}$ as follows:
\begin{gather*}
 \mk{v}_{\mk{p}} = (\mk{w}_{\mk{p}}\oplus I \mk{w}_{\mk{p}})^\CC,\quad \mk{v}_{\mk{p}}^+ = \mk{v}_{\mk{p}}\cap\mk{h}_{\mk{p}}^+,\quad \mk{v}_{\mk{p}}^- = \mk{v}_{\mk{p}}\cap\mk{h}_{\mk{p}}^- ;\\
\mk{j}_{\mk{p}}^+  = \mk{o}_{\mk{p}}\cap \mk{h}_{\mk{p}}^+,\quad \mk{j}_{\mk{p}}^- =  \mk{o}_{\mk{p}}\cap \mk{h}_{\mk{p}}^+,\quad
\mk{j}_{\mk{p}} = \mk{j}_{\mk{p}}^+ \oplus \mk{j}_{\mk{p}}^-.
\end{gather*}
\end{df}

\begin{prop}\label{mj5 vtora chast}  We have \be \label{mj5 vtora chast form1} \mk{v}_{\mk{p}}^+ = span_\CC\{P_\gamma\vert\ \gamma\in \Gamma_{\mk{p}}\},\quad \mk{v}_{\mk{p}}^- = span_\CC\{Q_\gamma\vert\ \gamma\in \Gamma_{\mk{p}}\}, \quad \mk{v}_{\mk{p}}= \mk{v}_{\mk{p}}^+ \oplus \mk{v}_{\mk{p}}^-. \ee Besides
\begin{gather}\label{mj7}
\mk{h}_{\mk{p}}  = \mk{v}_{\mk{p}} \oplus \mk{j}_{\mk{p}},\quad \mk{h}_{\mk{p}}^+  = \mk{v}_{\mk{p}}^+\oplus \mk{j}_{\mk{p}}^+,\quad\mk{h}_{\mk{p}}^- = \mk{v}_{\mk{p}}^-\oplus \mk{j}_{\mk{p}}^-.
\end{gather}
\end{prop}
\bpr  The important point here is  $I(\mk{w}_{\mk{p}}) \subset \mk{o}_{\mk{p}}$, proved in Proposition \ref{phf}.

  The formulas in  \eqref{mj5 vtora chast form1} are easily verified by dimension count (obviously $span_\CC\{P_\gamma\vert\ \gamma\in \Gamma_{\mk{p}}\} \subset \mk{v}_{\mk{p}}^+$, $span_\CC\{Q_\gamma\vert\ \gamma\in \Gamma_{\mk{p}}\} \subset \mk{v}_{\mk{p}}^-$).

The condition $I(\mk{w}_{\mk{p}}) \subset \mk{o}_{\mk{p}}$ implies  that for $\gamma, \delta \in \Gamma_{\mk{p}}$  $\gamma(P_\delta) = 0$, if $\delta \neq \gamma$ and $\gamma(P_\gamma) = \ri $. On the other hand by \eqref{formula for mk{0}_mk{p}} $\mk{j}_{\mk{p}}^+ = \{H\in \mk{h}_{\mk{p}}^+\vert\ \forall \gamma \in \Gamma_{\mk{p}} \ \ \gamma(H)  = 0\}$, thus $\mk{h}_{\mk{p}}^+ = \mk{v}_{\mk{p}}^+\oplus\mk{j}_{\mk{p}}^+$. In the same way
$\mk{h}_{\mk{p}}^- = \mk{v}_{\mk{p}}^-\oplus\mk{j}_{\mk{p}}^-$.

Now, in order to prove $\mk{h}_{\mk{p}}  = \mk{v}_{\mk{p}}\oplus \mk{j}_{\mk{p}}$, we have to  show only that $\mk{v}_{\mk{p}}\cap \mk{j}_{\mk{p}}=\{0\}$.
Let $ X \in \mk{v}_{\mk{p}} \cap \mk{j}_{\mk{p}} $. We may decompose $X = X^+ + X^-$, where $X^+ \in \mk{v}_{\mk{p}}^+, X^- \in \mk{v}_{\mk{p}}^-$. Obviously $I(\mk{j}_{\mk{p}})=\mk{j}_{\mk{p}}$, hence $I(X)= \ri X^+ - \ri X^- \in \mk{j}_{\mk{p}}$. The inclusions  $X^+ + X^- \in \mk{j}_{\mk{p}},  \ri X^+ - \ri X^- \in \mk{j}_{\mk{p}}$ imply $X^+ \in \mk{j}_{\mk{p}}$, $X^- \in \mk{j}_{\mk{p}}$, therefore  $X^+ \in \mk{j}_{\mk{p}}^+ \cap \mk{v}_{\mk{p}}^+$, $X^- \in \mk{j}_{\mk{p}}^- \cap \mk{v}_{\mk{p}}^-$. Now from $\mk{v}_{\mk{p}}^+\cap \mk{j}_{\mk{p}}^+ = \mk{v}_{\mk{p}}^-\cap\mk{j}_{\mk{p}}^-=\{0\}$ we obtain $X^+=X^-=0$ .

\epr

Using Propositions \ref{phf}, \ref{mj5 vtora chast} and repeating the proof of Proposition 4.31 in \cite{DimTsan10} we obtain:

\begin{prop}  The complex structure $J$ satisfies:
\begin{gather}
\label{the matrix xi in details 3}  J(\mk{j}_\mk{p}^+) = \mk{j}_\mk{p}^-. \end{gather}
In particular the subspace $\mk{j}_{\mk{p}}$ is $J$-invariant.
 \end{prop}

\begin{coro}  \label{cars}   The following inequality is valid:
\ben
rank(\mk{g}) +  srank(\mk{k}) \geq rank(\mk{k}) + srank(\mk{g})\geq rank(\mk{k}_s) + srank(\mk{g}).
 \een
 \end{coro}
\bpr
   From Theorem \ref{th for J(E_gamma) in frak{h}} and  Proposition \ref{phf}  we know that for all $\gamma \in \Gamma_{\mk{p}}$ \ \  $I(W_\gamma) \in  \mk{h}_{\mk{p}_u}\cap \mk{o}$ and $W_\gamma \in  \mk{h}_{\mk{p}_u}$. On the other hand $\mk{o} \cap span_\CC \{W_{\gamma}\vert \gamma \in \Gamma_{\mk{p}}\} = \{0\}$.        Therefore  $ dim_\RR(\mk{h}_{\mk{p}_u})  \geq 2 \#(\Gamma_{\mk{p}})$. Hence $ dim_\RR(\mk{h}_u) -  dim_\RR(\mk{h}_{\mk{k}_u}) \geq  2 (\#(\Gamma )- \# (\Gamma_{\mk{k}} ) ) $, i. e. $  dim_\RR(\mk{h}_u) + 2 \# ( \Gamma_{\mk{k}} )  \geq  dim_\RR(\mk{h}_{\mk{k}_u}) + 2 \#(\Gamma   )$. By Corollary \ref{pdp3} it follows that
$2 \# ( \Gamma_{\mk{k}} ) = srank(\mk{k})$, hence $ rank(\mk{g}) +  srank(\mk{k}) \geq rank(\mk{k}) + srank(\mk{g})$.
\epr

\subsection{The main theorem}

We may now state our main theorem.
\begin{theorem} \label{main theorem}
Let ${\bf U}$ be a compact, connected Lie group and let
${\bf K}_u\subset{\bf U}$ be a closed subgroup such that ${\bf U}$ acts effectively on $M ={\bf U}/{\bf K}_u$. Let $(I_M,J_M)$ be a LIHCS on $M$.
Let $\mk{u}$ and $\mk{k}_u$ be the Lie algebras of ${\bf U}$ and ${\bf K}_u$, respectively.

Then $(\mk{u},\mk{k}_u )$ is a hypercomplex pair and $(M,I_M,J_M)$ is associated with it (see Definition \ref{hypercomplex space associated with a hyp}).
\end{theorem}

\section{Semisimple hypercomplex pairs} \label{shcp0}
In this section we determine explicitly  all hypercomplex pairs
$(\mk{u}, \mk{k}_u)$, such that $\mk{u}$ is semisimple.  We call
these special  pairs  { \it semisimple hypercomplex pairs } and
show that they  correspond to the simply-connected, compact,
hypercomplex spaces. Now we give a precise definition, excluding
some trivial situations. The auxiliary notion " \textit{weak
semisimple hypercomplex pair}" is used only in this section.
\begin{df}\label{semishp}  Let $(\mk{a}, \mk{b})$  be a hypercomplex pair.
If $\mk{a}$ is semisimple with decomposition into simple ideals $\mk{a} = \mk{a}_1 \oplus \mk{a}_2 \oplus \dots \oplus \mk{a}_p$, such that  for any $i=1, \dots , k$  $\mk{b} \cap \mk{a}_i $ is a {\bf proper } subalgebra of $\mk{a}_i$, then we say that $(\mk{a}, \mk{b})$ is a \textbf{semisimple hypercomplex pair}.

\textbf{Weak semisimple hypercomplex pair} we shall call a pair $(\mk{a}, \mk{b})$, which satisfies all the conditions of a semisimple hypercomplex pair, except possibly $\dim(\mk{a})-\dim(\mk{b})\in 4 \ZZ$, $rank(\mk{g}) +  srank( \mk{k}) \geq rank( \mk{k}) + srank(\mk{g})$ (see definition \ref{shp}).
\end{df}

The aim of this section is to  show that if $(\mk{a},\mk{b})$ is a semisimple hypercomplex pair, then the simple components of both $\mk{a}$ and $\mk{b}$ are only  $su$ algebras. We shall use the explicit description of the stems of the simple Lie algebras from Section 3.3 of \cite{DimTsan10}.

\subsection{If \texorpdfstring{$\mk{g}$}{\space} is simple and \texorpdfstring{$srank(\mk{g}) = 2 rank(\mk{g})$}{\space}.}  \label{sect for srank(mk{g}) = 2 rank(mk{g})}
We get in this case  if $\mk{g}$ is any of the simple  compact Lie algebras: $\mk{so}(4 q, \CC), q \geq 2$, $\mk{so}(2 d +1, \CC)$,  $\mk{sp}(d, \CC) $, $d \geq 1$,  $\mk{e}_8$, $\mk{e}_7$, $\mk{f}_4$, $\mk{G}_2$.
\begin{prop} \label{prop for srank(mk{g}) = 2 rank(mk{g})} Let $( \mk{g}, \mk{k})$ be a complexified weak semisimple  hypercomplex pair with simple $\mk{g}$, such that $srank(\mk{g}) = 2 rank(\mk{g})$. Then
\be rank(\mk{g})+ srank(\mk{k}) - rank(\mk{k}_s) - srank(\mk{g}) < 0, \nonumber  \ee
where $\mk{k}_s$ denotes the semisimple part of the reductive Lie algebra $\mk{k}$.
\end{prop}
\bpr Now  $srank(\mk{g}) = 2 rank(\mk{g})$ and therefore \begin{gather*}  rank(\mk{g})+ srank(\mk{k}) - rank(\mk{k}_s) - srank(\mk{g})= -rank(\mk{g})+ srank(\mk{k}) - rank(\mk{k}_s)\\ \leq -rank(\mk{g}) + rank(\mk{k}_s)\end{gather*}
where we used the inequality $srank(\mk{k}) \leq 2 rank(\mk{k}_s)$, proved in Corollary 2.44 of \cite{DimTsan10}.

Let  ${\mc M}$ be as in Lemma \ref{pdp3} (now, by the definition, $\mk{k}$ is a $\Gamma_{\mk{k}}$-stemmed subalgebra of $\mk{g}$ for some substem $\Gamma_{\mk{k}}$). Then
\be \label{a common formula in semisimple hcp} \mk{k} = (\mk{h}\cap \mk{k})  + \sum_{\alpha \in \Delta_{\mk{k}}} \mk{g}(\alpha) \qquad  \Delta_{\mk{k}} = \bigcup_{\gamma \in {\mc M}} \Theta_\gamma \qquad  \Delta_{\mk{k}}^+ = \bigcup_{\gamma \in {\mc M}} \Theta_\gamma^+\ee and  a basis of $\Delta_{\mk{k}}$ is $\Pi \cap \Delta_{\mk{k}}$, where $\Pi$ is a basis of $\Delta$. Therefore
 $-rank(\mk{g}) + rank(\mk{k}_s)= - \# (\Pi) + \# (\Pi \cap \Delta_{\mk{k}})$. So we see that $rank(\mk{g})+ srank(\mk{k}) - rank(\mk{k}_s) - srank(\mk{g}) \geq 0 $ implies $- \# (\Pi) + \# (\Pi \cap \Delta_{\mk{k}}) \geq 0$, hence $\Pi = \Pi \cap \Delta_{\mk{k}} $. On the other hand $\Pi = \Pi \cap \Delta_{\mk{k}} $ implies $\Delta = \Delta_{\mk{k}}$ and then $\mk{g} = \mk{k}$, which contradicts the definition of weak semisimple  hypercomplex pair \ref{semishp}.
\epr
\subsection{If \texorpdfstring{$\mk{g}=\mk{so}(2 p,\CC)$, $p=2 q +1$, $q \geq 2$}{\space}}

Here the root system is $\mk{d}_p$. We take $p$ orthonormal vectors
$e_1,\dots,e_p$ in the standard Euclidean $p$-dimensional space
and put
\begin{gather}
\Delta  = \{\pm e_i \pm e_j\vert i, j = 1, \dots, p , \ i < j \}, \quad
\Delta^+  =  \{  e_i \pm e_j\vert i, j =1, \dots,p,   \   i< j \}.  \nonumber
\end{gather}
Let us recall that the stem is $\Gamma = \{\gamma_1, \dots , \gamma_d \},\ d = 2q $, where
\begin{gather*}
\gamma_1 = e_1+e_2,\ \gamma_2 = e_3+e_4,\dots, \gamma_k=e_{2k-1}+e_{2k}, \dots, \ \gamma_q = e_{2q-1} + e_{2q},\\
\gamma_{q+1} = e_1 - e_{2},    \dots , \gamma_{q+k}=e_{2k-1}-e_{2k}, \dots,  \gamma_{2q} = e_{2q -1} - e_{2 q}.
 \end{gather*}
In this case the  root systems $\Theta_\gamma$ for $\gamma \in \Gamma$  are as follows
\begin{gather}  \Theta_{\gamma_k} = \Theta_{e_{2k-1}+e_{2k}} = \{\pm e_i \pm e_j \vert 2 k -1 \leq i < j \leq p \} = \mk{d}_{p-2k+2}  \nonumber \\
      \label{mk{u}=mk{so}(2 m), m=2 q +1 Delta_gamma}       \Theta_{\gamma_{q+k}} =\Theta_{e_{2k-1}+e_{2k}}= \{e_{2k-1}-e_{2k}, e_{2k}-e_{2k-1} \} = \mk{a}_{1}   \\
          k=1,\dots, q.   \nonumber
\end{gather}
Therefore
\be rank(\mk{g}) =p \qquad srank(\mk{g})= 2(p-1).  \ee
Now we  may depict the partially ordered set ($\Gamma ,\prec $ ) in the following Hasse diagram
\be \begin{array}{ccccccccccccc} \gamma_1 &  \rightarrow & \gamma_2 & \rightarrow & \gamma_3 & \rightarrow & \dots & \rightarrow & \gamma_k & \rightarrow & \dots & \rightarrow & \gamma_q   \\
\downarrow &  \searrow &     \downarrow      & \searrow &  \downarrow    & \searrow & \dots & \searrow & \downarrow & \searrow & \dots & \searrow &  \downarrow  \\
  \gamma_{q+1} &        & \gamma_{q+2} &       &  \gamma_{q+3} &        & \dots &  & \gamma_{q+k} &  & \dots &  & \gamma_{2q} \end{array}  \ee

And we see that the subsets of $\Gamma$, which consist of pairwise incomparable elements and do not contain the highest root $\gamma_1$,  are of the type:
\begin{gather} \label{type of {mc M} so(2p)1}   {\mc M} = \emptyset \quad \mbox{or}  \\
\label{type of {mc M} so(2p)2} {\mc M} = A \cup \{ \gamma_k\}   \qquad 2\leq k \leq q, \ \ A \subset \{\gamma_{q+1}, \gamma_{q+2}, \dots ,\gamma_{q+k-1}  \}  \\
\ \ \mbox{or} \ \ \nonumber \\
\label{type of {mc M} so(2p)3} {\mc M} = A    \qquad  \ \ A \subset \{\gamma_{q+1}, \gamma_{q+2}, \dots ,\gamma_{2 q}  \}. \end{gather}
Now we shall prove
\begin{prop} \label{prop for so(2p) p odd} Let $( \mk{g}, \mk{k})$ be a  complexified  weak semisimple  hypercomplex pair. Let $\mk{g} = \mk{so}(2p,\CC)$, where $p=2 q +1$, $q \geq 2$.  Then
$ rank(\mk{g})+ srank(\mk{k}) - rank(\mk{k}_s) - srank(\mk{g}) < 0. \nonumber  $
 \end{prop}
\bpr
Let ${\mc M}$ be the subset of the stem, which corresponds to $\mk{k}$ as  in Lemma \ref{pdp3} (see \eqref{a common formula in semisimple hcp} also). Then about ${\mc M}$ we have only three possible cases \eqref{type of {mc M} so(2p)1}, \eqref{type of {mc M} so(2p)2}, \eqref{type of {mc M} so(2p)3}.

\ul{If \texorpdfstring{${\mc M} = \emptyset$}{\space}.}
Then $rank(\mk{k}_s) = 0$ and $srank(\mk{k}) = 0$, hence (we recall that $q \geq 2$)
\begin{gather*} rank(\mk{g})+ srank(\mk{k}) - rank(\mk{k}_s) - srank(\mk{g}) =p-2(p-1)=-2 q+1 <0 .\end{gather*}

\ul{If ${\mc M} = A \cup \{ \gamma_k\}$, where $2 \leq k \leq q$,  $A \subset \{\gamma_{q+1}, \gamma_{q+2}, \dots ,\gamma_{q+k-1}  \}$.}
In this case \\ $ \Delta_{\mk{k}} = \Theta_{\gamma_k} \cup \bigcup_{\gamma \in A } \Theta_\gamma$,
hence by \eqref{mk{u}=mk{so}(2 m), m=2 q +1 Delta_gamma} we have
\be srank(\mk{k}) &=& srank(\Theta_{\gamma_k}) +\sum_{\gamma \in A } srank(\Theta_\gamma)=  \nonumber \\
 &=& srank( \mk{d}_{p-2k+2} )+ \#(A) srank(\mk{a}_1) = 2(p-2k+1) + 2 \#(A) \nonumber  \\
 rank(\mk{k}_s) &=&  rank( \mk{so}(2(p-2k+2),\CC)  )+ \#(A) rank(sl(2,\CC)) \nonumber \\
 &=& p-2k+2 +  \#(A).  \nonumber \ee
 In this case we compute:
 \begin{gather*} rank(\mk{g})+ srank(\mk{k}) - rank(\mk{k}_s) - srank(\mk{g}) \\
 = p +  2(p-2k+1) + 2 \#(A) - (p-2k+2 +  \#(A))- 2(p-1)\\
 = 2 + \#(A) -2k \leq 2 + (k-1)-2 k=1-k <0 .\end{gather*}

\ul{If ${\mc M}= A $, where  $A \subset \{\gamma_{q+1}, \gamma_{q+2}, \dots ,\gamma_{2 q}  \}$.}
In this case  $\Delta_{\mk{k}} = \bigcup_{\gamma \in A } \Theta_\gamma$,
hence by \eqref{mk{u}=mk{so}(2 m), m=2 q +1 Delta_gamma} we have:
\be srank(\mk{k}) &=&  \sum_{\gamma \in A } srank(\Theta_\gamma)= \#(A)srank(\mk{a}_1) = 2 \#(A) \nonumber  \\
 rank(\mk{k}_s) &=&   \#(A) rank(sl(2,\CC)) = \#(A).  \nonumber \ee
 Hence (we recall that $q \geq 2$):
 \begin{gather*} rank(\mk{g})+ srank(\mk{k}) - rank(\mk{k}_s) - srank(\mk{g})  = p +  2 \#(A)\\ -  \#(A)- 2(p-1)
   \leq  q-p+2=1-q<0.
   \end{gather*}

\epr

\subsection{If \texorpdfstring{$\mk{g}=\mk{e}_6$}{\space}.}

Here $\mk{g} = \mk{e}_6$ (we denote by the same symbols the root system $\mk{e}_6$ and the corresponding complex simle Lie algebra).

Let us recall that the stem in this case contains four elements and it is linearly ordered, i. e. we may order the elements of $\Gamma$ in a sequence   $\Gamma = \{\gamma_1, \gamma_2 , \gamma_3, \gamma_4 \}$, so that the corresponding  Hasse diagram is
\ben \begin{array}{ccccccc} \gamma_1 &  \rightarrow & \gamma_2 & \rightarrow & \gamma_3 & \rightarrow & \gamma_4. \end{array}  \een
These statements follow from the results in Example 3.25, Case 3, \cite{DimTsan10}.
We see that now
\be rank(\mk{g}) =6, \qquad srank(\mk{g}) = 8.\ee
Furthermore, we have shown in Example 3.25, Case 3, \cite{DimTsan10} that
 the  root systems $\Theta_\gamma$ for $\gamma \in \Gamma$  are as follows
\begin{gather} \label{mk{u}=the comp form of e_6 Delta_{gamma}} \Theta_{\gamma_1} =  \mk{e}_{6}, \ \  \Theta_{\gamma_2} =  \mk{a}_{5}, \ \ \Theta_{\gamma_3} =  \mk{a}_{3}, \ \ \Theta_{\gamma_4} =  \mk{a}_{1}.  \end{gather}

And we see that the subsets of $\Gamma$, which consist of pairwise incomparable elements  are of the type
\begin{gather}\label{type of {mc M} e_6} {\mc M} = \{\gamma_i\} \quad i =1,2,3,4  \quad \mbox{or}  \quad   {\mc M} = \emptyset .\end{gather}
\begin{prop} \label{prop for e_6} Let $( \mk{g}, \mk{k})$ be a complexified weak semisimple  hypercomplex pair and $\mk{g} = \mk{e}_6$.   Then
$ rank(\mk{g})+ srank(\mk{k}) - rank(\mk{k}_s) - srank(\mk{g}) < 0. \nonumber  $
 \end{prop}
\bpr Let ${\mc M}$ be the subset of the stem, which corresponds to $\mk{k}$ as  in Lemma \ref{pdp3}(see \eqref{a common formula in semisimple hcp} also). Then about ${\mc M}$ we have the possibilities \eqref{type of {mc M} e_6}.

\ul{If ${\mc M} = \emptyset$.}
Then $rank(\mk{k}_s) = 0$ and $srank(\mk{k}_s) = 0$, hence
\begin{gather*} rank(\mk{g})+ srank(\mk{k}) - rank(\mk{k}_s) - srank(\mk{g}) =6-8=-2 <0 .\end{gather*}

\ul{If $ {\mc M}= \{ \gamma_k\})$, where $k=2,3,4$.}  In this case  $ \Delta_{\mk{k}} = \Theta_{\gamma_k} $, where $k=2,3,4$,
hence by \eqref{mk{u}=the comp form of e_6 Delta_{gamma}} we have (we do not consider $ {\mc M}= \{ \gamma_1\}$, since $\mk{k}$ is a proper subalgebra of $\mk{g}$):
\be srank(\mk{k}) &=&  srank( \mk{a}_{9-2 k} ) =srank( \mk{a}_{2(5-k)-1} ) = 2(5-k)  \nonumber\\
 rank(\mk{k}_s) &=&  rank( sl(9-2 k+1,\CC)  )= 9-2 k.  \nonumber \ee
Using these formlas  we compute:
 \begin{gather*} rank(\mk{g})+ srank(\mk{k}) - rank(\mk{k}_s) - srank(\mk{g}) =6+2(5-k) -(9-2 k)-8=-1 <0.\end{gather*}
 \epr

\subsection{If \texorpdfstring{$\mk{g}=sl(n+1, \CC)$, $n \geq 1$}{\space}.}

Here $\mk{g} = sl(n+1, \CC)$, i. e. the traceless $(n+1) \times
(n+1)$ complex matrices. We shall denote by $E^i_j$ the $(n+1)
\times (n+1)$ matrix with $1$ at the intersection of the $i$-th
row and the $j$-th column and zeros at the others positions, where
$i$ and $j$ vary from $1$ to $n+1$. Let $\{ e_i \}_{i=1}^{n+1} $
be the dual basis of the basis $\{ E^i_i \}_{i=1}^{n+1}$ of the
subspace in $gl(n+1, \CC)$ of all diagonal matrices.  We shall
work  with the following  Cartan subalgebra  of $sl(n+1, \CC)$
\ben \mk{h} = span\{E^i_i - E^{i+1}_{i+1} \vert i = 1, \dots, n \}
\een and the following root system and subset of positive roots
\ben  \Delta = \{e_i - e_j \vert 1 \leq i \neq j \leq n+1\} \qquad
\quad \Delta^{+} = \{e_i - e_j \vert 1 \leq i < j \leq
n+1\}.\nonumber \een Let us recall that the stem $\Gamma$ contains
$d= \ksco{\frac{n+1}{2}}$ elements and in this case it is linearly
ordered, i. e. we may order the elements of $\Gamma$ in a sequence
$\Gamma = \{\gamma_1, \gamma_2 , \dots, \gamma_d \}$, so that the
corresponding  Hasse diagram is \ben
\begin{array}{cccccccc} \gamma_1 &  \rightarrow & \gamma_2 &
\rightarrow & \cdots & \rightarrow &  \gamma_d. \end{array}  \een
We see that now \be rank(\mk{g}) =n, \qquad srank(\mk{g}) = 2 d =
2 \ksco{\frac{n+1}{2}}.\ee Furthermore, we have shown that
 the  root systems $\Theta_\gamma$ for $\gamma \in \Gamma$ are as follows
\begin{gather} \label{mk{u}=mk{su}(n+1) Delta_{gamma}} \Theta_{\gamma_1} =  \mk{a}_{n}, \ \  \Theta_{\gamma_2} =  \mk{a}_{n-2}, \dots, \Theta_{\gamma_k} =  \mk{a}_{n-2k+2},  \dots , \Theta_{\gamma_d} =  \mk{a}_{n -2d+2}.  \end{gather}

And now the possibilities for ${\mc M}$ are
\begin{gather*} {\mc M} = \{\gamma_i\} \quad i =1,\dots,d  \quad \mbox{or}  \quad   {\mc M} = \emptyset .\end{gather*}

\begin{prop}\label{prop mk{g}=sl(n+1,CC)} Let $( \mk{g}, \mk{k})$ be a complexified weak semisimple  hypercomplex pair and  $\mk{g} = sl(n+1, \CC)$, $n\geq 1$.  Then if $n$ is odd and $srank(\mk{k}) = 0$  we have $ rank(\mk{g})+ srank(\mk{k}) - rank(\mk{k}_s) - srank(\mk{g}) =-1$, otherwise
$ rank(\mk{g})+ srank(\mk{k}) - rank(\mk{k}_s) - srank(\mk{g}) =0. \nonumber  $
 \end{prop}
 \bpr
\ul{If ${\mc M} = \emptyset$.}
Then  $rank(\mk{k}_s) = 0$ and $srank(\mk{k}_s) = 0$, hence
\begin{gather*} rank(\mk{g})+ srank(\mk{k}) - rank(\mk{k}_s) - srank(\mk{g}) =n-2\ksco{\frac{n+1}{2}} \\
=\left \{ \begin{array}{cc} 0 & \mbox{if $n$ is even} \\ -1 & \mbox{if $n$ is odd} \end{array} \right. .\end{gather*}

\ul{If ${\mc M} = \{ \gamma_k\}$, where $k=2,\dots, d$.}
In this case  $ \Delta_{\mk{k}} = \Theta_{\gamma_k} $, where $k=2,\dots, d$,
hence by \eqref{mk{u}=mk{su}(n+1) Delta_{gamma}} we have (we do not consider $ {\mc M}= \{ \gamma_1\}$, since $\mk{k}$ is a proper subalgebra of $\mk{g}$):
\be srank(\mk{k}) &=&  srank( \mk{a}_{n-2k+2} ) = 2\ksco{\frac{n-2k+3}{2}}=2\ksco{\frac{n+1}{2}}+2-2k \nonumber\\
 rank(\mk{k}_s) &=&  rank( sl(n-2k+3,\CC)  )= n-2k+2.  \nonumber \ee
Using these formulas  we compute:
 \begin{gather*} rank(\mk{g})+ srank(\mk{k}) - rank(\mk{k}_s) - srank(\mk{g}) \\ =n+2\ksco{\frac{n+1}{2}}+2-2k-(n-2k+2)-2\ksco{\frac{n+1}{2}}=0.\end{gather*}
 \epr

\begin{remark} \label{remark for embedding in mk{su}} Let $( \mk{g}, \mk{k})$ be a complexified weak semisimple  hypercomplex pair,  s. t. $\mk{g} = sl(n+1, \CC)$, $n\geq 1$, and $\mk{k}$ is a semisimple subalgebra of $\mk{g}$.  Then  $\mk{k}$ is isomorphic to $sl(n+3-2i,\CC)$ for some $i \in \{2,\dots,d\}$, which  is embedded in $sl(n+1, \CC)$ in terms of matrices as follows
\be  \label{remark for embedding in mk{su} matrices}  A  \mapsto \begin{bmatrix}0      & 0      & \cdots  & \cdots  & \cdots  & 0      & 0 \\
                                               0      & 0      & \cdots  & \cdots  & \cdots  & 0      & 0 \\
                                               \vdots & \vdots & \ddots &  \cdots &  \begin{rotate}{65} $\ddots $ \end{rotate}     & \vdots & \vdots \\
                                               \vdots & \vdots & \vdots & A      & \vdots & \vdots & \vdots \\
                                               \vdots & \vdots &  \begin{rotate}{65} $\ddots $ \end{rotate}      &  \cdots & \ddots & \vdots & \vdots  \\
                                               0      & 0      & \cdots  & \cdots  & \cdots  & 0      & 0 \\
                                               0      & 0      & \cdots  & \cdots  & \cdots  & 0      & 0
                                               \end{bmatrix}, \ee
where $A$ is $(n+3-2i) \times (n+3-2i)$ matrix and it is nested in a $(n+1) \times (n+1)$ matrix. In this situation the real subalgebra $\mk{k}_u = \mk{u} \cap \mk{k}$ is isomorphic to $su(n+3-2i)$ and it is embedded in $\mk{u} = su(n+1)$ in the same way.

\end{remark}

\subsection{Semisimple hypercomplex pairs}

\begin{df} \label{shp1} A weak semisimple hypercomplex pair  $(\mk{a}, \mk{b})$  will be called  $su$-hypercomplex pair, if $\mk{a} \cong su(n+1)$, $n \geq 2$, $\mk{b}$ is of the type   $\mk{b} \cong su(n+3-2 k)$, $k\geq 2$,  $n+3-2 k > 0$, and $\mk{b}$ is embedded in $\mk{a} \cong su(n+1)$ as shown in  Remark \ref{remark for embedding in mk{su}}, \eqref{remark for embedding in mk{su} matrices}.

If $n$ is even  we  include the pairs with vanishing $\mk{b}$, i. e. we allow in this definition $\mk{b}= \{0\}$, if $\mk{a} \cong{su}(n+1)$ with even $n$.

The complexification of an $su$-hypercomplex pair  will be called \ul{$sl$-hypercomplex pair}.

From Proposition \ref{prop mk{g}=sl(n+1,CC)} we see that the  pairs described here  are indeed hypercomplex pairs (not just weak hypercomplex pairs).

\end{df}

  Taking into account this definition and Remark \ref{remark for embedding in mk{su}}   we see that we may unite Propositions \ref{prop for srank(mk{g}) = 2 rank(mk{g})}, \ref{prop for so(2p) p odd}, \ref{prop for e_6}, \ref{prop mk{g}=sl(n+1,CC)}  in the following proposition \footnote{ In  subsection \ref{sect for srank(mk{g}) = 2 rank(mk{g})}, where the algebras $\mk{so}(4q)$ are considered, we start from $q=2$. The reason is that   $\mk{so}(4)\cong \mk{su}(2)\oplus \mk{su}(2)$.}  \footnote{For dealing with the cases $q=1$ in Proposition \ref{prop for so(2p) p odd} we recall that  $\mk{so}(6)\cong\mk{su}(4)$. } \footnote{From Proposition \ref{prop mk{g}=sl(n+1,CC)} we see that if $( \mk{g}, \mk{k})$ is a complexified weak semisimple  hypercomplex pair  with $\mk{g}=sl(1+1)$ (hence by definition $\mk{k}$ is a proper subalgebra of $\mk{g}$, which impose in this case $srank(\mk{k})=0$), then $ rank(\mk{g})+ srank(\mk{k}) - rank(\mk{k}_s) - srank(\mk{g}) < 0$. That's why in the definition of $\mk{su}$, resp. $\mk{sl}$, hypercomplex pairs \ref{shp1} we start from $n=2$. } :

\begin{prop} \label{theorem for simple mk{g}} Let $\mk{g}$ be a complex simple Lie algebra and $( \mk{g}, \mk{k})$ be a complexified weak semisimple  hypercomplex pair. Then  $ rank(\mk{g})+ srank(\mk{k}) - rank(\mk{k}_s) - srank(\mk{g}) \leq 0 \nonumber  $.

Equality is attained iff  $(\mk{g}, \mk{k}_s)$ is a $sl$-hypercomplex pair. (We recall that if $\mk{k}$ is abelian, we write $\mk{k}_s = \{0\}$).
\end{prop}

Now we can  prove the final theorem of this section:

\begin{theorem} \label{schc1} Let $(\mk{u},\mk{k}_u)$ be a semisimple hypercomplex pair with complexification $(\mk{g},\mk{k})$. Let $\mk{u}=\mk{u}_1 \oplus \dots \oplus \mk{u}_p$, where $\{ \mk{u}_i \}_{i=1}^p$ are the simple ideals of $\mk{u}$.
 Then   for $i=1,\dots,p $ the pair  $(\mk{u}_i, \mk{u}_i \cap \mk{k}_u)$ is an $su$-hypercomplex pair and $\mk{k}_u = (\mk{u}_1 \cap \mk{k}_u) \oplus (\mk{u}_2 \cap \mk{k}_u) \dots \oplus (\mk{u}_p \cap \mk{k}_u)$.

 In particular every  semisimple hypercomplex pair is {\bf minimal}  i. e.
 $$
 rank(\mk{g}) + srank(\mk{k}) - rank(\mk{k})- srank(\mk{g}) =0.
 $$
\end{theorem}
\begin{proof}  Let $\mk{h}, {\mc M}, \Delta, \Delta_{\mk{k}}$ be as in Definition \ref{shp} and Lemma \ref{pdp3} applied for $(\mk{u}, \mk{k}_u)$ and $(\mk{g},\mk{k})$, i. e.
\be \mk{k} = \mk{k} \cap \mk{h} + \sum_{\alpha \in \Delta_{\mk{k}}} \mk{g}(\alpha); \qquad \Delta_{\mk{k}} = \bigcup_{\gamma \in {\mc M}} \Theta_\gamma, \ee
\be \label{schc10} rank(\mk{g}) +  srank(\mk{k}) \geq rank(\mk{k}) + srank(\mk{g})\geq rank(\mk{k}_s) + srank(\mk{g}),  \ee
where we denote the semisimple part of $\mk{k}$ by $\mk{k}_s$ (if $\mk{k} \subset \mk{h}$, then we write $\mk{k}_s = \{0\}$).

  Let $\mk{g} = \oplus_{i=1}^p \mk{g}_i $ be the decomposition of $\mk{g}$ into complex simple ideals and $\Delta = \cup_{i=1}^p \Delta_{i}$ be the decomposition of $\Delta$ into irreducible root systems with positive roots $\{ \Delta_i^+ \}_{i=1}^p$. Obviously $\mk{g}_i$ is the complexification of $\mk{u}_i$. Besides
  \be \label{schc11}  rank(\mk{g}) = \sum_{i=1}^p rank(\mk{g}_i) \qquad srank(\mk{g})=\sum_{i=1}^p srank(\mk{g}_i).\ee
 Let for $i=1,\dots, p$ \be {\mc M}_i = {\mc M} \cap \Delta_i \quad \qquad \Delta_{\mk{k}_i} =  \bigcup_{\gamma \in {\mc M}_i} \Theta_\gamma = \Delta_{\mk{k}_i} \cap  \Delta_{\mk{k}}\ee
 and let $\mk{k}_i\subset \mk{g}_i$ be the minimal  subalgebra of $\mk{g}$, containing $\sum_{\alpha \in \Delta_{\mk{k}_i}} \mk{g}(\alpha)$. Then obviously $\mk{k}_i$ is semisimple, if it is not trivial, and  $( \mk{g}_i,\mk{k}_i )$ is complexified weak semisimple hypercomplex pair. Furthermore, $ \mk{k}_s $ is the minimal  subalgebra of $\mk{g}$, containing $\sum_{\alpha \in \Delta_{\mk{k}}} \mk{g}(\alpha)$ and it  is decomposed as follows
 $ \mk{k}_s = \oplus_{i=1}^p \mk{k}_i $, hence
 \be \label{schc12} rank(\mk{k}_s) = \sum_{i=1}^p rank(\mk{k}_i) \qquad \quad  srank(\mk{k}) = \sum_{i=1}^p srank(\mk{k}_i).\ee
 From \eqref{schc10}, \eqref{schc11} and \eqref{schc12} it follows that
 \begin{gather*} rank(\mk{g}) + srank(\mk{k}) - rank(\mk{k}_s)- srank(\mk{g})\\ =\sum_{i=1}^p rank(\mk{g}_i) + srank(\mk{k}_i) - rank(\mk{k}_i)- srank(\mk{g}_i) \geq 0. \end{gather*}
 Since  $( \mk{g}_i,\mk{k}_i )$ is complexified weak semisimple hypercomplex pair  and $\mk{g}_i$ is simple for $i \in \{1,\dots,p\}$, then from Theorem \ref{theorem for simple mk{g}} and the inequality above it follows that all the pairs $(\mk{g}_i, \mk{k}_i)$ for $i=1,\dots,p$ are $sl$-hypercomplex pairs and that \begin{gather*} rank(\mk{g}) + srank(\mk{k}) - rank(\mk{k}_s)- srank(\mk{g})\\ =\sum_{i=1}^p rank(\mk{g}_i) + srank(\mk{k}_i) - rank(\mk{k}_i)- srank(\mk{g}_i) = 0. \end{gather*}
  If  $\mk{k}_s$ was a proper subalgebra of $\mk{k}$ then we would obtain $rank(\mk{k}_s)<rank(\mk{k})$ and this would give $rank(\mk{g}) + srank(\mk{k}) - rank(\mk{k})- srank(\mk{g})< $ $ rank(\mk{g}) + srank(\mk{k}) - rank(\mk{k}_s)- srank(\mk{g}) = 0 $, which contradicts \eqref{schc10}. Therefore $\mk{k} = \mk{k}_s$, i. e. $\mk{k}$ is semisimple, if it is  non trivial, and  $\mk{k} = \oplus_{i=1}^p \mk{k}_i$ $=\oplus_{i=1}^p (\mk{k}\cap\mk{g}_i)$. Recalling that $\mk{g}_i$ is the complexification of $\mk{u}_i$ and that $\mk{k}$ is the complexification of $\mk{k}_u$, we see that $\mk{k}_u = \oplus_{i=1}^p (\mk{k}_u \cap \mk{u}_i)$. Besides we showed that $(\mk{g}_i, \mk{k}_i)$ are $sl$-hypercomplex pairs, therefore $(\mk{u}_i, \mk{k}_u \cap \mk{u}_i)$ are  $su$-hypercomplex pairs. The theorem is proved.
\end{proof}
Obviously, Theorem \ref{schc1} shows that the semisimple hypercomplex pairs are exactly the direct sums of $su$-hypercomplex pairs.

\subsection{HC spaces}
\label{schc} Here  we shall give a list  of all the simply connected, compact,  hypercomplex, homogeneous spaces (briefly { \bf HC-spaces}).

In Theorem \ref{schc1}  we characterized the semisimple hypercomplex pairs.    Here we shall prove:
\begin{theorem}\label{prop for HC space as hypercomplex pair} Any HC space  is  associated with a semisimple hypercomplex pair.
\end{theorem}

So, we are given a  HC-space $(M, I_M, J_M)$.
Let ${\bf G}_I$ be the group of
all biholomorphisms of the complex manifold $(M, I_M)$.
 Bochner and Montgomery  \cite{MontgomeryBoch1} have shown that ${\bf G}_I$ has a
real analytic structure/ such that ${\bf G}_I$ is a Lie group and
the action
$$ {\bf G}_I \times M \rightarrow M
\qquad (f, m) \in {\bf G}_I \times M \mapsto f(m) \in M,
$$
 is smooth.

 Let us denote $ {\bf G}= \{ f \in {\bf G}_I\vert \rd f \circ J_M = J_M
\circ \rd f\}$.  Obviously ${\bf G}$ is a subgroup of ${\bf
G}_I$ and it is exactly the group of all hypercomplex
diffeomorphisms of $M$. One can easily show that  ${\bf G}$  is a closed subgroup of ${\bf
G}_I$ and therefore it is a Lie subgroup of ${\bf G}_I $.
Thus, the mapping
$$
{\bf G} \times M \rightarrow M \ \qquad (f,m) \mapsto f(m)
$$
is a smooth transitive action of ${\bf G}$ on M.

 By Montgomery (see \cite{Montgomery}) we know that ${\bf G}$ contains a compact subgroup, say ${\bf U}_0$,
which acts transitively on $M$.
Now from \cite[p. 4]{Wang54} we conclude that the maximal semisimple subgroup of ${\bf U}_0$, which we
shall denote by ${\bf U}$, acts transitively on $M$. Thus far, we obtained a semisimple, compact, connected Lie group
${\bf U}$, which acts transitively by hypercomplex diffeomorphisms on $M$. It acts effectively also, since ${\bf U} \subset {\bf G}$.
 Hence we obtain that $M$ is a coset of the type \be \label{M as coset} {\bf U}/ {\bf K}_u  \simeq M.
 \ee  Since $M$ is simply connected then ${\bf K}_u$ is
a connected subgroup.
We denote the Lie algebras of ${\bf U}$ and ${\bf K}_u$ by $\mk{u}$ and $\mk{k}_u$, respectively.
\begin{lemma} \label{remark we may think mk{u}_i is proper sub of mk{k}_0} Let $\mk{u} = \mk{u}_1 \oplus \mk{u}_2 \oplus \dots \oplus \mk{u}_p$ be the decomposition of the compact Lie algebra $\mk{u}$ into simple ideals. Then for any $i=1, \dots , k$ we may assume that $\mk{k}_u \cap \mk{u}_i $ is a proper subalgebra of $\mk{u}_i$.
\end{lemma}
\bpr If for some $i$ we have $\mk{k}_u \supset \mk{u}_i$, for
simplicity let us take $i=1$, then let us denote by ${\bf U}'$ the
connected subgroup of ${\bf U}$  generated by $ \mk{u}_2 \oplus
\dots \oplus \mk{u}_p$. Since $ \mk{u}_2 \oplus  \dots \oplus
\mk{u}_p$ is a compact semisimple Lie algebra, then ${\bf U}'$ is
a compact subgroup of ${\bf U}$. By $in:{\bf U}' \rightarrow {\bf
U}$ we shall denote the corresponding embedding. Now we consider
the coset space ${\bf U}' /({\bf K}_u \cap {\bf U}')$ and the map
\be \varphi: {\bf U}' /({\bf K}_u \cap {\bf U}') \rightarrow  {\bf
U} /{\bf K}_u \qquad u ({\bf K}_u \cap {\bf U}') \mapsto u {\bf
K}_u.   \ee  Let $\pi': {\bf U}' \rightarrow {\bf U}' /({\bf K}_u
\cap {\bf U}') $ and $\pi: {\bf U} \rightarrow {\bf U} /{\bf K}_u
$ be the corresponding projections. Obviously $\varphi \circ \pi'
= \pi \circ in$ and $\varphi $ is smooth  injection. Furthermore,
$\varphi$ is regular. Indeed, let $X \in  \mk{u}_2 \oplus  \dots
\oplus \mk{u}_p$ be such that $\rd \varphi (\rd \pi'(X)) = 0$,
hence $\rd \pi( \rd in (X) ) = 0$, which implies $X \in \mk{k}_u
\cap (\mk{u}_2 \oplus  \dots \oplus \mk{u}_p)$. Since the Lie
algebra of ${\bf K}_u \cap {\bf U}'$ is $\mk{k}_u \cap (\mk{u}_2
\oplus  \dots \oplus \mk{u}_p)$, then  $\rd \pi'(X) = 0$ and the
regularity of $\varphi$ is proved. On the other hand, since
$\mk{u}_1 \subset \mk{k}_u$, then $\mk{k}_u = \mk{u}_1 \oplus
\mk{k}_u \cap (\mk{u}_2 \oplus  \dots \oplus \mk{u}_p)$, i. e. $
dim({\bf K}_u \cap {\bf U}')= dim(\mk{k}_u) -  dim(\mk{u}_1)$.
Therefore $ dim({\bf U}) -  dim({\bf K}_u)=$ $ dim(\mk{u}) -
dim(\mk{k}_u) =  dim(\mk{u}) -  dim(\mk{u}_1) -  ( dim(\mk{k}_u) -
dim(\mk{u}_1))$ $= dim({\bf U}') -  dim({\bf K}_u \cap {\bf U}')$.
Thus we see that $ dim({\bf U}' /({\bf K}_u \cap {\bf U}')) =
dim({\bf U} /{\bf K}_u)$, hence $\varphi$ is an open map,
therefore it is surjective and therefore  ${\bf U}'$ acts
transitively on $M$ as well as ${\bf U}$. Therefore we may
substitute ${\bf U}'$ for ${\bf U}$. Now the validity of this
lemma is obvious.
 \epr
By the main Theorem \ref{main theorem}  $(\mk{u},\mk{k}_u )$ is a hypercomplex pair (see Definition \ref{shp}) and $M={\bf U}/{\bf K}_u$ is associated with it (see Definition \ref{hypercomplex space associated with a hyp}).
 Furthermore, Lemma \ref{remark we may think mk{u}_i is proper sub of mk{k}_0} shows that $(\mk{u},\mk{k}_u)$ is a semisimple hypercomplex pair and Theorem \ref{prop for HC space as hypercomplex pair} is proved.

Now we can prove:

\begin{coro}  \label{schc3} The HC spaces are exactly the coset spaces ${\bf U}/{\bf K}_u$,  with the LIHCS, given in the following list:
 \begin{gather} SU(n_1+1)/SU(n_1+3-2 k_1)  \times SU(n_2+1)/SU(n_2+3-2 k_2) \times  \dots \nonumber \\ \dots \times SU(n_l+1)/SU(n_l+3-2 k_l) \times SU(m_1+1) \times SU(m_2+1) \times \dots \times SU(m_p+1) \nonumber \end{gather}
where:
\begin{itemize}
    \item $l+p\geq 1, l,p \in \NN $;
    \item $ \forall i \in \{1,\dots,l\} \qquad   n_i \geq 2 , \ k_i\geq 2, \  n_i+3-2 k_i > 0 ;$
    \item $\forall i \in \{1,\dots,p\} \qquad   m_i \in 2 \NN, m_i \geq 2; $
    \item for $i \in \{1,\dots,l\}$    $ SU(n_i+3-2 k_1)$ is embedded in $SU(n_i+1)$ as shown in   \eqref{remark for embedding in mk{su} matrices}.
\end{itemize}
\end{coro}

\begin{proof}  From Theorem \ref{prop for HC space as hypercomplex pair} we see that any $HC$-space $(M, I_M, J_M)$  is of the type $M=\wt{{\bf U}}/ \wt{\bf K}_u$,
where the Lie algebras $\mk{u}$ and $\mk{k}$ of $\wt{\bf U}$ and $\wt{\bf K}_u$ are such that \textbf{$(\mk{u},\mk{k})$ is a semi-simple hypercomplex pair} and $(I_M, J_M)$ is a LIHCS  on $\wt{{\bf U}}/ \wt{\bf K}_u$, determined by hypercomplex structure $(I,J)$ on $\mk{p}_u$
 (see Definition \ref{hypercomplex space associated with a hyp}). Furthermore, the hypercomplex structure $(I,J)$ on $\mk{p}_u$ is obtained by the method in subsection \ref{chcs}.

 From Theorem \ref{schc1} and the definition of $su$-hypercomplex pair \ref{shp1}, we see that there is a coset  ${\bf U}/{\bf K}_u$ in our list, s. t. the  Lie algebras of ${\bf U}$,  ${\bf K}_u$ are $\mk{u} $, $\mk{k}$, respectively. The space ${\bf U}/{\bf K}_u$ together with the LIHCS on it obtained by left translation of  $(I,J)$ on $\mk{p}_u$  throughout ${\bf U}/{\bf K}_u$   give another HC space $({\bf U}/{\bf K}_u,I_{ U/ K}, J_{ U/ K})$ (see the end of subsection \ref{chcs})\footnote{Any coset of the type $SU(n)/SU(m)$ with $m<n$ is simply connected.}.

   Since ${\bf U}$,  ${\bf K}_u$ are simply connected, then one can easily show that the identity map $Id: (\mk{u}, \mk{k}) \rightarrow (\mk{u}, \mk{k})$ extends to a hypercomplex smooth map  $({\bf U}/{\bf K}_u,I_{ U/ K}, J_{ U/ K}) \rightarrow  (\wt{{\bf U}}/ \wt{\bf K}_u,I_M,J_M) $, furthermore this is a covering map. Since ${\bf U}/{\bf K}_u$,  $  \wt{{\bf U}}/ \wt{\bf K}_u$ are both simply-connected, then this map is diffeomorphism.
\end{proof}

\end{document}